\documentclass[opre,nonblindrev]{informs3}

\DoubleSpacedXI % Made default 4/4/2014 at request
\usepackage{endnotes}
\let\footnote=\endnote
\let\enotesize=\normalsize
\def\notesname{Endnotes}%
\def\makeenmark{\hbox to1.275em{\theenmark.\enskip\hss}}
\def\enoteformat{\rightskip0pt\leftskip0pt\parindent=1.275em
  \leavevmode\llap{\makeenmark}}

\usepackage{tikz}
\usetikzlibrary{shapes}
\usetikzlibrary{decorations}
\usetikzlibrary{backgrounds}
\usetikzlibrary{calc}
\usetikzlibrary{shapes}
\usetikzlibrary{arrows}
\usetikzlibrary{automata,positioning}
\usepackage{soul}
\usepackage{enumitem}
\usepackage{bbm}
\usepackage{hyperref}
\newcommand{\ba}{\mathbf{a}}
\newcommand{\bb}{\mathbf{b}}
\newcommand{\bc}{\mathbf{c}}
\newcommand{\bd}{\mathbf{d}}
\renewcommand{\be}{\mathbf{e}}

\newcommand{\bh}{\mathbf{h}}

\newcommand{\bp}{\mathbf{p}}
\newcommand{\bq}{\mathbf{q}}

\newcommand{\bs}{\mathbf{s}}

\newcommand{\bu}{\mathbf{u}}
\newcommand{\bv}{\mathbf{v}}

\newcommand{\bx}{\mathbf{x}}
\newcommand{\by}{\mathbf{y}}
\newcommand{\bz}{\mathbf{z}}
\newcommand{\bzero}{\mathbf{0}}
\newcommand{\bone}{\mathbf{1}}
\newcommand{\bA}{\mathbf{A}}
\newcommand{\bB}{\mathbf{B}}
\newcommand{\bbeta}{\boldsymbol{\beta}}
\newcommand{\btbeta}{\boldsymbol{\tilde{\beta}}}
\newcommand{\tbeta}{\tilde{\beta}}
\newcommand{\bkappa}{\boldsymbol{\kappa}}
\newcommand{\bxi}{\boldsymbol{\xi}}
\newcommand{\tzeta}{\tilde{\zeta}}

\newcommand{\bbR}{\mathbb{R}}

\newcommand{\bbZ}{\mathbb{Z}}

\newcommand{\calS}{\mathcal{S}}
\newcommand{\calB}{\mathcal{B}}
\newcommand{\calC}{\mathcal{C}}
\newcommand{\calP}{\mathcal{P}}

\newcommand{\hy}{\hat{y}}
\newcommand{\bhy}{\mathbf{\hat{y}}}

\newcommand{\talpha}{\tilde{\alpha}}
\newcommand{\tl}{\tilde{\lambda}}
\newcommand{\hlambda}{\hat{\lambda}}
\newcommand{\tnu}{\tilde{\nu}}
\newcommand{\tmu}{\tilde{\mu}}
\newcommand{\tchi}{\tilde{\chi}}
\newcommand{\bchi}{\boldsymbol{\chi}}
\newcommand{\bhchi}{\boldsymbol{\hat{\chi}}}
\newcommand{\hchi}{\hat{\chi}}
\newcommand{\btnu}{\boldsymbol{\tilde{\nu}}}
\newcommand{\barq}{\bar{q}}
\newcommand{\barw}{\bar{w}}
\newcommand{\bara}{\bar{a}}
\newcommand{\barx}{\bar{x}}
\newcommand{\barz}{\bar{z}}
\newcommand{\baru}{\bar{u}}
\newcommand{\bars}{\bar{s}}

\newcommand{\hatmu}{\hat{\mu}}

\newcommand{\bhatmu}{\boldsymbol{\hat{\mu}}}

\newcommand{\bbarq}{\mathbf{\bar{q}}}
\newcommand{\bbarw}{\mathbf{\bar{w}}}
\newcommand{\bbara}{\mathbf{\bar{a}}}
\newcommand{\bbarx}{\mathbf{\bar{x}}}
\newcommand{\bbary}{\mathbf{\bar{y}}}
\newcommand{\bary}{\bar{y}}

\newcommand{\dfn}{\overset{\Delta}{=}}

\newcommand{\blambda}{\boldsymbol{\lambda}}
\newcommand{\bmu}{\boldsymbol{\mu}}
\newcommand{\bnu}{\boldsymbol{\nu}}

\newcommand{\btlambda}{\boldsymbol{\tilde{\lambda}}}
\newcommand{\btmu}{\boldsymbol{\tilde{\mu}}}
\newcommand{\btchi}{\boldsymbol{\tilde{\chi}}}

\newcommand\inner[2]{\left \langle #1, #2 \right \rangle}
\newcommand\E[2]{\mathbb{E}_{#1}\left[#2\right]}
\renewcommand\P[2]{\mathbb{P}_{#1}\left[#2\right]}
\newcommand\Var[1]{\textrm{Var}\left[#1\right]}
\newcommand\Cov[2]{\textrm{Cov}\left[#1,#2\right]}
\newcommand{\stg}{\geq_{st}
}
\newcommand{\stl}{\leq_{st}
}

\allowdisplaybreaks % To allow page breaks inside equation environment
\newcommand{\cost}{c}
\newcommand{\Peq}{\mathcal{P}^{\text{EQ}}}
\newcommand{\StablePol}{\mathcal{E}}

\usepackage{natbib}
 \bibpunct[, ]{(}{)}{,}{a}{}{,}%
 \def\bibfont{\small}%
\TheoremsNumberedThrough     % Preferred (Theorem 1, Lemma 1, Theorem 2)
\ECRepeatTheorems

\EquationsNumberedThrough    % Default: (1), (2), ...
\begin{document}
%%%%%%%%%%%%%%%%

%\setcounter{page}{1}
%\DOI{10.1287/xxxx.0000.0000}%

\RUNAUTHOR{Varma, Castro and Maguluri}

\RUNTITLE{Near Optimal Control in Ride Hailing Platforms with Strategic Servers}

% \TITLE{Bundling Information Goods of Decreasing Value}

\TITLE{Near Optimal Control in Ride Hailing Platforms\\ with Strategic Servers}

\ARTICLEAUTHORS{%
\AUTHOR{Sushil Mahavir Varma}
\AFF{H.Milton Stewart School of Industrial and Systems Engineering, Georgia Institute of Technology 30332 \EMAIL{sushil@gatech.edu}} %, \URL{}}
\AUTHOR{Francisco Castro}
\AFF{Anderson School of Management, University of California Los Angeles 90095 \EMAIL{francisco.castro@anderson.ucla.edu}}
\AUTHOR{Siva Theja Maguluri}
\AFF{H.Milton Stewart School of Industrial and Systems Engineering, Georgia Institute of Technology 30332 \EMAIL{siva.theja@gatech.edu}}
} % end of the block

\ABSTRACT{%
Motivated by applications in online marketplaces such as ride-hailing, we study how strategic servers impact the system performance. We consider a discrete-time process in which, heterogeneous types of customers and servers arrive. Each customer joins their type's queue, while servers might join a different type's queue depending on the prices posted by the system operator and an inconvenience cost.
Then the system operator, constrained by a compatibility graph, decides
the matching.
The objective is to design an optimal control (pricing and matching scheme) to maximize the profit minus the expected waiting times.  
We develop a general framework that enables us to analyze a broad range of strategic behaviors. In particular, we encode servers' behavior in a properly defined \emph{cost function} that can be tailored to various settings. Using this general cost function, we introduce a novel probabilistic fluid problem. 
The probabilistic fluid model provides an upper bound on the achievable net profit. We then study the system under a large market regime in which the arrival rates are scaled by $\eta$ and present a probabilistic two-price policy and a max-weight matching policy which results in a net profit-loss of at most $O(\eta^{1/3})$. 
In addition, under a broad class of customer pricing policies, we show that any matching policy has net profit-loss of at least $\Omega(\eta^{1/3})$. To show generality of our framework, we present multiple extensions to our model and analysis. 
%We explain the asymptotic scaling by establishing a fundamental scale free trade-off between waiting time and profit. 
%In addition, we capture the server behavior more accurately by introducing a more detailed utility function and show optimality. 
We conclude the discussion by presenting numerical simulations comparing different cost models and analyzing performance of the proposed pricing and matching policies.
}%

%\KEYWORDS{deterministic inventory theory; infinite linear programming duality;
%  existence of optimal policies; semi-Markov decision process; cyclic schedule}

%\KEYWORDS{queueing games, max-weight matching, dynamic pricing, two-sided queues, ridehailing}

\maketitle
%%%%%%%%%%%%%%%%%%%%%%%%%%%%%%%%%%%%%%%%%%%%%%%%%%%%%%%%%%%%%%%%%%%%%%

% Samples of sectioning (and labeling) in OPRE
% NOTE: (1) \section and \subsection do NOT end with a period
%       (2) \subsubsection and lower need end punctuation
%       (3) capitalization is as shown (title style).
%
%\section{Introduction.}\label{intro} %%1.
%\subsection{Duality and the Classical EOQ Problem.}\label{class-EOQ} %% 1.1.
%\subsection{Outline.}\label{outline1} %% 1.2.
%\subsubsection{Cyclic Schedules for the General Deterministic SMDP.}
%  \label{cyclic-schedules} %% 1.2.1
%\section{Problem Description.}\label{problemdescription} %% 2.

% Text of your paper here
\section{Introduction}
The rise of the gig economy has brought dynamic pricing and matching to the foreground
of two-sided markets. Ridehailing and meal delivery platforms such as Uber, Lyft, or Doordash adjust these levers so as to maintain a reliable system operation and manage their revenue. Control mechanisms --- dynamic pricing and matching are essential for these markets as they determine not only the demand response but also the behavior of strategic heterogeneous supply agents. A fundamental issue that emerges in this context, however, is the misalignment of supply and demand preferences (or types). At any given time, a supply unit might not be compatible with a specific type of demand request. To mitigate this, in addition to implementing dynamic pricing and matching, some platforms have designed their market to give supply units the option of reporting their type. In ridehailing, for example, drivers have the ability to set a specific destination (such as their home or their children's school)--effectively filtering the trips they are willing to serve--and the platform will match them with riders going in that direction. However, this can lead to undesirable outcomes because supply agents might misreport their types in order to boost their earnings. Indeed, highly profitable destinations such as airports or concert venues are widely preferred by drivers in the ridehailing market. This can negatively impact the performance of such systems by reducing the availability of agents that are willing to serve other trips. The goal of this paper is to provide a framework to analyze and characterize the optimal dynamic pricing and matching decisions in a two-sided market in which supply units strategically report their type to the system operator. 

We consider a general two-sided queueing system with strategic servers and customers arriving stochastically in a discrete time setting. Servers and customers have different types and their compatibility is captured by a bipartite graph. In each time period, the system operator posts a price for each type of customer and server. This leads to a fraction of customers and servers accepting the price and joining the system. Each customer entering the system pays the price posted by the system operator. Meanwhile, servers can misreport their type in order to maximize their utility, which is given by the price of the reported type minus a type-specific inconvenience cost. At the end of a time period, the system operator, constrained by the compatibility graph, decides which server-customer pairs to match,  and the matched pairs depart from the system.
%instantaneously. 

The system operator’s objective is to formulate a pricing and matching policy so that the difference of the long-run profit obtained by the system operator and the long-run cost incurred due to waiting times is maximized.  
\subsection{Main Contributions}
\label{sec: main_contributions}
We develop a general discrete-time, game-theoretical, stochastic framework to study pricing and matching decisions in two-sided markets. In line with our ridehailing motivation, supply agents can strategically misreport their type while demand agents report their type truthfully. The main challenge in this problem is that the selfish behavior of servers leads to correlated arrival rates across different types of servers and customers. In particular, if the price for one type of server changes, the arrival rates of all types of servers are affected which also affects the customer arrival rates. Consequently, a mathematical model that incorporates general kinds of strategic behavior that is also amenable to analysis is not clear. %it is not clear how to model this problem so that one can write a fluid model.
%\siva{Francisco, please check if the previous sentence is fine. We deleted the following sentence:  In turn, we show that the inherent game-theoretical nature of the problem leads to a combinatorial non-linear problem formulation.}

We overcome this challenge by judiciously defining a \emph{cost function} that  corresponds to the total price paid to the servers by the system operator. %This cost function enables us to model a broad class of strategic behaviors of the servers. 
%To tackle this challenge, we synthesize the strategic behavior of servers via a  
We formulate this function as an optimization problem with equilibrium constraints for a broad range of strategic behaviors that servers might exhibit. Indeed, the generality of this formulation enables us to encompass different scenarios: 1) selfish servers maximizing their own utility, 2) a system in which incentive compatibility (truthful) constraints must be satisfied, 3) a partially incentive-compatible system such that at least a fraction of servers are truthful, and 4) a non-game theoretic model in which servers always report truthfully (c.f. \citep{varma2020dynamic}).  We present relations between these models using simulations and theoretical results.

In order to maximize profit in presence of strategic servers, we allow the platform to set randomized server prices. We then present a 
%Under such a general cost function, we present the 
first-order approximation of the system as a \emph{probabilistic fluid problem}. 
This is novel in that it
%Our fluid problem is novel in that it 
allows for probabilistic  pricing policies given by probability measures defined over the set of feasible prices, even in the fluid limit. 
We begin the analysis by providing structural properties for the probabilistic fluid problem. First, we establish that its optimal value provides an upper bound on the profit obtained under any policy.
The strategic behaviour of servers leads to possible non-convexity of the cost function, due to which, %Due to the possible non-convexity of the cost function, 
the standard fluid model with deterministic pricing policy does not give an upper bound on the profit, in general. Thus, extending the space of pricing policies and proposing a probabilistic fluid model is a key contribution. We then present conditions under which the probabilistic fluid problem is equivalent to a standard fluid problem, that only considers deterministic pricing policies. %that is, conditions under which the optimal probability measure is a Dirac measure.
The proposed probabilistic fluid problem is an infinite dimensional optimization problem. We reformulate it as a finite dimensional MINLP which can be solved using standard optimization packages.
We use simulations to compare the optimal fluid objective under the four different cost function models mentioned before, thus presenting a case study on the loss to the platform due to the servers' strategic behavior.

To analyze the stochastic system, for a given policy, we consider a large market regime in which all the arrival rates are scaled by $\eta \rightarrow \infty$. We develop a simple \emph{probabilistic two-price policy} and max-weight matching policy which attains the fluid upper bound asymptotically with $O(\eta^{1/3})$ rate of convergence.  An attractive feature of the matching policy is that it doesn't use the solution of the fluid problem. Instead, it follows state-dependent matching decisions. The pricing policy is a %an $\epsilon$
perturbation of the one prescribed by the fluid problem. It is probabilistic and state-independent for servers, and a dynamic two-price policy for customers. We also show that under a broad class of customer pricing policies, any matching policy will result in $\Omega(\eta^{1/3})$ profit loss. %In addition, under the probabilistic two-price policy, given a desired expected queue length, we present scale-free bounds on the achievable profit. This provides insights on the attainable profit for a firm given a desired service level.

We conclude our discussion by presenting four extensions of our model to exhibit the generality of our framework. (1) We show that an equivalent formulation of our results establishes a fundamental trade off between the profit and queue length, in a scale free manner for any given arrival rate. %\siva{Does this sentence look OK? }
(2) We consider a slightly different model, wherein we penalize the system based on the expected waiting time rather than the expected queue length and show that our proposed policies achieves the optimal scaling of $O(\eta^{-1/3})$. (3) We generalize the utility function to also depend on the steady-state rate of matching for all customer-server pairs. Under this endogenous model, we formulate a probabilistic fluid model and prove that it provides an upper bound on the net-profit obtained under any pricing and matching policy. Following our previous framework, we propose a stochastic policy which is a perturbation of the fluid optimal policy to prove that it attains the fluid upper bound asymptotically with $O(\eta^{1/3})$ rate of convergence. (4) We allow the servers to choose equilibrium of their choice among the ones that maximize their utility. We analyze the worst case scenario by considering adversarial servers. We formulate it as an min-max problem, present a probabilistic fluid model, and show that it provides an upper bound on the achievable net profit.

\subsection{Literature Review}
% \begin{figure}
% \FIGURE{
%  \begin{tikzpicture}[scale=0.6]
%  \draw (0,0) rectangle (6,1); \node[black,very thick] at (3,0.5) {Dynamic Pricing};
%   \draw (0,-1) rectangle (6,-2); \node[black,very thick] at (3,-1.5) {Dynamic Matching};
%     \draw (0,-3) rectangle (6,-4); \node[black,very thick] at (3,-3.5) {Games};
%  \draw (8,-1) rectangle (10,-2); 
%  \node[black,very thick] at (9,-1.5) {for};
%   \draw (12,-3) rectangle (20,-4); \node[black,very thick] at (16,-3.5) {Two Sided Queues};
%   \draw (12,-1) rectangle (20,-2); \node[black,very thick] at (16,-1.5) {Queues};
%   \draw (12,1) rectangle (20,0); \node[black,very thick] at (16,0.5) {Ride Hailing (Applications)};
%   \draw (6,0.5) -- (8,-1.5);
% \draw (6,-1.5) -- (8,-1.5);
% \draw (6,-3.5) -- (8,-1.5);
% \draw (10,-1.5) -- (12,0.5);
% \draw (10,-1.5) -- (12,-1.5);
% \draw (10,-1.5) -- (12,-3.5);
% \end{tikzpicture}}
%     {\centering{Overview: Literature Review.}
%     \label{fig:literature_review}}{}
% \end{figure}
In this paper, we consider dynamic pricing and matching for two sided queues with servers as individual decision makers.
We discuss prior literature on each of these  separately.
\subsubsection{Two-Sided Queues.}
Different variants of two-sided queues were studied in the literature.  \cite{caldentey2009fcfs} pointed out that a two-sided queue is fundamentally unstable. They  analyzed a two-sided queueing model given by a bipartite graph and deduced necessary conditions on the arrival rates for stability. The results were extended by \cite{adan2012exact}. \cite{matchingqueues} considered a more general model of two-sided queues: matching queues, which is a multi-sided queue. They presented a matching policy and proved that it is asymptotically optimal with rate of convergence $O(\eta^{1/2})$ where arrival rates are scaled by $\eta$. \cite{ondemandservers} considered a two-sided queueing model with server arrivals by invitation and also allowed customers and servers to abandon the system. There are numerous applications of two sided queues, such as routing cryptocurrency in payment processing networks \citep{varma2019throughput}, ridehailing systems \citep{banerjee2016dynamic} and \citep{banerjee2018state},  general setting of dynamic matching markets \citep{akbarpour2017thickness}, and dynamic barter exchange \citep{yashkanoriabarter}.
\subsubsection{Dynamic Pricing.}
Dynamic pricing is a fundamental problem in the revenue management literature \citep{talluri2006theory}. In the context of queueing theory, different models have been considered in \citep{low1974optimal}, \citep{pricinglow1974}, \citep{pricinginqueuing2001}, and \citep{Tsitsiklis2000congestion}. The main results in these paper present different structural properties of the optimal pricing policy by studying the underlying control problem. Some of the papers involving dynamic pricing which are closely related to our work are presented in detail below.

\cite{amywardpricingmatching} consider a general dynamic pricing and matching problem. Their goal is to maximize the total number of matches in a finite time. In the same asymptotic regime as ours, they provide an asymptotically optimal policy but do not provide the rate of convergence. The setting of joint optimization of pricing and matching decisions was extended in \citep{ozkan2020joint} to study strategic servers but does not establish rate of convergence to the fluid limit.

\cite{kim2017value} consider the fundamental problem of dynamic pricing in an M/M/1 queue. They consider customers joining the system depending on the offered price and their waiting times. The objective is to maximize the profit of the system operator.  They present an asymptotically optimal pricing policy and also prove that its rate of convergence is $O(\eta^{1/3})$
\subsubsection{Dynamic Matching.}
Dynamic Matching is a fundamental problem in two-sided queues with heterogeneous customer and server arrivals. A FCFS matching discipline was studied by \cite{caldentey2009fcfs} and \cite{adan2012exact}. In a related context, \cite{matchingqueues} considered a multi-sided matching queue and provided an asymptotically optimal matching policy. Delayed matching (batching) in the hope that better matching opportunities will arrive in the future was analyzed by \cite{yashkanoriabarter} and \cite{akbarpour2017thickness}. In both these papers, they concluded that delayed matching does not provide significant benefits. Some of the papers involving dynamic matching that are close to our work are presented in detail below.

 \cite{dynamictypematchinghu}  consider a two sided market given by a bipartite graph with  associated penalties depending on the type of demand and supply matched. Their objective is to find a matching policy which maximizes the discounted reward in finite time. They present multiple structural properties of the optimal matching policy and also present an asymptotically optimal matching policy.

Our paper is an extension of the work by \cite{varma2020dynamic}, where a similar model was considered. There are two key differences. They did not consider the strategic behavior of  servers. This addition to the model results in fundamentally different problem and optimal policy. Moreover, in contrast to the simplistic Poisson arrival model stated in \cite{varma2020dynamic}, this paper consider a more realistic general arrival process in discrete time. This generalization results in technical difficulties and more involved proofs.
\subsubsection{Queueing Games.}
The book by
\cite{rationalqueueing} provides a comprehensive overview on game theory applied to queueing systems. In the present paper, with a large number of servers arriving in the system, we are dealing with non atomic games. \cite{nonatomicgames} deal with non atomic games and show that equilibrium constraints can be equivalently written as a fixed point equation. \cite{mpec} provides a comprehensive theory of solving and reformulation of the optimization problem with equilibrium constraint which is known to be NP-hard. A comprehensive background on algorithmic game theory can be found in \cite{algorithmicgametheory}.

We combine all these aspects that has been studied in the literature. In particular, we combine dynamic pricing, dynamic matching in a strategic setting and carry out fluid as well as stochastic analysis. Allowing probabilistic policies is a novel approach of formulating the fluid model. These leads to a probabilistic optimal pricing policy and it seems to be fundamental to the systems with strategic behavior and is novel in the literature.  

\subsection{Notation}
We denote the set of  real numbers, the set of non negative real numbers, the set of integers and the set of non negative integers by $\bbR$, $\bbR_+$, $\bbZ$ and $\bbZ_+$, respectively. In addition, we denote the extended real line $\bbR \cup \{\infty,-\infty\}$ by $\bar{\bbR}$. We denote the set of natural numbers from 1 to $n$ by $[n]$. In the entire paper, we use $1$, $i$ and $n$ for parameters concerning servers and $2$, $j$ and $m$ for customers. We refer to servers by she/her/her and to customers by he/his/him. In the entire paper, vectors are boldfaced. We denote a vector of zeros of dimension $n$ by $\bzero_n$ and a vector of ones of size $m$ by $\bone_m$. We omit the subscript if the dimension of the vector is clear from the context.  For two vectors $\bx \in \bbR^n$ and $\by \in \bbR^m$, we denote the concatenated vector $\bz \in \bbR^{n+m}$ by $\bz=(\bx,\by)$. The dot product of two vectors is denoted by $\inner{.}{.}$. An $n$ dimensional probability simplex is denoted by $\Delta_n \subseteq \bbR_+^n$. In particular, if $\bnu \in \Delta_n$, then $\inner{\bone_n}{\bnu}=1$. In addition, a collection of $m$ such probability simplex-es is denoted by $\Delta_n^m \subseteq \bbR_+^{m \times n}$. In particular, if $\bnu' \in \Delta_n^m$, then $\sum_{l=1}^n \nu_{il}=1$ for all $i \in [m]$. For functions $F_j: \bbR \rightarrow \bbR$ with $j \in [m]$ and a vector $\blambda \in \bbR^m$, we write $F(\blambda)$  to denote $(F_1(\lambda_1),\hdots, F_m(\lambda_m))$. For two matrices $\bA$ and $\bB$ of size $m \times n$, the sum of the entries of their Hadamard product is denoted by $\bA \circ \bB$, that is, $\bA \circ \bB=\sum_{i=1}^m \sum_{j=1}^n A_{ij}B_{ij}$. The inequality $\bA \leq \bB$ is component wise, i.e. $A_{ij} \leq B_{ij}$ for all $i\in [m], j \in [n]$. Variance of a random variable is denoted by $\Var{.}$ and co-variance is denoted by $\Cov{.}{.}$. For two random variables $X$ and $Y$, if $\P{}{X \leq a} \leq \P{}{Y \leq a}$ for all $a \in \bbR$, then we say that $X$ stochastically dominates $Y$ and denote it by $X \stg Y$. Quantities pertaining to the fluid model are denoted with a `tilde' on top and  quantities pertaining to the steady state of the stochastic model are denoted with a `bar' on top.

Let $\alpha$ be a probability measure defined on the Borel sigma-algebra generated by $\Omega \subseteq \bbR^d$ for some $d \in \bbZ_+$. Let $f : \Omega \rightarrow \bbR$ be a Borel measurable function and $\omega \in \Omega$ be an element of the probability space. Then, expectation of $f$ with respect to $\alpha$ is denoted by $\E{\alpha}{f(\omega)}=\int_{\Omega} f(\omega) d\alpha(\omega)$. This can also be interpreted as $f(\omega)$ is a random variable such that $\P{}{f(\omega) \in B}=\alpha(B)$ for all Borel subsets $B \subseteq \Omega$. For a Markov chain $\{\bq(k): k \in \bbZ_+\}$ with state space $S$, we denote the expectation of $f : S \rightarrow \bbR$, with respect to its stationary distribution by $\E{\bbarq}{f(\bbarq)}$ and sometimes omit the subscript when it is clear from the context.

\section{Model} \label{sec: model}
We consider a general two-sided, discrete time queueing system modeled as 
a bipartite graph $G(N_1 \cup N_2, E)$.
We refer to 
$G(N_1 \cup N_2, E)$ as the compatibility graph, where $N_1=[n]$ is the set of server types, $N_2=[m]$ is the set of customer types, and $E$ is the set of compatible edges that represent the feasible matches between  customers and servers.
%edges which can be used to match a customer-driver pair and we call it the compatibility graph. 
Each node in the graph denotes a queue of a customer/server waiting to be matched. We consider two control mechanisms --- dynamic pricing and matching.
%We consider a discrete time model, where in each time slot, the system operator carry out the following two tasks:
%We consider a discrete time model 
In each time slot,
first, the system operator determines prices for each customer and server queue. Then,
agents arrive to the system and, given the prices, make joining decisions. Customers always join their type's queue, while servers can strategically choose which queue to join. After this, the system operator possibly matches the compatible pairs of customers and servers waiting in the system. 
%1) Determines the price for each type of customer and server which governs the arrival process. 2) Matches (possibly) the compatible pairs of customers and drivers waiting in the system. 
%Now, we will define a discrete time queueing system and expound the arrival and service process below:
Next, we present each component of the model in detail.
%will define a discrete time queueing system and expound the arrival and service process below:
%\subsubsection*{State:}
We denote the state of the system at time $k$ by $\{\bq(k) \in \bbZ_+^{n+m}: k \in \bbZ_+\}$, where the vector is defined as $\bq=(q_1^{(1)},\hdots,q_n^{(1)},q_1^{(2)},\hdots,q_m^{(2)})$ and $q_i^{(1)}(k)$
is the number of servers in the $i$ type queue waiting in the system at time $k$, and $q_j^{(2)}(k)$ is the number of $j$ type customers waiting in the system at time $k$. The state space of the system is denoted by $\calS \subseteq \bbZ_+^{(n+m)}$. Note that $\calS$ depends on the pricing and matching policy set by the system operator which are defined later in this section. 

\textbf{Customers} arrive in the system with an exogenous arrival rate. In each time epoch $k \in \bbZ_+$, the system operator posts a price $\bp^{(2)}(k) \in \bbR_+^m$ which leads to only a fraction of arriving customers to join the system. This results in an effective customer arrival with rate denoted by $\blambda(k) \in \bbR_+^m$. We capture this relation between the posted price and the effective arrival rate by inverse supply curves $F_j: \bbR_+ \rightarrow \bbR_+$ for all $j \in [m]$. In particular, we have $F_j(\lambda_j(k))=p_j^{(2)}(k)$. We allow the arrivals to be correlated across types but they are independent across time. This models a broad range of situations. For example, in the context of ride hailing, at certain times of the day, the number of customers going in certain directions can be correlated. We denote the co-variance matrix of the arrivals in $k^{th}$ time epoch by $\Sigma^{(2)}(k) \in \bbR^{m \times m}$ which depends on the posted price $\bp^{(2)}(k)$. In particular, we denote the effective arrival of customers of type $j$
 by a sequence of independent random variables $\{a^{(2)}_j(k) : k \in \bbZ_+\}$ for all $j \in [m]$ with mean $\E{}{a^{(2)}_j(k)}=\lambda_j(k)$ and co-variance matrix $\Var{\ba(k)}=\Sigma^{(2)}(k)$. Specifically, $a^{(2)}_j(k)$ is the number of customers of type $j$ that arrive to the system at time $k$. We assume $|a_j^{(2)}(k)| \leq A_{\max}$ with probability 1 for all $j \in [m]$ and for all $k \in \bbZ_+$. As $|a_j^{(2)}(k)| \leq A_{\max}$, there exists $\Sigma_{\max}^{(2)}, \Sigma_{\min}^{(2)} \in \bbR^{m \times m}$ such that $\Sigma_{\min}^{(2)} \leq \Sigma^{(2)}(k) \leq \Sigma_{\max}^{(2)}$. We make the following standard assumptions on the inverse demand curve.
\begin{assumption} \label{ass: monotonic}
The inverse demand curve $F_j(\cdot)$ is strictly decreasing and twice continuously differentiable for all $j \in [m]$.
\end{assumption}
In words, %we make the assumption that 
if the posted price for customers is higher, then less customers would be willing to avail that service and vice versa. In addition, we assume that the inverse demand curve is twice continuously differentiable which is a technical assumption required for our analysis. 
\begin{assumption} \label{ass: concave}
The function $\lambda_j F_j(\lambda_j)$ is  concave  for all $j \in [m]$.
\end{assumption}
By the law of diminishing marginal utility, if the arrival rate increases, then the marginal utility derived from each new customer $(\frac{d}{d\lambda_j} \lambda_j F(\lambda_j))$ decreases. This condition is equivalent to requiring that 
the demand curve comes from a regular distribution -- a typical assumption in economics.

%\fc{---a typical assumption in economics. }

\textbf{Servers.} Now, we define the arrival process of servers. We identify servers as decision makers that make strategic joining decisions. A type $i$ server arriving to the system can join the $l$ type queue for some $l \in [m]$ or leave the system depending on its personal utility $u_{il}$ which, in turn, depends on the price set by the system operator $\bp^{(1)} \in \bbR_+^n$ and the detour penalties $\bc \in \bbR^{n \times n}$. In particular,
a server of type $i$ who joins the $l$ type queue earns $u_{il}=f_{il}(\bp^{(1)})$ for some function $f_{il}: \bbR_+^{n} \rightarrow \bbR$. One typical example of utility function which we will use later for simulations is given by
\begin{align}
    u_{il}=p^{(1)}_l-c_{il} \quad \forall i \in [n], \ \forall l \in [n], \label{eq: utility}
\end{align}
where
$p_l^{(1)}$ is the price set by the system operator for servers that join the $l$ type
queue, and $c_{il}$ is the penalty due to lying incurred by a server of type $i$ when she joins the  $l$ type server queue.
%This penalty captures the cost incurred by the server \delfc{commute  \delfc{or time} due to lying.
In our ridehailing application, $c_{il}$ represents a detour cost experienced  by a driver when she is assigned a non compatible trip. An illustration of this as a tripartite graph is given in Fig. \ref{fig: multiple_link}. %\fc{Explain why we are introducing $f$, why do we want to be general.}
A type $i$ driver entering the system will compare her utilities $u_{il} \ \forall l \in [n]$ with her outside option. If the maximum possible utility $u_i\triangleq \max_{l \in [n]} u_{il}$ is greater than 
her outside option, then the driver will join the queue which maximizes her utility. Otherwise, the driver will not join the system at all. For the real life system, the utility may depend on several other factors and may not vary with the posted prices linearly. Thus, we consider a general utility function. Note that all our results holds for any continuous utility function given by $\mathbf{f}(\cdot)$.
\begin{figure}[hbt]
\FIGURE{
 \begin{tikzpicture}[scale=0.6]
\draw[black, very thick] (0,0) -- (2,0) -- (2,1) -- (0,1);
\node[black,very thick] at (0.25,0.5) {2};
\draw[black, very thick] (0,1.5) -- (2,1.5) -- (2,2.5) -- (0,2.5);
\node[black,very thick] at (0.25,2) {1};
\draw[black,very thick] (8,0) -- (6,0) -- (6,1) -- (8,1);
\node[black,very thick] at (7.75,0.5) {2};
\draw[black, very thick] (8,1.5) -- (6,1.5) -- (6,2.5) -- (8,2.5);
\node[black,very thick] at (7.75,2) {1};
\draw[black,very thick] (13,0) -- (11,0) -- (11,1) -- (13,1);
\node[black,very thick] at (12.75,0.5) {2};
\draw[black, very thick] (13,1.5) -- (11,1.5) -- (11,2.5) -- (13,2.5);
\node[black,very thick] at (12.75,2) {1};
\draw[black,thick]  (2.75, 2.1) edge[<->]  (5.25, 2.1);
\draw[black,thick]  (2.75, 0.6) edge[<->]  (5.25, 1.9);
\draw[black,thick]  (2.75, 0.4) edge[<->]  (5.25, 0.4);
\draw[black,thick]  (8.25, 0.4) edge[<->]  (10.75, 0.4);
\draw[black,thick]  (8.25, 0.6) edge[<->]  (10.75, 1.9);
\draw[black,thick]  (8.25, 2.1) edge[<->]  (10.75, 2.1);
\draw[black,thick]  (8.25, 1.9) edge[<->]  (10.75, 0.6);
\node[black, align=center] at (9.5,2.4){\footnotesize$c_{11}=0$};
\node[black, align=center] at (9.5,0.2){\footnotesize$c_{22}=0$};
\node[black, align=center] at (9.85,1.7){\footnotesize$c_{12}$};
\node[black, align=center] at (9.95,0.8){\footnotesize$c_{21}$};
\node[black, align=center] at (1,3.5) {\footnotesize Customer};
\node[black, align=center] at (4,3.5) {\footnotesize \shortstack{Compatible\\Matchings}};
\node[black, align=center] at (7,3.5) {\footnotesize \shortstack {Server \\  Queue}};
\node[black, align=center] at (12,3.5) {\footnotesize \shortstack {Server \\ Type}};
\node[black, align=center] at (9.5,3.5) {\footnotesize \shortstack {Complete \\ Graph}};
\fill[black] (7,-0.3) circle (0.05);
\fill[black] (7,-0.5) circle (0.05);
\fill[black] (7,-0.7) circle (0.05);
\draw[black, very thick] (8,-1) -- (6,-1) -- (6,-2) -- (8,-2);
\node[black,very thick] at (7.75,-1.5) {n};
\fill[black] (1,-0.3) circle (0.05);
\fill[black] (1,-0.5) circle (0.05);
\fill[black] (1,-0.7) circle (0.05);
\draw[black, very thick] (0,-1) -- (2,-1) -- (2,-2) -- (0,-2);
\node[black,very thick] at (0.25,-1.5) {m};
\fill[black] (12,-0.3) circle (0.05);
\fill[black] (12,-0.5) circle (0.05);
\fill[black] (12,-0.7) circle (0.05);
\draw[black, very thick] (13,-1) -- (11,-1) -- (11,-2) -- (13,-2);
\node[black,very thick] at (12.75,-1.5) {n};
\draw[black,thick]  (2.75, -1.5) edge[<->]  (5.25, -1.5);
\draw[black,thick]  (2.75, 1.9) edge[<->]  (5.25, -1.4);
\fill[black] (4,-0.3) circle (0.05);
\fill[black] (4,-0.5) circle (0.05);
\fill[black] (4,-0.7) circle (0.05);
\fill[black] (9.5,-0.3) circle (0.05);
\fill[black] (9.5,-0.5) circle (0.05);
\fill[black] (9.5,-0.7) circle (0.05);
\draw[black,thick]  (8.25, 1.75) edge[<->]  (10.75, -1.4);
\draw[black,thick]  (8.25, -1.4) edge[<->]  (10.75, -1.4);
\node[black, align=center] at (9.5,-1.1){\footnotesize$c_{nn}=0$};
\node[black, align=center] at (10.45,-0.5){\footnotesize$c_{n1}$};
\end{tikzpicture}}{
\centering{A tripartite graph for two-sided queues with strategic servers.}
    \label{fig: multiple_link}}{}
\end{figure}

A server of type $i$ who joins the system at time $k$ uses the strategy  $\boldsymbol{\nu}_i(k)\in \Delta_n$, where, for each $l\in [n]$, $\nu_{il}(k)$ is the probability with which a type $i$ server joins queue $l$.
We say that the strategy profile 
$\boldsymbol{\nu}\triangleq(\boldsymbol{\nu}_1(k),\dots, \boldsymbol{\nu}_n(k)) \in \bbR_+^{n\times n}$
is an \emph{equilibrium} if and only if 
\begin{align}\tag{EQ}\label{eq: variational_inequality}
\nu_{il}(k)&>0 \Rightarrow u_{il}(k) \geq u_{il'}(k) \:\: \forall i,l' \in [n].
\end{align}
The above constraint implies that $\bnu_i$ for all $i \in [n]$ has a positive support only for the queues where the servers' utility is the largest. Given the server joins the system, the equilibrium constraints \eqref{eq: variational_inequality} determines which queue it will join. To consider the case when servers do not join the system, we will introduce continuous inverse supply curves denoted by $G_i : \bbR_+ \rightarrow \bbR$ for all $i \in [n]$. This is defined similar to the inverse demand curve with one crucial distinction. The system operator sets a price vector $\bp^{(1)}(k)$ which will determine the maximum possible utilities $\{u_i(k)\}_{i \in [n]}$. Then, we have $G_i(\hatmu_i(k))=u_i(k)$, where $\hatmu_{i}(k)$ represents the arrival rate of type $i$ servers. Moreover, the effective arrival rate of servers to queue $i$ is given by $\mu_i(k) \triangleq \sum_{l=1}^n \hatmu_{l}(k)\nu_{li}(k)$. Observe that because a given queue may receive servers of different type, the arrival processes to different queues can be correlated. Denote the co-variance matrix of the server arrivals by $\Sigma^{(1)}(k)$ which depends on the posted price $\bp^{(1)}(k)$. 

We define the effective arrival process of servers to queue $i$ as a sequence of independent random variables $\{a^{(1)}_i(k) : k \in \bbZ_+\}$ with mean $\E{}{a^{(1)}_i(k)}=\mu_i(k)$ for all $i \in [n]$ and $\Var{\ba^{(1)}(k)}=\Sigma^{(1)}(k)$. We assume that $|a_i^{(1)}(k)| \leq A_{\max}$ with probability 1 for all $i \in [n]$ and for all $k \in \bbZ_+$. As $|a_j^{(2)}(k)| \leq A_{\max}$, there exists $\Sigma_{\max}^{(1)}, \Sigma_{\min}^{(1)} \in \bbR_+^{m \times m}$ such that $\Sigma_{\min}^{(1)} \leq \Sigma^{(1)}(k) \leq \Sigma_{\max}^{(1)}$.

\textbf{Policies.}
%We now specify the decisions made by the system operator. 
The system operator uses a stationary policy and makes both pricing and matching decisions. We describe the pricing policy first. Given the state of the system $(\bq \in \calS)$, a stationary pricing policy is
a vector $(\bp^{(1)}(\bq),\bp^{(2)}(\bq))\in \bbR_{+}^n\times\bbR_{+}^m$ where $p^{(1)}_i(\bq)$ is the payment to servers in queue $i\in [n]$, and $p^{(2)}_j(\bq)$ is the price charged to customers in queue $j\in[m]$.
In order to simplify the analysis, we work in a general space of feasible rates instead of prices. For any stationary rates, we identify a corresponding stationary pricing policy, hence, with some abuse of language, we will refer to the stationary rates as  stationary pricing policies. Note that, we are only interested in stationary, Markovian, state dependent pricing policies and thus, we omit the dependence of $\blambda$ and $\bmu$ on the time index $k$.

On the customer side,  there is a bijection between prices and the arrival rate of customers to the customer queues. Therefore, for ease of exposition, we consider the arrival rates to be our controls, instead of the prices, and so, we define a \emph{customer stationary pricing policy} by the arrival rate vector $\blambda(\bq)\in \bbR^m_+$. In addition, as the co-variance matrix only depends on the price posted by the system operator, we can re-write it as $\Sigma^{(2)}(\blambda(\bq))$.

%\siva{Next two paras need to be rewitten. Separate the cost function thing and the randomized poly thing. First say that on the customer side, there is  a simple bijection between the prices and the arrival rates, and so we work directly with the arrival rates as control levers. That cannot be done here because for a given arrival vectors into server queues, it is not even clear if there is a price vector that leads to it. Moreover, there could be multiple price vectors that lead to it. To over come the first, we define the $\Omega$ set. To address the latter, among all the price vectors that lead to a given arrival rate, we assume that we use the prices that minimize the cost of the platform, u.e., this is an optimistic model... In Sec... we relax this assumption.}

On the server side, such a bijection may not exist. Firstly, it is not clear if there exists a price vector $\bp^{(1)}$ that results in a given arrival rate of servers to the server queues $\bmu$. If there exists such a price vector, it may not be unique, and %we may not have a one-to-one correspondence --- for a given rate $\bmu$ there may not exist or there could be 
many price vector(s) can lead to the arrival rate $\bmu$.
%that are consistent with $\bmu$.
To address these difficulties, we define the set of prices that are consistent with $\bmu$ by
  \begin{align*}
     \mathcal{M}(\bmu)\triangleq\left\{
     \bp^{(1)}\in \bbR_+^n:\exists \boldsymbol{\nu}
     \in \bbR_+^{n\times n}\:\: \text{satisfying} \:\: \eqref{eq: variational_inequality},\: 
     G_i(\hatmu_i) = u_i,\: 
     \mu_i = \sum_{l=1}^n \hatmu_{l}\nu_{li}\:\: \forall i\in[n]
     \right\}. 
 \end{align*}
The set $\mathcal{M}(\bmu)$
is composed by those prices for which there exists an equilibrium 
that leads to the arrival rates $\bmu$ in the servers queues. We define $\Omega$ to be the
set of rates $\bmu$ such that $\mathcal{M}(\bmu)\neq  \emptyset$ which addresses the first difficulty. The corresponding price $\bp^{(1)}$ is then selected among the consistent prices so that it minimizes the cost $\inner{\bmu}{\bp^{(1)}}$ for the service provider which addresses the second difficulty. The optimal cost function $\cost:\Omega\rightarrow \bbR_+$ is defined by (We show that the cost function defined below is Borel measurable in Appendix \ref{app: borel}.)
\begin{align}
    \cost(\bmu)\triangleq\min \inner{\bmu}{\bp^{(1)}} \quad \textit{subject to} \quad  \bp^{(1)} \in\mathcal{M}(\bmu) .  \label{eq: cost_function}
\end{align}
In this paper, we are interested in different types of equilibrium behavior depicting different objectives, for example, incentive compatible system wherein the servers are incentivised appropriately such that they join their own queue. Thus, to keep the model general enough, we will work with a general cost function $\cost(\bmu)$ throughout the paper and discuss the behavior of each different equilibria or objective by specializing the definition of
$\cost(\bmu)$ in later sections.

Note that, we are implicitly assuming that if there are multiple possible equilibrium for the servers to choose from, they will choose the one that leads to the lowest cost for the system operator. This is often referred as the optimistic model and is often considered in the literature (see: \cite{optimistic_highway_pricing} \cite{optimistic_bilevel} and the references therein). It is interpreted as the system operator nudging the servers to choose the equilibrium which is best for the system performance. We relax this condition in Section \ref{sec: pess_nash_eq}.

%\siva{In the next paragraph, again use the words from main contribution, talk about need for allowing for randomized pricing scheme for servers. So, we want to work with randomized policies. Then, say that instead of working with measure on $\Omega$, we work with $\Omega \times \calS$ for technical reasons }

 Similar to the customers, we will operate in the space of arrival rates as opposed to the space of prices as it is more amenable to analysis. In order to maximize profit in presence of strategic servers, we allow the platform to set randomized server prices. We define a \emph{server stationary pricing policy} as a set of probability measures $\alpha_\bq(\cdot)$ for all $\bq \in \calS$ such that $\alpha_\bq$ determines the randomized arrival rate set by the system operator when the state of the system is $\bq$. For technical reasons, we define $\alpha_\bq$ as a measure on the Borel sigma-algebra generated by $\Omega \times \calS$ for all $\bq \in \calS$ such that it is non zero only on $\Omega \times\{\bq\}$, i.e. $\alpha_\bq(\Omega \times \calS \backslash \Omega \times \{\bq\})=0$. 
%Thus, it acts as a measure over $\mu \in \Omega$ for a given state $\bq \in \calS$. 
We use $\Peq_\bq$ to denote the set of such probability measures. For the simplicity of notation, we denote $(\alpha_\bq(\cdot))_{\bq \in \calS}$ by $\alpha(\cdot)$. Note that the co-variance matrix of servers will only depend on the price posted by the system operator or equivalently, the probability measure $\alpha(\cdot)$. Thus, we denote it by $\Sigma^{(1)}(\alpha(\cdot))$. 

We make two important remarks about the server side policy. First, the reason we allow randomized policies for servers is because they result in  a richer class of pricing policies; also, they enable us to tackle the inherent combinatorial structure and non convexity of the service provider's objective. The latter, materializes through the cost function $\cost(\cdot)$ and the strategic
behavior of servers. Second, by allowing randomized policies for servers may result in a higher overall profit for the system operator. We will later show in Section \ref{sec: prob_fluid_model} that allowing randomized policies for the customers will not result in an increased overall profit. %Note that allowing randomized policies for the customers will not result in an increased overall profit due to the concavity of the revenue function $(F_j(\lambda_j)\lambda_j)$ by the Assumption \ref{ass: concave}.  \siva{Replace the last sentence with: `We will later show in Section ?? that allowing randomized policies for the customers will not result in an increased overall profit'. Put the concavity comment later where it is understandable. At this point, it is not clear why concavity has anything to do with this. }

Now we specify the matching policy. We denote by $\{\by(k) \in \bbZ_+^{n \times m}: k \in \bbZ_+\}$  the decision of matching customer-servers pairs at time $k$. Here, $y_{ij}(k)$ is the number of servers in the $i$ type queue that are matched with 
customers of type $j$ at time $k$. For the matching decisions to be feasible, the following conditions must be satisfied 
\begin{subequations} 
\begin{align} 
    x_i^{(1)}(k)&\dfn\sum_{j=1}^m y_{ij}(k) \leq q_i^{(1)}(k)+a_i^{(1)}(k) \ \forall i \in [n], \label{eq: column_sums}  \\
    x_j^{(2)}(k)&\dfn\sum_{i=1}^n y_{ij}(k) \leq q_j^{(2)}(k)+a_j^{(2)}(k) \ \forall j \in [m], \label{eq: row_sums} \\
    y_{ij}(k)&=0 \ \forall (i,j) \notin E, \quad y_{ij}(k) \geq 0 \ \forall (i,j) \in E,
\end{align}
\label{eq: matching_constraints}
\end{subequations}
where $\bx(k)=(\bx^{(1)}(k),\bx^{(2)}(k))$ for all $k \in \bbZ_+$ denote the total amount of servers and customers matched in each queue at a given time period. 
The set of constraints \eqref{eq: matching_constraints} ensures
that the number of servers in the type $i$ queue that are matched cannot be larger than the total number of servers in that queue plus the arrivals. Similarly, the  number of customers in the type $j$ queue that are matched are at most  equal to the total number of customers in that queue plus the arrivals.
Moreover, the only matches allowed are those given by the compatibility graph $G(N_1 \cup N_2, E)$. In turn,
given the state of the system after  arrivals ($\bq+\ba$),  a \emph{stationary  matching policy} is defined as the decision of choosing $\by(\bq+\ba)$ or, equivalently, $\bx(\bq+\ba)$ subject to \eqref{eq: matching_constraints}.

In sum, a policy is a triplet $(\blambda(\bq),\alpha_\bq(\cdot),\bx(\bq))_{\bq \in \calS}$ where $\blambda(\cdot)$ denotes the state dependent customer arrival rates, $\alpha_\bq$ denotes the state dependent probabilistic server arrival rates and $\bx(\cdot)$ are the matching decisions.
%\fc{Maybe summarize with what constitutes a policy: In summary, a policy is a triplet $(\blambda,\alpha,\bx)$ where...}

\textbf{System dynamic.} Given the pricing and matching policy, the system evolves as a discrete time Markov chain. 
%Thus, the time between two epochs is a given finite constant.
%In practice, the pricing and matching decisions are made typically every $\sim 15$ seconds. 
The queue evolution equation is given by:
\begin{align*}
    \bq(k+1)=\bq(k)+\ba(k)-\bx(k),
\end{align*}
where $\bx$ satisfies \eqref{eq: matching_constraints}. 
%\delfc{We define the system is stable as follows:} 
We consider policies that render the system stable.
\begin{definition}[Stability] \label{defn: stability}
The discrete time Markov chain is stable if under a given pricing and matching policy, the communicating class containing the state $\bzero_{n+m}$ is positive recurrent and all the other states (if any) are transient.
\end{definition} We use $\StablePol$ to denote the set of stationary Markovian pricing and matching  policies that make the system stable. For a stable system, we denote the steady state parameters with a bar on top. In particular, $\bq(k)$ converges in distribution as $k \rightarrow \infty$ to a random vector denoted by $\bbarq$. The arrival rate vector given the queue length $\bbarq$, is denoted by $\bbara$ such that $\E{}{\bbara}=(\blambda(\bbarq),\E{\alpha_\bbarq}{\bmu})$ and
the co-variance matrix of $\bbara^{(2)}$ is $\Sigma^{(2)}(\bbarq)$ and that of $\bbara^{(1)}$ is $\Sigma^{(1)}(\bbarq)$. The matching decision given the queue length ($\bbarq$) and arrival rate vector ($\bbara$) is denoted by $\bbarx/\bbary$.

\textbf{Objective.} 
Each customer entering the system pays the posted price and each server receives the posted price. In addition, the system operator incurs a type specific penalty $\bs\in \bbR_+^{m+n}$ due to the waiting of customers and servers. The objective of the system operator is to design the pricing and matching policies such that the difference of average profit  obtained and the average penalty incurred due to waiting---the net average profit---is maximized.
Mathematically,
\begin{subequations}
\begin{align}
    R^\star\triangleq \sup_{(\blambda(\cdot),\alpha(\cdot),\bx(\cdot))\in \StablePol} \E{\bbarq}{\inner{F(\blambda(\bbarq))}{\blambda(\bbarq)}-\E{\alpha_\bbarq}{\cost(\bmu)}-\inner{\bs}{\bbarq}} \span \label{eq: net_profit_stochastic}\\
    \textit{subject to,} \quad &\blambda(\bq)\in \bbR^m_+,\quad \forall \bq\in \calS\\
    &\alpha_\bq(\cdot)\in \Peq_\bq \quad \forall \bq \in \calS \\
    &\bx(\cdot)\quad \textit{satisfies} \:\: \eqref{eq: matching_constraints}  
\end{align}
\label{eq:opt_stoch}
\end{subequations}
%The constraint \eqref{eq: consistency_alpha_nu_model} is a consistency condition making sure that the marginal distribution of $\alpha(\cdot)$ coincides with the stationary distribution of the system. 
We will use $\pi$ to refer to a policy $(\blambda,\alpha,\bx)$ and denote by $R(\pi)$ the expected net profit associated to that policy. Moreover, 
$P^\star$ and $P(\pi)$ will denote the the optimal profit (when $\bs=\bzero_{n+m}$) and the profit evaluated at $\pi$, respectively.

For a given pricing and matching policy, computing the objective function of the above optimization problem is itself challenging as the state space of the DTMC can be very large. In addition, the optimization problem becomes a non-convex, integer optimization problem due to the equilibrium constraints captured in $\Peq$.

Moreover, if the price of one type of the server is changed, it will lead to a change of arrival rates of all  types of servers. Due to this, the system operator will be required to adjust the customer prices as well to compensate for the server arrival rates. This dependence of server arrival rates and its influence on the customer arrival rates makes the analysis of the pricing policy non trivial.

To tackle these challenging problem, we start by introducing a novel probabilistic fluid model. Intuitively, we ignore the stochasticity of the system to make the optimization problem tractable and, in addition, we relax the stability constraint. In  further sections, 
based on the solution to the fluid model, we propose ``near-optimal'' pricing and matching policies for the stochastic system respecting the stability constraints.

\section{A Probabilistic Fluid Model} \label{sec: prob_fluid_model}
In this section, we introduce a fluid counterpart of the optimization problem \eqref{eq:opt_stoch} and study some of its structural properties.
A novel feature in our fluid optimization problem is that in order to determine the pricing policy, we must optimize 
over the space of probability measures.
%system as an optimization problem over the space of measures. 

Firstly, we present the \emph{probabilistic fluid optimization problem} (We show that there exists an optimal solution to the fluid optimization problem in Appendix \ref{app: existence}).
\begin{subequations}
\begin{align}
    \tilde{R}^\star\triangleq \lefteqn{\max_{\btlambda,\talpha,\btchi} \inner{F(\btlambda)}{\btlambda}-\E{\talpha}{\cost(\btmu)}}\\
 \textit{subject to} \quad \tl_j&=\sum_{i=1}^n \tchi_{ij} \quad \forall j \in [m] \label{eq: cust_rate_balance} \\
    \E{\talpha}{\tmu_i}&=\sum_{j=1}^m \tchi_{ij} \quad \forall i \in [n] \label{eq: serv_rate_balance} \\
    \tchi_{ij}&=0 \quad \forall (i,j) \notin E, \quad \tchi_{ij} \geq 0 \quad \forall (i,j) \in E, \label{eq: compatibility}
\end{align}
\label{eq: prob_fluid_model}
\end{subequations}
where $\cost(\cdot)$ is the cost function given in \eqref{eq: cost_function}, $\btlambda \in \bbR_+^m$ is the `average' flow of customers in the system. Next, $\talpha$ is the `average' probability measure defined on the Borel sigma algebra of $\Omega$ governing the probabilistic server pricing policy which can be interpreted as a distribution over the feasible rates $\btmu$. Lastly, $\tchi_{ij} \in \bbR_+^{n \times m}$ is the `average' flow of $i$ type of servers matched to $j$ type of customer. The objective function is the profit obtained by the system operator. Equations
 \eqref{eq: cust_rate_balance}, \eqref{eq: serv_rate_balance} are flow conserving constraints and \eqref{eq: compatibility} is the compatibility constraint. As the model formulation involves the richer class of probabilistic policies for server, the respective fluid model is also probabilistic. We discuss such a formulation may result in a higher profit compared to its deterministic counterpart in the next section.

%\fc{Would be nice here if we could say something about how this can apply to other settings and how we expect this to be a general approach.}

\subsection{Bounds and Structural Properties}

 We leverage the convexity of the objective in \eqref{eq: prob_fluid_model} to draw a connection between the probabilistic fluid problem and the stochastic problem in \eqref{eq:opt_stoch}.
\begin{proposition} \label{prop: fluid_model}
Let $\pi=(\blambda(\cdot),\alpha(\cdot),\bx(\cdot))$ be a feasible solution of \eqref{eq:opt_stoch} then 
\begin{equation*}
R(\pi)\leq P(\pi)
%\E{\bbarq}{\inner{F(\blambda(\bbarq))}{\blambda(\bbarq)}}-\E{\alpha(\bbarq)}{\cost(\bmu)}
    \leq \tilde{R}^\star.
\end{equation*}
That is, the fluid profit is an upper bound for the stochastic profit and net profit. 
\end{proposition}
The proposition provides an upper bound on the net average profit achievable under any policy. We now present the following lemma which is a crucial step in the proof of the proposition.
\begin{lemma} \label{lemma: necessary_constraints}
For a given stationary Markovian pricing and matching policy $(\blambda(\cdot),\alpha(\cdot),\bx(\cdot))$ with stationary distribution %\fc{shall we use a different symbol here given \eqref{eq: variational_inequality}, or maybe different notation in \eqref{eq: variational_inequality}.}
$\psi(\cdot)$, let $\btlambda=\E{}{\blambda(\bbarq)}$, $\talpha(A)=\sum_{\bq \in \calS} \alpha_\bq(A \times\{\bq\}) \psi(\bq)$ for all Borel $A \subseteq \Omega$ and $\btchi=\E{}{\bx(\bbarq)}$. If $(\blambda(\cdot),\alpha(\cdot),\bx(\cdot))\in \StablePol$, $\E{}{\inner{\bone_{n+m}}{\bbarq}}<\infty$ and $\bx(\cdot)$ satisfies \eqref{eq: matching_constraints},
then $(\btlambda,\talpha,\btchi)$ is feasible in the probabilistic fluid problem \eqref{eq: prob_fluid_model}.
%the constraints of the optimization problem \eqref{eq: prob_fluid_model} are necessary for stability.
\end{lemma}
Intuitively, the above lemma is enforcing that the 'average' arrival rates over the states must be balanced, otherwise, in the long run, some of the queues will keep accumulating the arrivals and that will lead to an unstable system. Thus, the constraints of the fluid model are necessary for stability. This implies that the set $\StablePol$ is a subset of the feasible region of the fluid model. In addition, as the objective \eqref{eq: net_profit_stochastic} is convex in its parameters, we can use Jensen's inequality to obtain the objective of the fluid solution. These two key steps, Lemma \ref{lemma: necessary_constraints} and Jensen's inequality, together deliver Proposition \ref{prop: fluid_model}.

As mentioned above, an important feature of our probabilistic fluid formulation is that by optimizing over the spaces of measures we gain a convex objective. As a result, we obtain an infinite dimensional program, however, as we will see in Section \ref{sec: variations}, this program can be converted into a finite dimensional one. Importantly, there are conditions under which we can reduce our fluid problem to a traditional fluid formulation which showcases the connection between our new approach with classical ones that typically replace stochastic quantities by their deterministic counterparts.
%and also discuss the implications of the solution to the optimal pricing policy.
We identify a condition for the cost function $\cost(\cdot)$ such that the optimal value of the probabilistic fluid optimization problem \eqref{eq: prob_fluid_model} coincides with that of 
a non-probabilistic fluid problem. Moreover, we establish that, under this condition, the optimal fluid server pricing policy is a Dirac probability measure, i.e. a deterministic policy. The optimization problem and the result are presented below.
\vspace{-30pt}
\begin{subequations}
\begin{align}
    \span \tilde{R}^{\star}_{co}\triangleq\max_{(\btlambda,\btmu,\btchi):\btmu\in \Omega} \inner{F(\btlambda)}{\btlambda}-\cost(\btmu)\\
 \textit{subject to,} \ \tl_j&=\sum_{i=1}^n \tchi_{ij} \ \forall j \in [m], \quad \\
    \tchi_{ij}&=0 \ \forall (i,j) \notin E, \ \tchi_{ij} \geq 0 \ \forall (i,j) \in E, \label{eq: convex_compatibility}
\end{align}
\label{eq: convex_prob_fluid_model}%
\end{subequations}
\vspace{-30pt}
\begin{proposition} 
If $\cost(\cdot)$ is convex, then $\tilde{R}^\star = \tilde{R}^\star_{co}$ and there exists an optimal solution of \eqref{eq: prob_fluid_model}
$(\btlambda^\star,\talpha^\star,\btchi^\star)$
such that $\talpha^\star$ is a  Dirac probability measure.
\label{prop: convexity_LP_reduction}%
\end{proposition}
The proof follows by using Jensen's inequality in the objective function of \eqref{eq: prob_fluid_model} and then viewing the optimization problem in the space of 
$(\btlambda,\E{\talpha}{\btmu},\btchi)$.
This proposition simplifies the fluid model and also provides sufficient conditions under which the optimal fluid pricing policy is deterministic. In turn, in our problem, whenever the cost function $\cost(\cdot)$ is convex, it is possible to simply replace the stochastic quantities in \eqref{eq:opt_stoch} by their deterministic counterparts and obtain a fluid upper bound. However, when $\cost(\cdot)$ is general so that it can encompass different strategic settings (c.f. Section \ref{sec: variations}), the convexity assumption might not hold and our probabilistic fluid formulation is needed.

To understand further, consider the case when the cost function is not convex. Then, 
under a probabilistic policy $\talpha$, Jensen's inequality might be violated, that is, $\E{\talpha}{\cost(\btmu)} < \cost(\E{\talpha}{\btmu})$. This will lead to lower cost under the probabilistic policy compared to the corresponding deterministic policy, $\E{\talpha}{\btmu}$. Thus, the richer space of probabilistic policies may obtains a larger profit with respect to the deterministic space of policies, and it also makes the optimization problem more amenable to analysis by turning the objective function into a convex function. This is the main advantage of our probabilistic fluid model: it convexifies the objective of an otherwise intractable problem. Lastly, as the revenue function $(F_j(\lambda_j)\lambda_j)$ is concave, probabilistic policies on the customer side is not essential as they will not result in higher profit compared to its deterministic counterpart. 

To extend this analysis to the stochastic setting, if we naively try to use the optimal solution of the probabilistic fluid model as the pricing policy for all $\bq \in \calS$, then in each time slot, we will receive fluid profit in expectation. However, the system becomes unstable as argued in \citep{caldentey2009fcfs} and thus, the stationary distribution doesn't exist. To see why the system is unstable, consider the case of single link two sided queue operating under the pricing policy given by the fluid solution. It will just be a random walk on
$\bbZ$ which is known to be null recurrent. This provides us with the intuition that we need to operate close to the fluid solution but we need to perturb the arrival rates for customers and/or the probability measure for servers such that the system becomes stable. In the next sections, we will consider pricing policies which are a small perturbation of the fluid solution and show that the net profit and profit under that policy is `sufficiently' close to $\tilde{R}^\star$. 

\section{Asymptotic Optimality of Two Price Policy}
In this section, we will analyze the stochastic system and show that the net profit obtained is `sufficiently' close to the upper bound $\tilde{R}^\star$. To show this, we will consider a large market asymptotic regime indexed by $\eta$.
In this regime, we propose a dynamic two-price policy   and show that its corresponding net profit  converges to that of the scaled optimal fluid solution. As we consider a DTMC, let $w$ denote the time between two transitions. We define our asymptotic regime below.
\begin{definition}[Asymptotic Regime] \label{defn: asymptotic_regime}
We study the system in the large market regime, wherein for the $\eta^{th}$ system, the time between two decision epochs is scaled by $1/\eta$ and the arrivals between two decision epochs remains the same in the stochastic sense. Mathematically,
\begin{align*}
    w_\eta=\frac{w}{\eta}, \quad \ba_\eta(k)=\ba(k), \bq_\eta(k)=\bq(k) \ \forall k 
    \in \bbZ_+.
\end{align*}
\end{definition}
Our convention is to subscript by $\eta$ all the parameters which are associated with the $\eta^{th}$ system. For example, the steady state queue length vector is denoted as $\bbarq_{\eta}$ and the corresponding arrival and matching random variables by $\bbara_{\eta}$ and $\bbarx_{\eta}$, respectively.
The time scaling leads to a large volume of arrivals per unit time and more frequent matching decisions. This is desirable as the inflow of customers and servers increases, it is advantageous to make the matching decision more frequently. 

Note that, under the asymptotic regime, the optimal fluid solution will be $\tilde{R}^\star_\eta=\eta \tilde{R}^\star$. This is because the time is scaled by $\eta$ which leads to the profit per unit time to be scaled by $\eta$.
Now, motivated by our upper bound in Proposition \ref{prop: fluid_model}, we will define our main metric of analysis, the `net profit-loss'.
%\fccut{the `net profit-loss', a metric to analyze various policies.}
\begin{definition}[Net Profit-Loss]
For a given pricing and matching policy $\pi_\eta$, the net profit-loss, $L_\eta(\pi_\eta)$, for the $\eta^{th}$ system is defined as the difference of the optimal profit $\tilde{R}^\star_\eta$
%\fc{should this be $\tilde{R}^\star_\eta$ (also look at the proof of lemma 6 to make sure the notation is right).} 
and the long run average net profit obtained under that policy
\begin{equation*}
L_\eta(\pi_\eta) \triangleq \tilde{R}^\star_\eta-R_\eta(\pi_\eta).    
\end{equation*}
In addition, we define the profit-loss as $L_\eta^P(\pi_\eta)
\triangleq \tilde{R}_\eta^\star -P_\eta(\pi_\eta)$.
\end{definition}
We say that a sequence of policies $\{\pi_\eta\}$ is asymptotically optimal if 
\begin{equation}
\limsup_{\eta \rightarrow \infty} \frac{L_{\eta}(\pi_\eta)}{\eta}=0.
\end{equation}
%A sequence of policies is optimal if $\lim\sup_{\eta \rightarrow \infty} \frac{L^{\eta}}{\eta}=0$ \fc{we say that a sequence of policies is optimal if.....? check}.
Thus, any policy which leads to $o(\eta)$ net profit-loss is asymptotically optimal. 

Now that we have defined a criterion to analyze a given policy, we introduce a sequence of policies which are asymptotically optimal.
The idea is to design a policy that operates as close to the fluid solution as possible because that will result in fluid optimal profit. Denote the optimal solution of the probabilistic fluid problem as $(\btlambda^\star,\talpha^\star,\btchi^\star)$. Note that, without loss of generality, we can assume $\btlambda^\star>\bzero_m$, $\E{\talpha^\star}{\bmu}>\bzero_n$ and $\tchi^\star_{ij}>0$ for all $(i,j) \in E$, otherwise, we can remove that vertex/edge from the graph and work with a smaller graph such that the above conditions are satisfied.  
%For the pricing policy, the idea is to operate as close to the fluid solution as possible because that will result in fluid optimal profit. 
%Making this intuition concrete, 
Now, we introduce the two price policy:
\begin{align}
    \lambda_{\eta,j}(\bq)&=\begin{cases}
    \tilde{\lambda}_j^\star+\epsilon_{\eta} &\textit{if } q_{j}^{(2)}=0; \\
     \tilde{\lambda}_j^\star-\epsilon_{\eta} &\textit{otherwise}; 
    \end{cases} \quad\quad
    \alpha_{\eta,\bq}=\tilde{\alpha}^\star, \ \forall \bq \in \calS, \label{eq: two_price_policy}
\end{align}
%The co-variance matrix of the customer arrival is $\Sigma^{(2)}$ and that of drivers is $\Sigma^{(1)}(\alpha^\star)$ for all $\bq \in \calS$. 
where $\alpha_{\eta,\bq}=\tilde{\alpha}^\star$ is to denote that for all Borel subsets $A \subseteq \Omega$ and all $\bbarq, \bbarq' \in \calS$, we have $\alpha_{\eta,\bq}(A \times\{\bbarq'\})=\tilde{\alpha}^\star(A)\mathbbm{1}\{\bbarq'=\bbarq\}$. We assume that $\epsilon_\eta \rightarrow 0$ as $\eta \rightarrow \infty$ as we want to approach the fluid optimal pricing policy. Without loss of generality, we can assume $\epsilon_\eta \leq 1$ for all $\eta$.
%Note that, the pricing policy introduced here is considerably simpler than \citep{} as we use two different rates only on customer side and on the driver side, we use a fixed probabilistic policy. 
We highlight the simplicity of this pricing policy in which
 we use two different rates only on the customer side and on the server side, we use the fluid optimal probabilistic policy. 
 In addition, the threshold at which we change the rate is at $q_j^{(2)}=0$.
 This captures how the service provider needs to adjust its pricing policy 
 to maintain a stable system and, in turn, sustain the proper balance of supply and demand. 

The matching policy we use is the max-weight matching policy which is defined as:
\begin{align}
  \by_\eta(k)=\arg\max_{\bv \in \bbZ_+^{|E|}} \sum_{(i,j) \in E} v_{ij}(q_{\eta,i}^{(1)}(k)+q_{\eta,j}^{(2)}(k)) \quad
    \textit{subject to } \eqref{eq: matching_constraints}. \label{eq: modified_max_weight_matching}
\end{align}
We use $\pi_\eta$ to refer to the policy defined in 
\eqref{eq: two_price_policy} and \eqref{eq: modified_max_weight_matching}.
Now, we will present the main theorem of this paper.

\begin{theorem} \label{theo: two_price_policy}
Consider a sequence of DTMCs parametrized by $\eta$ operating under the pricing and matching policy $\pi_\eta$.
%as defined in \eqref{eq: two_price_policy} and \eqref{eq:modified_max_weight_matching}. 
Then the net profit-loss $L_{\eta}(\pi_\eta)$ is $O(\eta^{1/3})$ for the choice $\epsilon^{\eta}=\eta^{-1/3}$.
\end{theorem}
\proof{Proof sketch.}
 The main reason we obtain an $\eta^{1/3}$ net profit loss is due to the trade off between the expected queue length and profit-loss. Consider a pricing policy which deviates from the fluid optimal pricing policy by at-most $\epsilon$, that is, for all $\bq \in \calS$, we have $|\lambda_j(\bq)-\tilde{\lambda}_j^\star|\leq \epsilon$ for all $j \in [m]$ and $|\E{\alpha(\bq)}{\mu_i}-\E{\talpha^\star}{\mu_i}|\leq \epsilon$ for all $i \in [n]$. 
 Then, then the  expected queue length is of the order $\frac{1}{\epsilon}$ and the profit loss is of the order $\eta\epsilon^2$. In particular, $\epsilon$ characterizes the drift of the DTMC towards zero which is analogous to the traffic intensity in a single sided queue. It is known that the queue length in a single sided queue scales as $1/\epsilon$ when $\epsilon$ is small (which is called the heavy traffic regime). In addition, the expression of profit-loss can be expanded using Taylor's series expansion. The first order term can be shown to be zero by using the optimality of $(\btlambda^\star,\talpha^\star)$. The second order term results in order $\epsilon^2$ loss.
Hence, considering the trade off between expected queue length and profit loss, the best $\epsilon$ is $\eta^{-1/3}$ which results in $\eta^{1/3}$ net profit loss.

 %There are two major steps in proving this theorem. First, we need to bound the expected sum of queue lengths under the given policy and second, we need to bound the loss in average profit $(R^\star_\eta-P_\eta)$.
 
 Both the steps mentioned above require special treatment because of the strategic 
%\fc{I think we are using both behaviour and behavior, let's stick to behavior which I believe it the American version.}
behavior of the servers. In addition, analyzing the queueing system is more complicated than \citep{varma2020dynamic} as the arrival process has a general distribution. 
$\Halmos$
\endproof
Here, we present multiple lemmas which assists us in proving the theorem and outline the major steps in the proof. Firstly, under the given pricing and matching policy, we show that the system is stable and we upper bound the expected sum of queue lengths. 
\begin{lemma} \label{lemma: stability}
For all $\eta>0$, the discrete time Markov chain operating under the pricing and matching policy $\pi_\eta$ is positive recurrent and there exists a constant $B>0$ such that $
 \E{}{\inner{\bs}{\bbarq_\eta}}\leq B\frac{1}{\epsilon_\eta}.$
\end{lemma}
We use the Foster-Lyapunov theorem \citep[Theorem 3.3.7]{srikantbook}  to prove positive recurrence. In particular, we considered a quadratic Lyapunov function and analyzed its drift. Then we use the moment bound theorem \cite[Proposition 6.14]{hajek2015random} to get bounds on the sum of expected queue lengths. Due to the strategic behaviors of the servers, the arrivals to different queues are co-related and the co-variance of the arrival process appears in the constant %\fc{make sure this is the right constant from the statement}
$B$.

After we prove that the system is positive recurrent by Lemma \ref{lemma: necessary_constraints}, we know that the arrival rates under the two price policy satisfy the constraints of the fluid optimization problem \eqref{eq: prob_fluid_model}. We will use this idea to show the following equality which, in turn, will be useful to obtain the profit-loss bound.
\begin{lemma} \label{lemma: first_order_optimality}
For all $\eta>0$, the DTMC operating under the 
pricing and matching policy $\pi_\eta$,
%given by \eqref{eq: two_price_policy} \eqref{eq: modified_max_weight_matching}, 
such that $\E{}{\inner{\bone_{n+m}}{\bbarq_\eta}}<\infty$ the following holds:
\begin{align*}
    \sum_{j=1}^m \left(\tl^\star_jF'(\tl^\star_j)+F_j(\tl^\star_j)\right)\left(\P{}{\barq_{\eta,j}^{(2)}>0}-\P{}{\barq_{\eta,j}^{(2)}=0}\right)=0.
\end{align*}
\end{lemma}
%\fc{Let's talk about this. I think here there is room to explain a bit more why the probabilistic fluid is different.}
The above lemma is proved based on the optimality of $(\btlambda^\star,\talpha^\star,\btchi^\star)$ in the probabilistic fluid model \eqref{eq: prob_fluid_model}. If the optimization problem \eqref{eq: prob_fluid_model} was over a finite dimensional vector space, we can simply use KKT conditions to get the result. The difficulty here is that  the optimization is  over the space of measures $\talpha$. We overcome this challenging by restricting over $\talpha=\talpha^\star$ and considering an optimization problem over $(\btlambda,\btchi)$ and then using the KKT conditions.

To apply KKT conditions, we find a feasible direction at the optimal point. A feasible point of the optimization problem is the `average' arrival rates of the two price policy \eqref{eq: two_price_policy} as the DTMC is stable under two price policy and $\StablePol$ contains the feasible region of the fluid model.

Now, we use the above lemma to find the profit loss $L_\eta^P$.
\begin{lemma} \label{lemma: profit_loss}
For all $\eta>0$, the profit loss of the DTMC operating under
the 
pricing and matching policy $\pi_\eta$
%given by \eqref{eq: two_price_policy} \eqref{eq: modified_max_weight_matching} 
is
\begin{align*}
    \frac{L_\eta^P}{\eta}= -\epsilon_\eta^2\sum_{j=1}^m \left(\frac{\tl^\star_jF''(\tl^\star_j)}{2}+F'_j(\tl_j^\star)\right)+O\left(\epsilon_\eta^3\right) 
   \quad  \textit{where,} \quad \sum_{j=1}^m \left(\frac{\tl^\star_jF''(\tl^\star_j)}{2}+F'_j(\tl_j^\star)\right)<0.
\end{align*}
\end{lemma}
We use Assumption \ref{ass: monotonic} and Taylor's series expansion up to second order of the demand curve $F_j(\cdot)$ for all $j \in [m]$. We then apply Lemma \ref{lemma: first_order_optimality} to eliminate the first order term of the expansion which, in turn, delivers the desired result.

\section{Lower Bounds}
In this section, we will make the intuition provided for $\eta^{1/3}$ rigorous by showing that, under a broad class of policies, this is the best possible trade off between expected queue length and profit loss. 
We first establish that the
expected queue length is at least
$O\left(\frac{1}{\epsilon}\right)$. Then, we consider a broad class of policies and show that the profit loss is exactly of order $O(\eta \epsilon^2)$. 
In turn, by choosing $\epsilon=1/\eta^{1/3}$, we deduce that
the profit loss is of order $\eta^{1/3}$.

For the remainder of this section, we make the following mild additional assumptions on the arrival process.
%in this section. 
For a given $j$, if we have $\lambda_j(k) \geq \lambda_j(k')$ for some $k,k' \in \bbZ_+$, we assume that $a_j^{(2)}(k) \stg a_j^{(2)}(k')$. 
This assumption, in the economic context, translates to rationality of customers, i.e. if the system operator offers the same service for a lower price, the customer arrival distribution can only shift to the right.
Similarly, for servers, if $\E{\alpha(k)}{\mu_i} \geq \E{\alpha(k')}{\mu_i}$, we assume that $a^{(1)}_i(k) \stg a^{(1)}_i(k')$.

\subsection{Expected sum of queue length}

To provide intuition, we will first describe an illustrative example. Consider a single link two sided queue $(q^{(1)},q^{(2)})$ with Bernoulli arrivals, that is $a^{(1)}(\bq) \sim \textrm{Bernoulli}(\E{\alpha_\bq}{\mu})$ and $a^{(2)}(\bq) \sim \textrm{Bernoulli}(\lambda(\bq))$. Note that for the case of single link two sided queue, there is no selfish behavior of servers as there is only a single type of server. Now, let us analyze the imbalance given by $z(k)=q^{(1)}(k)-q^{(2)}(k)$. Note that, as there is no incentive to keep the customers or servers waiting in the system, we will immediately match any pair of customer-server waiting in the system. Thus, $q^{(1)}(k)q^{(2)}(k)=0$ for all $k \in \bbZ_+$ with probability 1. So, $z$ completely describes the state of the system and thus, it is a Markov chain. In fact, it is a birth and death process as shown in Figure \ref{fig: intuition_singlelinkbirthanddeath} where $l_z=\lambda(z)(1-\E{\alpha_z}{\mu})$, $m_z=\E{\alpha_z}{\mu}(1-\lambda(z))$ and $p_z=1-l_z-m_z$ for all $z \in \bbZ$. Now, consider a general pricing policy such that we are at most $\epsilon$ away from the optimal fluid solution, that is, for all $z \in \calS$, we have $|\lambda(z)-\tl^\star|\leq \epsilon$ for all $j \in [m]$ and $|\E{\alpha_z}{\mu}-\E{\talpha^\star}{\mu}|\leq\epsilon$ for all $i \in [n]$. Thus,
\begin{align*}
    |l_z-\tl^\star(1-\E{\talpha^\star}{\mu})| &\leq \epsilon-\epsilon^2,\quad l_{\min}\geq (\tl^\star-\epsilon)(1-\E{\talpha^\star}{\mu}-\epsilon), \\ |m_z-\E{\talpha^\star}{\mu}(1-\tl^\star)| &\leq \epsilon-\epsilon^2, \quad m_{\max}\leq (\E{\talpha^\star}{\mu}+\epsilon)(1-\tl^\star+\epsilon).
\end{align*}
We can couple this birth and death process with an M/M/1 queue, $q^\dagger$, with $\textrm{Bernoulli}(\tl^\star-\epsilon)$ arrival and $\textrm{Bernoulli}(\E{\talpha^\star}{\mu}+\epsilon)$ service. The coupling is such that $q^\dagger(k) \leq |z(k)|$ for all $k \in \bbZ_+$ with probability 1. By Kingman's bound, we know that $\E{}{\barq^\dagger}\sim \frac{1}{\epsilon}$. Thus, by the above defined coupling, $\E{}{|\barz|}$ is at least $\frac{1}{\epsilon}$. In short, if we perturb the arrival rates of a two sided queue by at most $\epsilon$ then it behaves like a single server queue in heavy traffic. 

In the next theorem, we show a similar lower bound for the more general system of multiple link two sided queue with arbitrary arrival process.
%We will make this intuition rigorous in Theorem \ref{theo: lowerbound_queue}, where we show this for the more general system of multiple link two sided queue with arbitrary arrival process.
\begin{figure}
\FIGURE{
      \begin{tikzpicture}[->, >=stealth', auto, semithick, node distance=2.5cm,scale=0.65,transform shape]
\tikzstyle{every state}=[fill=white,draw=black,thick,text=black,scale=1]
\node[state]    (0)                     {$0$};
\node[state]    (1)[right of=0]   {$1$};
\node[state]    (nminusone)[right of=1]   {\tiny{$n-1$}};
\node[state]    (n)[ right of=nminusone]   {$n$};
\node[state,white]    (nplusone)[ right of=n] {};
\node [state](-1)[left of=0] {$-1$};
\node[state](-n+1)[left of=-1]{\tiny{$-n+1$}};
\node[state](-n)[left of =-n+1]{{$-n$}};
\node[state,white](-n-1)[left of=-n]{};
\path
(0) edge[bend right,below]     node{$m_0$}         (1)
(1) edge[bend right,below,dashed] node{$m_1$} (nminusone)
(nminusone) edge[bend right,below]     node{$m_{n-1}$}         (n)
(n) edge[bend right,below]     node{$m_n$}         (nplusone)
(nplusone) edge[bend right,above]     node{$l_{n+1}$}         (n)
(n) edge[bend right,above]     node{$l_n$}         (nminusone)
(nminusone) edge[bend right,above,dashed]     node{$l_{n-1}$}         (1)
(1) edge[bend right,above]     node{$l_1$}         (0)
(0) edge[bend right,above] node{$l_0$} (-1)
(-1) edge[bend right,above,dashed] node{$l_{-1}$} (-n+1)
(-n+1) edge[bend right,above] node{$l_{-n+1}$} (-n)
(-n) edge[bend right,above] node{$l_{-n}$} (-n-1)
(-n-1) edge[bend right,below] node{$m_{-n-1}$} (-n)
(-n) edge[bend right,below] node{$m_{-n}$} (-n+1)
(-n+1) edge[bend right,below,dashed] node{$m_{-n+1}$} (-1)
(-1) edge[bend right,below] node{$m_{-1}$} (0)
(-n) edge[loop, above] node{$p_{-n}$} (-n)
(-n+1) edge[loop, above] node{$p_{-n+1}$} (-n+1)
(-1) edge[loop, above] node{$p_{-1}$} (-1)
(0) edge[loop, above] node{$p_{0}$} (0)
(1) edge[loop, above] node{$p_{1}$} (1)
(nminusone) edge[loop, above] node{$p_{n-1}$} (nminusone)
(n) edge[loop, above] node{$p_{n}$} (n);
\end{tikzpicture}}{
\centering{Single link two sided queue with Bernoulli arrivals}
\label{fig: intuition_singlelinkbirthanddeath}}{}
\end{figure}
\begin{theorem} \label{theo: lowerbound_queue}
Consider a DTMC operating under any matching policy and pricing policy in $\StablePol$ such that for all $\bq \in \calS$, we have $|\lambda_j(\bq)-\tl_j^\star|\leq \epsilon$ for all $j \in [m]$ and $|\E{\alpha_\bq}{\mu_i}-\E{\talpha^\star}{\mu_i}|\leq\epsilon$ for all $i \in [n]$, then there exists $\epsilon_0>0$ such that for all $\epsilon<\epsilon_0$
\begin{align*}
    \E{}{\inner{\bone_{n+m}}{\bbarq}}  \geq \frac{\bone_{n \times n}\circ \Sigma^{(1)}(\talpha^\star)+\bone_{m \times m} \circ \Sigma^{(2)}_{\min}}{8\max\{m,n\}\epsilon}.
\end{align*}
\end{theorem}
The theorem is proved based on the intuition for the case of single-link two-sided queue by constructing a coupling between the imbalance of the original DTMC and a G/G/1 queue.
\subsection{(Net) Profit-Loss}
In this section, we will restrict ourselves to a broad class of pricing policies and show that the profit-loss is $\tilde{R}^\star_\eta-P_\eta \sim \eta\epsilon^2$. We fix the server pricing policy to the optimal fluid pricing policy and consider a broad class of pricing policies for customers. In particular, we consider policies of the following form:
\begin{align}
    \lambda_{j,\eta}(\bq)=\tilde{\lambda}_j^\star+\phi_j\left(\frac{\bq}{\eta^\alpha}\right)\eta^\beta, \quad \forall \bq \in \calS, \eta >0, j \in [m]. \label{eq: general_pricing_policy}
\end{align}
%The motivation behind this format is same as mentioned in the literature \citep{}. 
The first component is the optimal fluid rate and the second component is a queue length dependent adjustment
(c.f., \citep{kim2017value,varma2020dynamic}). The adjustment is decomposed into a scaled queue length dependent adjustment $\phi_j(\cdot)$ and a factor that depends on the scaling parameter $\eta$. We take $\beta<0$ as we want to approach the optimal fluid solution as $\eta \rightarrow \infty$. We impose some technical conditions on the functions $\phi_j(\cdot)$ for all $j \in [m]$.
\begin{assumption} For all $j\in [m]$, $\phi_j(\cdot)$ satisfies the following.
\begin{enumerate}[label=(\alph*)]
    \item There exists $M<\infty$ such that $\sup_{\bq \in \calS}\left(\phi_j(\bq)\right) \leq M$ for $j \in [m]$. \label{condition: bounded}
    \item We have $\alpha+\beta \leq 0$. \label{condition: alpha_beta}
    \item There exists $K>0$ and $\sigma>0$ such that for all $j \in [m]$, if $q_j^{(2)}/\eta^\alpha > K$ or there exists an $i$ such that $(i,j) \in E$ and $q_i^{(1)}/\eta^\alpha>K$, then $\bigg|\phi_j\left(\frac{\bq}{\eta^\alpha}\right)\bigg|>\sigma$. \label{condition: stability}
\end{enumerate}
\label{condition: general_pricing}
\end{assumption}
These conditions are similar to the conditions given in \citep{varma2020dynamic}. Condition \ref{condition: bounded} is a technical assumption which is required for our analysis. Condition  \ref{condition: alpha_beta} states that the scaling of the system state should be less than the scaling of the pricing policy converging to the fluid optimal.
%We believe, there is a phase transition at the critical point $\alpha+\beta=0$ and the system behaves differently when $\alpha+\beta>0$. Relaxing this condition and showing the lower bound is a future work.
Condition  \ref{condition: stability}
establishes the intuitive condition that 
 as the queue length of a customer (or any of its compatible counterparts) is very large, the system operator should decrease (or increase) the arrival rate of the customer. We now present the lower bound on profit-loss.
\begin{theorem} \label{theo: lowerbound_profit_loss}
Consider a sequence of DTMCs parametrized by $\eta$ operating under any pricing policy satisfying Assumption \ref{condition: general_pricing} and any matching policy $\pi_\eta \in \StablePol$. There exists a constant $K>0$, that depends on
$\sigma,M,K,\{F_j(\cdot)\}_{j \in [m]}$ and $\cost(\cdot)$,
and $\eta_1(\beta)>0$ such that for all $\eta>\eta_1$ we have $
    \tilde{R}_\eta^\star-P_\eta(\pi_\eta) \geq K\eta^{2\beta+1}.$
\end{theorem}
The proof involves using Taylor's Theorem to expand the profit-loss and then using Lemma \ref{lemma: first_order} which is the generalization of Lemma \ref{lemma: first_order_optimality} for any policy,
%\fc{do you mean Lemma 5? Otherwise would need to explain a bit more here.}
to drop the first order term. The proof is concluded by showing that the coefficient of the second order term is non zero. 
From Theorem \ref{theo: lowerbound_queue}, we have $\E{}{\bbarq} \geq 1/\eta^\beta$. In turn, from Theorem \ref{theo: lowerbound_profit_loss}, we deduce $ \tilde{R}_\eta^\star-P_\eta \geq K\eta^{2\beta+1}$. To make the best use of the trade off, we should pick $\beta=-1/3$ which will give us $\eta^{1/3}$ net profit-loss. This shows that there exists a broad class of policies under which the upper bound given by Theorem \ref{theo: two_price_policy} is tight. We present the result formally in the following corollary.
\begin{corollary}
\label{corollary: lower_bound_net_profit_loss}
Under the hypothesis of Theorem \ref{theo: lowerbound_profit_loss}, for any sequence of policies $\pi_\eta \in \StablePol$, there exists a constant $K'$,
that depends on $\sigma,M,K,\{F_j(\cdot)\}_{j \in [m]}$ and $\cost(\cdot)$
, and $\eta_2(\beta)>0$ such that for all $\eta>\eta_2$, we have $L_\eta(\pi_\eta) \geq K'\eta^{1/3}$
\end{corollary}
The proof follows by Theorem \ref{theo: lowerbound_queue} and Theorem \ref{theo: lowerbound_profit_loss} and is presented in the Appendix \ref{app: lower_bound_corollary}.
\section{Variations of Cost function and Simulations} \label{sec: variations}
In this section, we demonstrate the generality of our framework by considering four different variations of the cost function and compare them using numerical simulations. Before that, we expound on the cost function $\cost(\cdot)$ and the equilibrium condition \eqref{eq: variational_inequality} which will provide further insights on the server pricing policy which, in turn, will further impact the customer pricing policy.
%\fc{Should we move the following to section 7? We should talk about this with Siva.}
\subsection{Cost Function Reformulation}
In this section, we will reformulate the cost function given in \eqref{eq: cost_function} as a mixed-integer non-linear program. To do this, we will first reformulate \eqref{eq: variational_inequality} as bi-linear constraints using the KKT conditions \citep{mpec} and the supply constraint $G_i(\bmu_i)=u_i$ as  mixed-integer non-linear constraints. The result is presented below.
\begin{proposition} \label{prop: cost_function_reformulation}
There exists $M>0$ such that the cost function defined in \eqref{eq: cost_function} is equivalent to the value function of the following optimization problem:
\begin{subequations}
\begin{align}
   \span c(\bmu)=\min_{\bp^{(1)}\in \bbR_+^n} \inner{\bmu}{\bp^{(1)}} \\
    \textit{subject to, } \bnu &\in \Delta_n^n, \ u_{il}=\kappa_i-\xi_{il}, \ \xi_{il}\nu_{il}=0  \quad \forall i,l \in [n] \label{eq: EQ_reformulation} \\
    G_i(\hatmu_i) &\geq u_{il}, \ G_i(\hatmu_i) \leq u_{il}+(1-b_{il})M, \ b_{il} \in \{0,1\}, \ \sum_{i=1}^n b_{il}=1 \quad \forall i,l \in [n] \label{eq: supply_curve_reformulation}\\
    u_{il}&=f_{il}(\bp^{(1)}) \quad \forall i,l \in [n] \label{eq: utility_reformulation} \\
    \mu_i&=\sum_{l=1}^n \hatmu_l\nu_{li} \quad \forall i \in [n] \label{eq: effective_arrival_reformulation}
\end{align}
\end{subequations}
\end{proposition}
The constraint \eqref{eq: EQ_reformulation} is an equivalent reformulation of the equilibrium constraint \eqref{eq: variational_inequality}. This can be seen by interpreting $\kappa_i$ as $u_i$: If $\nu_{il}>0$, then $\xi_{il}=0$ which implies that $u_{il}=\kappa_i-\xi_{il}=\kappa_i=u_i$ and if $\nu_{il}=0$, then $\xi_{il} \geq 0$ which implies that $u_{il}=\kappa_i-\xi_{il}\leq \kappa_i=u_i$. Next, the constraint \eqref{eq: supply_curve_reformulation} is equivalent to the supply constraint $G_i(\hatmu_i)=u_i$ as $b_{il}=1$ for exactly one $l \in [n]$ which corresponds to the maximizer $(\argmax_{l \in [n]} u_{il})$. The constraint \eqref{eq: utility_reformulation} is the definition of the utility function and \eqref{eq: effective_arrival_reformulation} is the relation between the arrival rates of different types of servers $(\bhatmu)$ and the effective arrival rate of servers to different queue $(\bmu)$. The proof of the proposition follows by showing the equivalences discussed above and is deferred to the appendix. If the supply curve and utility function are linear, then the cost function becomes the value function of a  mixed-integer linear program which can be solved efficiently using standard optimization software packages.
\subsection{Cost Models}
Now we consider four different variations of the cost function. Each variation corresponds to a different model of strategic behavior we impose on servers. We restrict our attention to the utility function given by \eqref{eq: utility} and use numerical simulation to compare the different models. We begin by stating the variations of $\cost(\cdot)$.
% Now, we move our attention to variations of the cost function. We refer to the cost function defined in Section \ref{sec: model} as Selfish Drivers (SD). The cost function is given by \eqref{eq: cost_function} and denote the optimal objective value of \eqref{eq: prob_fluid_model} by $R_*^0$ and optimal solution with $(*,0)$ as the super-script. Now, we define three different variations.

\textbf{Selfish Servers (SD):} This corresponds to the cost function defined in Section \ref{sec: model} and given by \eqref{eq: cost_function}.
We denote the optimal objective value of \eqref{eq: prob_fluid_model} by $R_*^0$ and the optimal solution with $(*,0)$ as the super-script.

\textbf{Incentive Compatible (IC):} In this model we enforce the constraint that servers do not lie. This is equivalent to designing an incentive compatible pricing policy. That is, we ensure that $u_{ii} \geq u_{il}$ for all $l \in [n]$, for all $i \in [n]$. We make an additional assumption that a server will choose its own queue if possible. Thus, we will have $\hatmu_{il}=0$ for all $l \neq i$. The cost function with this new constraint can be re written as follows:
\begin{align*}
  \span \cost^1_{*}(\bmu)=\min_{\bp^{(1)}} \inner{\bmu}{\bp^{(1)}} \\
    \textit{subject to, } u_{il}&=p_l^{(1)}-c_{il} \ \forall i \in [n], \ \forall l \in [n], \quad
    G_i(\mu_i)=u_{ii} \ \forall i \in [n], \quad
    u_{ii} \geq u_{il} \ \forall i \in [n] \ \forall l \in [n].
\end{align*}
By setting the baseline as $c_{ii}=0$ for all $i \in [n]$ and eliminating $\bu$ and $\bp^{(1)}$, we get

\begin{align}
    \cost^1_*(\bmu)=\begin{cases}
    \sum_{i=1}^n G_i(\mu_i)\mu_i &\textit{if } G_i(\mu_i) \geq G_l(\mu_l)-c_{il}, \quad \forall i \in [n], \ \forall l \in [n], \\
    \infty &\textit{otherwise}.
    \end{cases} \label{eq: truthful_nash_eq}
\end{align}
This can be non-convex. To see this, consider a simple case with 2 customer types and 2 server types, i.e. $n=m=2$ and utility given by: $u_{il}=p_l^{(1)}$ for all $i,l \in [n]$. In this case, we have $\Omega=\{\bmu : G_1(\mu_1)=G_2(\mu_2)\}$. Now, even if we consider $G_1(\cdot)$ and $G_2(\cdot)$ to be twice continuously differentiable, monotonically increasing, convex functions, $c(\cdot)$ can still be highly non convex. In particular, $c(\cdot)$ is convex if $G_1-G_2$ is an affine function. We present it in the following corollary.
\begin{corollary}
If $G_i(\cdot)$ is an affine, monotonically increasing function for all $i \in [n]$, then $\cost(\cdot)$ is convex. Thus, by Proposition \ref{prop: convexity_LP_reduction}, we have $\tilde{R}_*=\tilde{R}_c^\star$.
\label{corollary: IC_convex}
\end{corollary}
\proof{Proof}
We know that $G_i(\mu_i)=b_i\mu_i+b_i'$ such that $b_i \geq 0$. Thus, $\inner{G(\bmu)}{\bmu}$ is a quadratic function in $\bmu$ with a positive semi-definite Hessian. Thus, it is convex. In addition, the domain of $\cost(\cdot)$ is a polyhedron as it is defined by a finite number of affine inequalities, thus it is convex.  $\Halmos$
\endproof
The cost function in this variation is given by \eqref{eq: truthful_nash_eq}. We denote the optimal objective value of \eqref{eq: prob_fluid_model} by $R_*^1$ and  the optimal solution with $(*,1)$ as the super-script.

\textbf{$\beta-$ Incentive Compatible ($\beta-$ IC):} In this model we consider a convex combination of the two cases we considered before. That is, we enforce that at least $0<\beta<1$ fraction of each type of servers are truthful, that is, they join their own queue. Thus, we add an additional constraint $\nu_{ii} \geq \beta$ for all $i \in [n]$ or equivalently, $\hatmu_{ii} \geq \beta \sum_{l=1}^n \hatmu_{il}$ for all $i \in [n]$. For $\beta=0$, it is equivalent to the first case (SD) and for $\beta=1$, it is equivalent to the second case (IC). The cost function is given by \eqref{eq: cost_function} with an additional constraint $\hatmu_{ii} \geq \beta \sum_{l=1}^n \hatmu_{il}$ for all $i \in [n]$. We denote the optimal objective value of \eqref{eq: prob_fluid_model} by $R_*^\beta$ and the optimal solution with $(*,\beta)$ as the super-script.

\textbf{First Best, Incentive Compatible (FB-IC):} In this case, all the servers join their own queue irrespective of their utilities. The fluid model is given by \citep{varma2020dynamic}. We denote its optimal objective value by $R_{**}^{1}$ and the optimal solution with  super-script $(**,1)$.

We first present some straightforward relations between the optimal values of \eqref{eq: prob_fluid_model}.
\begin{proposition}
The following statements are true:
\begin{enumerate}
    \item $R_{**}^1 \geq R_*^1$ and if $c_{il} \geq G_l(\mu^{**,1}_l)-G_i(\mu_i^{**,1})$ for all $i,l \in [n]$, then $R_{**}^1 = R_*^1$.
    \item $R_*^{\beta_1} \geq R_*^{\beta_2}$ for all $1 \geq \beta_2 \geq \beta_1 \geq 0$.
    %\item For $\bc_1 \geq \bc_2$, we have $R_*^1(\bc_1) \leq R_*^1(\bc_2)$. %\fc{clarify this notation and explain this last part in the next paragraph.}
\end{enumerate}
\label{prop: comparison_fluid_models}
\end{proposition}
%It is obvious that (2) is true as the feasible region of the optimization problem defining the cost function for $\beta_1$-IC servers contains the feasible region of $\beta_2$-IC servers and their objective functions are identical. In addition, (1) follows by noting that the domain of $c_{*}^1$ is a subset of the domain of $c_{**}^1$ and they are equal in the domain of $c_{*}^1$. %\fc
(1) in the proposition, establishes that if the detour costs are high enough then the optimal IC solution achieves the first best. This is the case because when the costs are high, the service provider does not need to incentivize servers to act truthfully as not doing so is not in their best interest. %Next, (3) in the proposition implies that the optimal SD solution is monotonically decreasing in the detour cost $\bc$. We make the dependence of $R_*^1$ on $\bc$ explicit in the statement of the proposition by denoting it as $R_*^1(\bc)$. Higher cost component wise will result in lower utility for all types of drivers for a given pricing policy which will reduce the arrival rate of drivers joining the system. Thus, a higher incentive is required to maintain the driver inflow to meet the customer demand.

Now, to solve these fluid models numerically, we present an equivalent reformulation of the probabilistic fluid model as a finite dimensional optimization problem in the proposition below.
\begin{proposition}
The probabilistic fluid model \eqref{eq: prob_fluid_model} is equivalent to the following finite dimensional optimization program:
\begin{align*}
    \max_{\btlambda,\{\btmu^l\}_{l=1}^{n+1},\btchi,\btbeta} \inner{F(\btlambda)}{\btlambda}-\sum_{l=1}^{n+1}c(\btmu^{l})\tbeta_l \span \\
    \textit{subject to, } \tl_j&=\sum_{j=1}^n \tchi_{ij} \ \forall j \in [m], \quad
    \sum_{l=1}^{n+1}\tbeta_l \tmu^l_i=\sum_{j=1}^m \tchi_{ij} \quad \forall i \in [n] \\
    \tchi_{ij}&=0 \ \forall (i,j) \notin E, \quad \tchi_{ij} \geq 0 \ \forall (i,j) \in E, \quad
    \inner{\bone_{n+1}}{\btbeta}=1, \ \btbeta \geq \bzero_{n+1}.
\end{align*}
\end{proposition}
\proof{Proof}
We first identify that the primal problem \eqref{eq: prob_fluid_model} is a class of risk averse optimization problem that falls into the category of the problem of moments (Section 6.6, \citep{shapiro2014lectures}). Thus, by \citep[Proposition 6.40]{shapiro2014lectures}, the result follows. \hfill $\Halmos$
\endproof
\subsection{N-Network}
\subsubsection{Cost Function and Fluid Model}
We compare the cost functions and the resultant fluid model for the different cases discussed above. In this subsection, we consider an N-network graph and carry out simulations by varying the inverse supply curves and the penalty due to lying. We start by plotting the contour plots of the cost functions with the penalty $c_{12}=2$ and $c_{21}=5$ for all the different cases and for two sets of supply curves. The results are summarized in Fig. \ref{fig: contour}. It can be observed that for the case of IC and FB-IC, the cost function is convex and for all the other cases, it is non convex. Although, for some choices of supply curves, the cost function is close to convex as in Fig.  \ref{fig: contour} (e), (f).
\begin{figure}[t]
\FIGURE{
    \includegraphics[width=\textwidth]{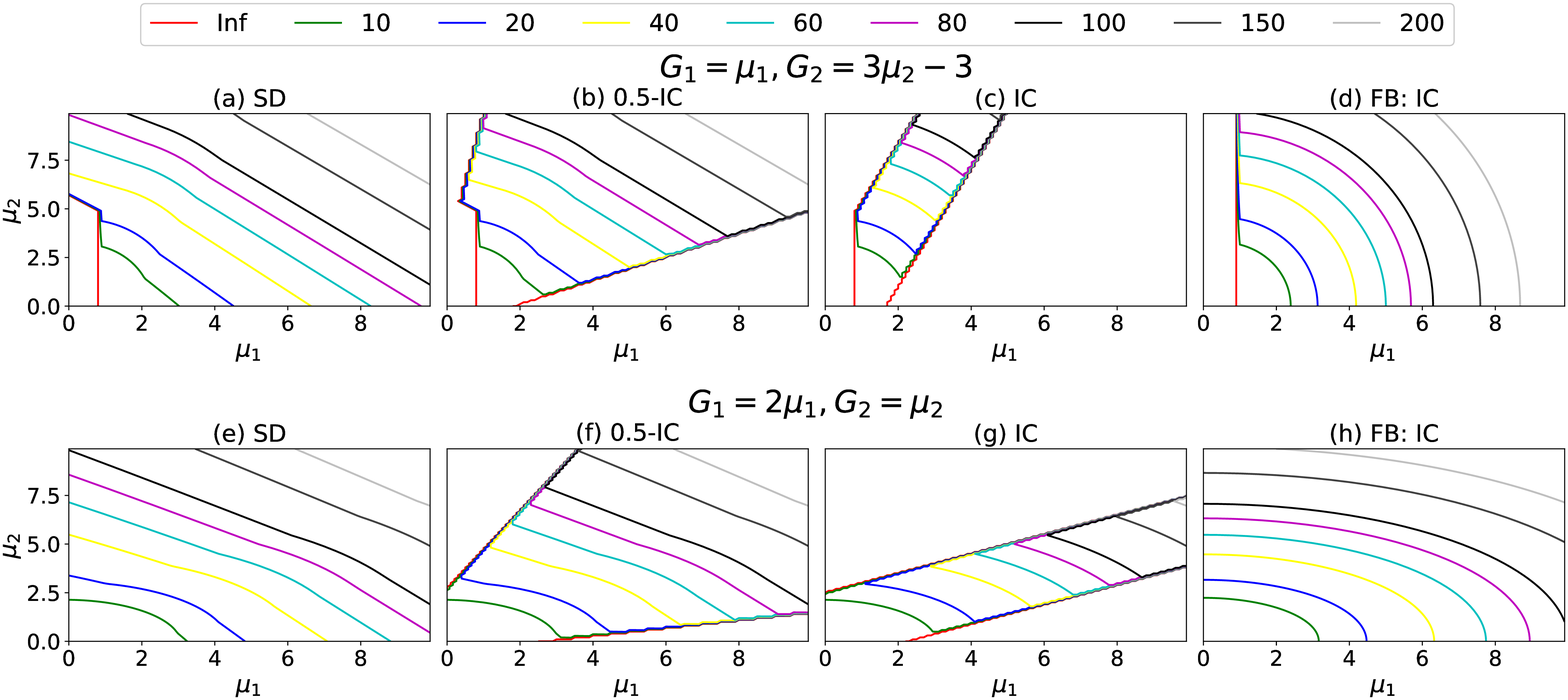}}{
    \centering{Contour plot of cost function $c(.)$ for SD, 0.5-IC, IC and FB: IC with $c_{12}=2, c_{21}=5$.}
    \label{fig: contour}}{}
\end{figure}
We pick linear demand curves given by $F_1(\lambda_1)=10-\lambda_1/2$ and $F_2=15-\lambda_2$. The resultant optimal objective values of \eqref{eq: prob_fluid_model} is summarized in Table \ref{table: fluid_modles}.
\begin{figure}[t]
\FIGURE{
 \begin{tikzpicture}[scale=0.6]
\draw[black, very thick] (0,0) -- (2,0) -- (2,1) -- (0,1);
\node[black,very thick] at (0.25,0.5) {2};
\draw[black, very thick] (0,1.5) -- (2,1.5) -- (2,2.5) -- (0,2.5);
\node[black,very thick] at (0.25,2) {1};
\draw[black,very thick] (8,0) -- (6,0) -- (6,1) -- (8,1);
\node[black,very thick] at (7.75,0.5) {2};
\draw[black, very thick] (8,1.5) -- (6,1.5) -- (6,2.5) -- (8,2.5);
\node[black,very thick] at (7.75,2) {1};
\draw[black,very thick] (13,0) -- (11,0) -- (11,1) -- (13,1);
\node[black,very thick] at (12.75,0.5) {2};
\draw[black, very thick] (13,1.5) -- (11,1.5) -- (11,2.5) -- (13,2.5);
\node[black,very thick] at (12.75,2) {1};
\draw[black,thick]  (2.75, 2.1) edge[<->]  (5.25, 2.1);
\draw[black,thick]  (2.75, 0.6) edge[<->]  (5.25, 1.9);
\draw[black,thick]  (2.75, 0.4) edge[<->]  (5.25, 0.4);
\draw[black,thick]  (8.25, 0.4) edge[<->]  (10.75, 0.4);
\draw[black,thick]  (8.25, 0.6) edge[<->]  (10.75, 1.9);
\draw[black,thick]  (8.25, 2.1) edge[<->]  (10.75, 2.1);
\draw[black,thick]  (8.25, 1.9) edge[<->]  (10.75, 0.6);
\node[black, align=center] at (9.5,2.4){\footnotesize$c_{11}=0$};
\node[black, align=center] at (9.5,0.2){\footnotesize$c_{22}=0$};
\node[black, align=center] at (9.85,1.7){\footnotesize$c_{12}$};
\node[black, align=center] at (9.95,0.8){\footnotesize$c_{21}$};
\node[black, align=center] at (1,3) {\footnotesize Customer};
\node[black, align=center] at (4,3) {\footnotesize \shortstack{Compatible\\Matchings}};
\node[black, align=center] at (7,3) {\footnotesize \shortstack {Server \\  Queue}};
\node[black, align=center] at (12,3) {\footnotesize \shortstack {Server \\ Type}};
\end{tikzpicture}}{
    \centering{N-Network.}
    \label{fig:ex_multiple_link}}{}
\end{figure}
\begin{table}[b]
\TABLE {\caption{Comparison of optimal value of different fluid models.}
\label{table: fluid_modles}}{
    \begin{tabular}{c|c c c| c c c}
        Supply Curve &  \multicolumn{3}{c|}{$G_1=2\mu_1$, $G_2=\mu_2$} & \multicolumn{3}{c}{$G_1=\mu_1$, $G_2=3\mu_2-3$}\\
         $(c_{12},c_{21})$ & $(0,0)$ & $(2,5)$&$(20,50)$ & $(0,0)$ & $(2,5)$&$(20,50)$ \\
         \hline
         $R_*^0$ & 38.19 & 38.19 & 38.19 & 39.75 & 37.37 & 36.91 \\
          $R_*^1$ & 38.19 & 38.19 & 38.19 & 36.86 & 36.91 & 36.91 \\
           $R_{**}^1$ & 38.19 & 38.19 & 38.19 & 36.91 & 36.91 & 36.91
    \end{tabular}}{}
\end{table}
The optimal solution in the case of incentive compatible servers for all the cases is a deterministic pricing policy for the servers as the supply curves are chosen to be linear and thus, the simulation results conform with Corollary \ref{corollary: IC_convex}. In addition, as expected, we have $R_{*}^1 \leq R_{**}^1$. For the first set of supply curves, by statement one of Proposition \ref{prop: comparison_fluid_models}, for all $c_{12},c_{21} \geq 0$, we have $R_{*}^1= R_{**}^1$ and for the second set of supply curves, for all $c_{12} \geq 0.42$ and $c_{21} \geq -0.42$, we have $R_{*}^1= R_{**}^1$. In words, if $\bc$ is large enough, the system operator doesn't need to incentivize the servers and this threshold of penalty depends on the network topology, supply and demand curves.

% The optimal solution in the case of selfish servers for the first and second set of supply curves is a deterministic and probabilistic policy, respectively. By the contour plots of the cost function given in Fig. \ref{fig: contour} (a), (d), we can see that in the first case, it is approximately convex and in the second case, it is non convex.
% This verifies Proposition \ref{prop: convexity_LP_reduction}. The optimal solution with $G_1=\mu_1$, $G_2=3\mu_2-3$, and $\bc=\bzero_{2\times 2}$ are as follows:
% \begin{align*}
%     (\bhatmu^{1})^{0,*}=\begin{bmatrix}
%     3.42 & 2.18 \\
%     0.12 & 0
%     \end{bmatrix} \ (\bhatmu^{2})^{0,*}=\begin{bmatrix}
%     0 & 0.21 \\
%     3.53 & 1.97
%     \end{bmatrix} \ (\bhatmu^{3})^{0,*}=\begin{bmatrix}
%     0.025 & 2.18 \\
%     3.52 & 0
%     \end{bmatrix} \ \begin{matrix}
%     \bbeta^{0,*}=(0.15, 0.31,  0.54) \\
%     \blambda^{0,*}=(2.15,3.58) 
%     \end{matrix}
% \end{align*}
% \begin{align*}
%     (\bhatmu^{1})^{1,*}=\begin{bmatrix}
%     3.42 & 0 \\
%     0 & 2.14
%     \end{bmatrix}  \quad \begin{matrix}
%     \bbeta^{1,*}=(1,0,0) \\
%     \blambda^{1,*}=(3.42,2.14) 
%     \end{matrix} \quad
%     (\bhatmu^{1})^{1,**}=\begin{bmatrix}
%     3.33 & 0 \\
%     0 & 2.25
%     \end{bmatrix}  \quad \begin{matrix}
%     \bbeta^{1,**}=(1,0,0) \\
%     \blambda^{1,**}=(3.33,2.25) 
%     \end{matrix}
% \end{align*}
One crucial observation is that the optimal solution of IC and FB-IC are close to each other even when $\bc$ is small. We analyze this further in the Appendix \ref{app: simulation}.
\subsubsection{Stochastic Simulation}
Now, we analyze the proposed two-price policy and max-weight matching policy for the stochastic system. To analyze the pre-limit behavior of the policy, we calculate the percentage loss compared to the upper bound $\eta R^\star$. Mathematically,
\begin{align*}
    \% \textit{Loss} = \frac{L_\eta}{\eta R^\star} \times 100.
\end{align*}
We will consider the same two sets of supply curves as in the fluid model simulations.

For the case when $G_1(\mu_1)=2\mu_1$ and $G_2(\mu_2)=\mu_2$, all the fluid models have the same optimal solution. We use this optimal fluid arrival rates in the two price policy and simulate the system for different distributions of the arrival rate. We consider binomial distribution $(n',p)$ with $n'=5$ and $n'=8$, and a perturbed uniform distribution with support $\{0,1,2,3,4,5\}$. The success probability of binomial are chosen so that the mean arrival rate matches the two-price policy. Similarly, the uniform distribution is perturbed to match the mean arrival rate with the two-price policy.

For the case when $G_1(\mu_1)=\mu_1$ and $G_2(\mu_2)=3\mu_2-3$, all the fluid models results in different optimal solutions. Thus, we analyze the stochastic system under all these optimal solutions for the case when $c_{12}=c_{21}=0$. For this case, the distribution of arrivals we use is uniform distribution with perturbation on support $\{0,1,2,3,4,5\}$.
\begin{figure}
    \begin{minipage}{0.33\textwidth}
    \FIGURE{
    \includegraphics[width=0.9\linewidth]{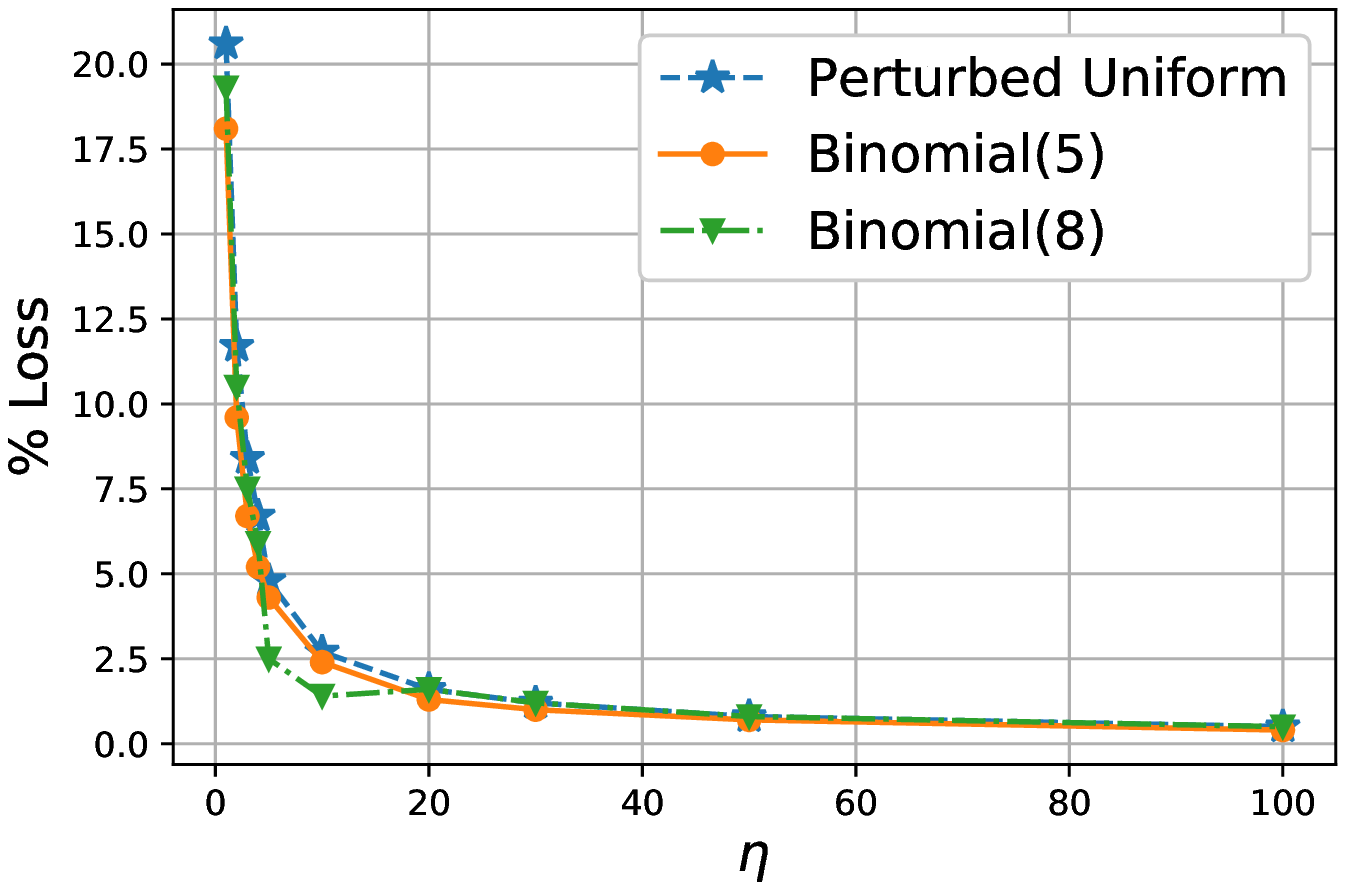}}{
    \centering{Percentage Loss vs $\eta$ with linear supply curves for different arrival distributions}\label{fig: percent_dist}}{}
    \end{minipage}
     \begin{minipage}{0.33\textwidth}
     \FIGURE{
    \includegraphics[width=0.9\linewidth]{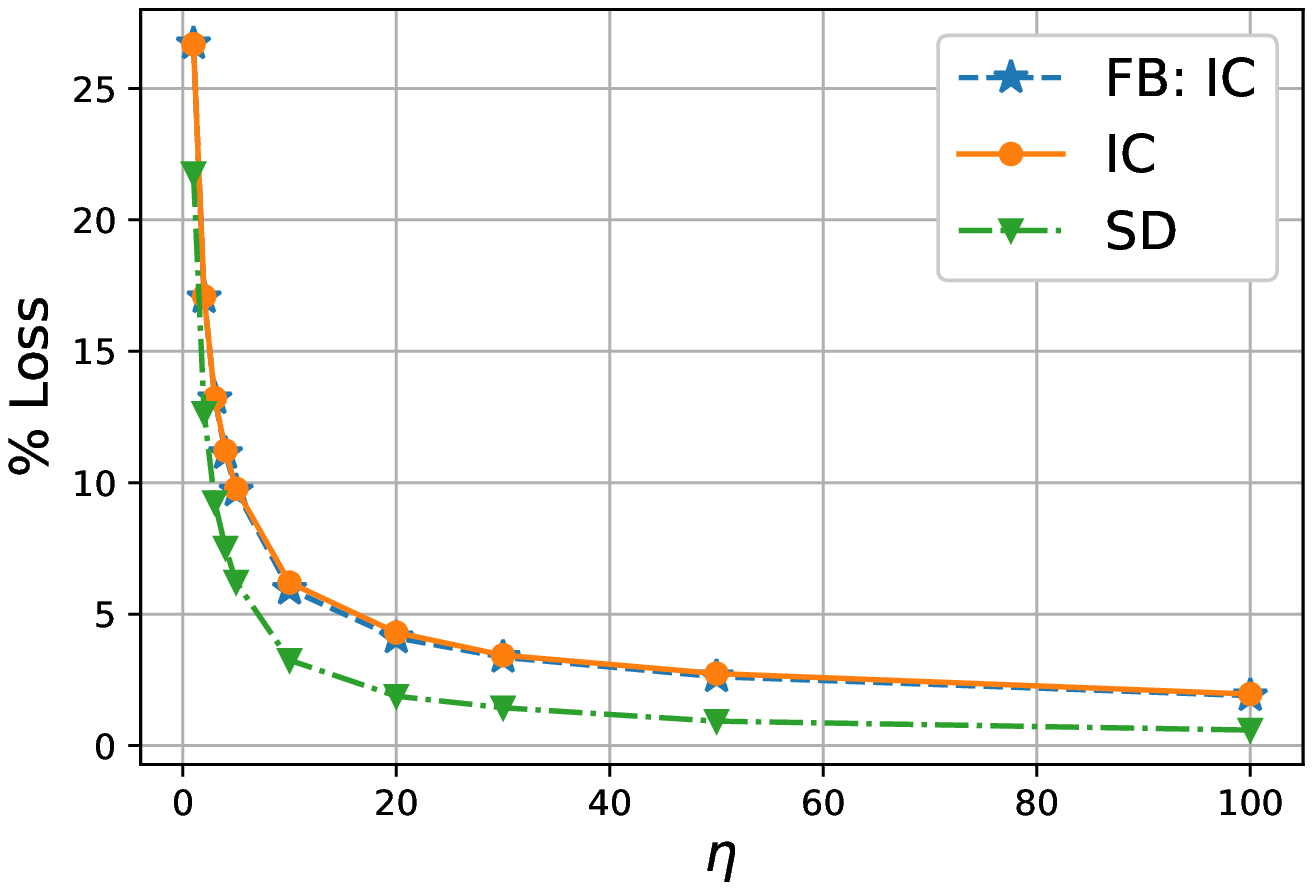}}{
    \centering{Percentage loss vs $\eta$ with affine supply curves under different cost functions}\label{fig: percent_cost}}{}
    \end{minipage}
    \begin{minipage}{0.33\textwidth}
   \FIGURE{\includegraphics[width=.9\linewidth]{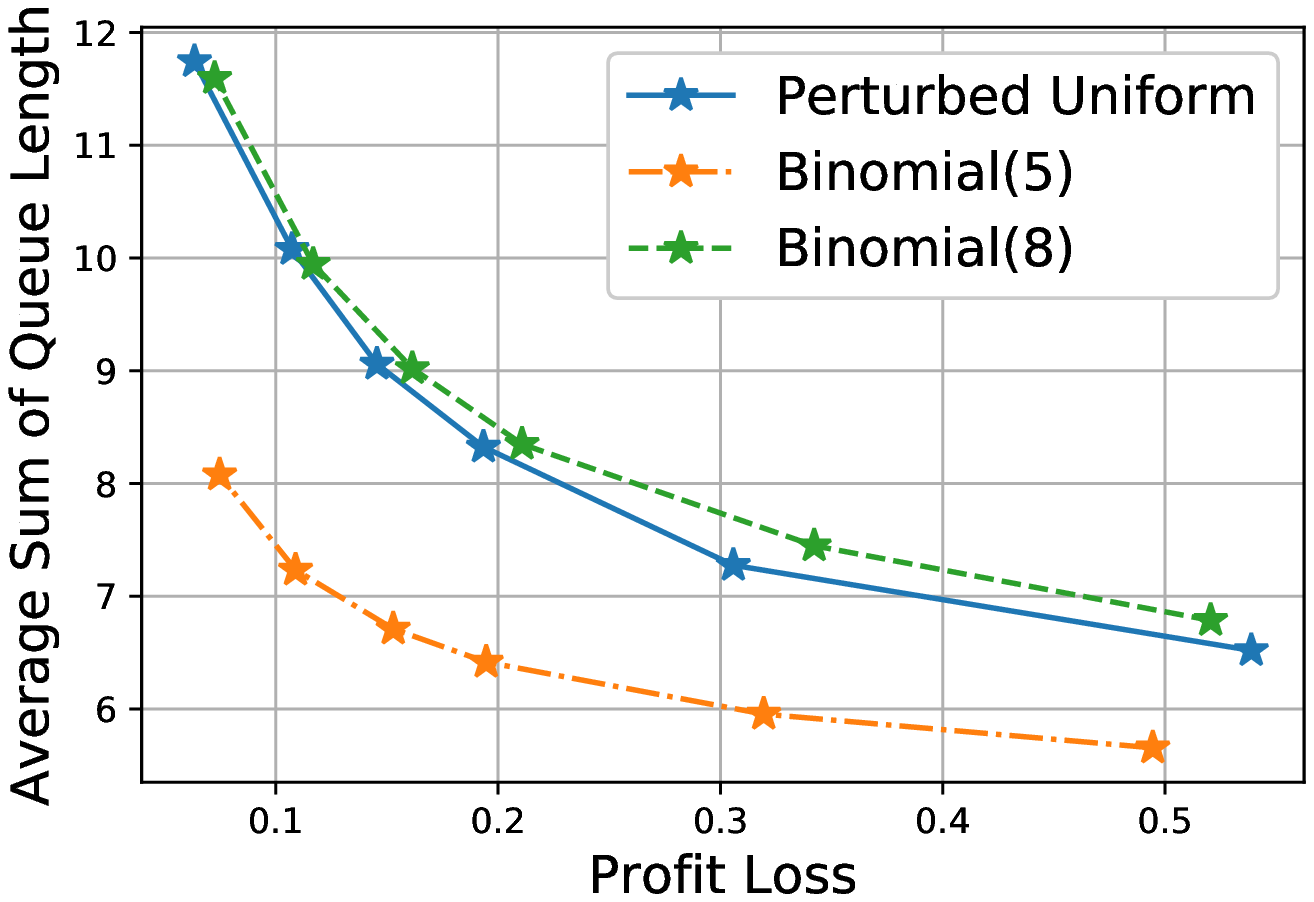}}{
    \centering{Average sum of queue length vs loss in profit with linear supply curves for N-Network}
    \label{fig: q_pl_dist}}{}
    \end{minipage}
    \end{figure}
    %\begin{figure}[t]
    % \FIGURE{\includegraphics[width=.432\linewidth]{trade_off.eps}}{
    % \centering{Average sum of queue length vs loss in profit for non asymptotic systems with linear supply curves for N-Network}
    % \label{fig: q_pl_dist}}{}
    % \end{figure}

The result for both  cases are summarized in Fig. \ref{fig: percent_dist} and Fig. \ref{fig: percent_cost}. The percentage loss decays very fast and less than 5\% error is achieved for $\eta$ as small as 10. This shows the effectiveness of the proposed policy even in the pre-limit system. By Fig. \ref{fig: percent_dist} we can observe that the percentage loss is robust to the change of distribution of the arrival rate and by Fig. \ref{fig: percent_cost} we conclude that it is robust to different cost functions as well.

In addition, we also study the system free of asymptotic regime and the result for the N-Network is plotted in Fig. \ref{fig: q_pl_dist}. Here, we can observe the trade-off between the average sum of queue length and the loss in profit incurred by the system operator. This provides insight on how to appropriately choose the parameter of two price policy to optimize this trade-off. We can observe that higher variance of the arrival process leads to higher queue lengths for the same loss in profit. This is coherent with Theorem \ref{theo: lower_bound}. 
\section{Conclusion}
In this paper, we considered a very general model of two-sided queues with strategic servers. The cost paid to the servers as a function of their arrival rates is formulated as an optimization problem with equilibrium constraints. We consider multiple different models and present their comparison using theoretical and simulation results. Using a general cost function, we introduced a novel probabilistic fluid model which provides an upper bound on the achievable profit under any policy. Then, we presented a two price policy and max-weight matching policy which achieves this upper bound under the large market regime with $O(\eta^{1/3})$ rate of convergence. We also showed that under a broad class of customer pricing policy, the rate of convergence is lower bounded by $\Omega(\eta^{1/3})$ under any matching policy. We conclude our discussion by comparing different equilibrium and analyzing the real-life performance of the probabilistic two-price policy and max-weight matching policy using simulations. We also consider several extensions of our paper which are presented in Appendix \ref{app: extensions}. This asserts that our framework of probabilistic fluid model and stochastic analysis is quite general. In particular, we consider the following four extensions: (1) We establish a concrete trade-off between waiting time and profit-loss by analyzing a scale free system. (2) We consider a slightly different model which penalizes based on the waiting time as opposed to the queue length and show that our proposed policy achieves optimal profit with $O(\eta^{-1/3})$ rate of convergence. (3) We generalize the utility function to additionally depend on the rate of matching customer-server pair and propose an optimal pricing and matching policy. (4) We allow the servers to choose an equilibrium among the ones that maximizes their utility. We analyze the system performance by considering adversarial servers and present a probabilistic fluid model and show that it provides an upper bound on the achievable profit.
% Next, given the service level of the system, we presented the bounds on the achievable profit under the two price policy and max-weight matching policy which is an asymptotic regime free result. Finally, we present some extensions of our model by generalizing the utility function and analyzing the pessimistic model which asserts that our framework of probabilistic fluid model and stochastic analysis is quite general.

\bibliographystyle{informs2014}
\bibliography{references}
\ECSwitch
\renewcommand{\theHsection}{A\arabic{section}}
\begin{APPENDICES}{}%
\section{Extensions} \label{app: extensions}
In this section, we will consider four extensions of our model: (1) We consider a scale-free setup, and analyze the profit obtained by the system operator given the service level constraint which establishes a fundamental trade off between profit and queue length. (2) We consider a slightly different model which penalizes the system operator based on the waiting time and not the queue length. (3) We generalize the utility function to additionally depend on the long run average matching rates $\E{}{\barx}$. (4) We allow the servers to choose equilibrium of their choice among the ones that maximize their utility and analyze the worst case scenario by considering adversarial servers. These extensions will make it apparent that our framework involving the probabilistic fluid model and stochastic analysis of the perturbed policies is very general.
\subsection{A Quality Driven View of the Near Optimal Policy} \label{sec: asymptotic_free}
We present an alternate view of the 
sequence of policies 
we considered in the previous sections. In particular, instead of considering
an asymptotic regime, we analyze the system under a near optimal policy and, critically, impose a given service quality requirement.

To gain intuition, let us consider the
 two price policy given by \eqref{eq: two_price_policy} and \eqref{eq: modified_max_weight_matching}. From Lemma \ref{lemma: stability} and Lemma \ref{lemma: profit_loss}, we know that 
 %\fc{$R_*$ notation}
 $\E{}{\inner{\bone_{n+m}}{\bbarq}}\sim\frac{1}{\epsilon}$ and $\tilde{R}_*-P \sim \epsilon^2$  with $\epsilon_\eta=\epsilon$. 
%In particular, from Lemma \ref{lemma: stability} and Lemma \ref{lemma: profit_loss}, we know that $\E{}{\inner{\bone_{n+m}}{\bbarq}}\sim\frac{1}{\epsilon}$ and $R_*-P \sim \epsilon^2$ under the two price policy given by \eqref{eq: two_price_policy} with $\epsilon_\eta=\epsilon$. 
Now, for the profit to approach the fluid solution, we need to let $\epsilon \rightarrow 0$. However, this causes the expected sum of queue length to go to infinity and, therefore, there could be an arbitrary large loss of service quality impacting both servers and customers.

%is like the service level of the system which measures the performance of the system as it is directly reflects the quality of service provided to the customers as well as drivers. 

In this section, we  maximize the profit (or equivalently, minimize the profit-loss) given a target service level of the system. 
Specifically, we consider the additional constraint that $\E{}{\inner{\bone_{n+m}}{\bbarq}}=C$ for some constant $C>0$. In this case, we need to pick $\epsilon$ to be of the order $\frac{1}{C}$. This will lead to an $O(\frac{1}{C^2})$ profit-loss as
%\fc{notation (also in the statement of theorem 8)} 
$\tilde{R}_*-P \sim \epsilon^2$. If $C$ is large, then the system is allowed to keep  customers and servers waiting  for a longer period of time. This allows the system operator
%which will allow them
to use the policy which is closer to the fluid optimal policy and thus, the profit-loss is lower; but, at the same time, the service quality is hurt. We make this discussion rigorous in the following theorem.
\begin{theorem} \label{theo: lower_bound}
Consider a DTMC operating under a pricing and matching policy $\pi \in \StablePol$ such that the following is true:
\begin{itemize}
    \item $|\lambda_j(\bq)-\tl_j^\star| \leq \epsilon$ for all $j \in [m]$, $\E{\alpha_\bq}{\mu_i}=\E{\talpha^\star}{\mu_i}$ for all $i \in [n]$ for all $\bq \in S$
    \item There exists $K,\sigma>0$ such that if $q_i^{(1)}>K$ or there exists $j \in [m]$ such that $(i,j) \in E$ and $q_j^{(2)}>K$, then $|\lambda_j(\bq)-\tl^\star_j|>\sigma\epsilon$ for all $j \in [m]$.
    \item $\E{}{\inner{\bone_{n+m}}{\bbarq}} = C$.
\end{itemize}
Then there exists $C_0, \epsilon_1>0$ such that for all $C>C_0$ and $\epsilon<\epsilon_1$, there exists some constant $B_3>0$ depending on $((\phi_j)_{j \in [m]},F(\cdot),c(\cdot),A_{\max},\Sigma^{(1)},\Sigma^{(2)},n,m,E)$ such that
\begin{align*}
    P(\pi) \leq \tilde{R}^\star-\frac{B_3}{C^2}.
\end{align*}
In addition, let $\pi$ be the pricing and matching policy given by \eqref{eq: two_price_policy} and \eqref{eq: modified_max_weight_matching}, and $\epsilon'=\epsilon_\eta$ such that $\E{}{\inner{\bone_{n+m}}{\bbarq}} = C$, then there exists a constant $B_4>B_3$ depending on $(F(\cdot),c(\cdot),A_{\max},\Sigma^{(1)},\Sigma^{(2)},n,m,E)$ such that
\begin{align*}
    P(\pi) \geq  \tilde{R}^*-\frac{B_4}{C^2}+O\left(\frac{1}{C^3}\right) 
%   \textit{where} \quad B_3&=-\sum_{j=1}^m \left(\frac{\tl^\star_jF''(\tl^\star_j)}{2}+F'_j(\tl_j^\star)\right)\left(\frac{B_1+2B_2}{2\min_{i \in [n]}\left\{\sum_{j : (i,j) \in E}\frac{\tchi_{ij}^\star}{\tl_j^\star}\right\}}\right)^2>0, \\
%   \textit{and} \quad B_4&=-\sum_{j=1}^m \left(\frac{\tl^\star_jF''(\tl^\star_j)}{2}+F'_j(\tl_j^\star)\right)\left(\frac{\bone_{n \times n}\circ \Sigma^{(1)}+\bone_{m \times m} \circ \Sigma^{(2)}-1}{4\max\{m,n\}}\right)^2>0.
\end{align*}
\end{theorem}
%\sv{Siva, please look at the proof of this as well}
Note that the first condition is analogous to the condition required for the validity of Theorem \ref{theo: lowerbound_queue}, the second condition is analogous to the stability condition given by \ref{condition: general_pricing} \ref{condition: stability}. The first part of the proposition follows from Theorem \ref{theo: lowerbound_queue} and Theorem \ref{theo: lowerbound_profit_loss} and the second part follows from Lemma \ref{lemma: stability} and Lemma \ref{lemma: profit_loss} with
$\epsilon_\eta=\epsilon'$. There are several key conclusions that can be drawn from the above proposition: (1) It implies that two-price policy and max-weight matching policy provides the optimal order of profit given the service level. (2) It also explains the $O(\eta^{1/3})$ loss in net profit obtained in the asymptotic regime by establishing a fundamental trade off between the profit and queue length. (3) Lastly, this result can be directly applied in practice to estimate the profit given the service level constraint.

%\fc{We should try to think how to get more out of this. Can we say something about how sensitive things are to changes in $C$? for example , the function $a*x +1/x^2$ is kind of flat at the optimum if $a$ is small. This means that there is some margin of error that does not affect things by a lot. Can we say something similar here but in terms of service level? check with Siva.}
\subsection{Waiting Time Model}
In this section, we modify the objective \eqref{eq:opt_stoch} and impose penalty based on the total expected waiting times rather than the queue lengths. In particular, let $w_j^{(2)}(k)$ be the waiting time of the $k^{th}$ customer of type $j$ and similarly,  $w_i^{(1)}(k)$ be the waiting time of the $k^{th}$ server of type $i$. Then, the objective of the system operator is given by
\begin{subequations}
\begin{align}
    R^\star_w\triangleq \sup_{(\blambda(\cdot),\alpha(\cdot),\bx(\cdot))\in \StablePol} \E{\bbarq}{\inner{F(\blambda(\bbarq))}{\blambda(\bbarq)}-\E{\alpha_\bbarq}{\cost(\bmu)}-\inner{\bs}{\bbarw}} \span \label{eq: net_profit_stochastic_waiting}\\
    \textit{subject to,} \quad &\blambda(\bq)\in \bbR^m_+,\quad \forall \bq\in \calS\\
    &\alpha_\bq(\cdot)\in \Peq_\bq \quad \forall \bq \in \calS \\
    &\bx(\cdot)\quad \textit{satisfies} \:\: \eqref{eq: matching_constraints}  
\end{align}
\end{subequations}
For any pricing and matching policy, it is trivially true that $R_w \leq P$, where $R_w=P-\E{\bbarq}{\inner{\bs}{\bbarw}}$. Thus, by Proposition \ref{prop: fluid_model}, we have $R_w \leq \tilde{R}^\star$ for any pricing and matching policy. Denote the net-profit loss under this model for the policy $\pi \in \StablePol$ by $L_w(\pi)=\tilde{R}^\star-R_w(\pi)$. Now, we show that the probabilistic two-price policy and max-weight matching policy is optimal by showing that the net-profit loss decays to zero as $\eta \rightarrow \infty$. The result is presented below.
\begin{proposition} \label{prop: waiting_time}
Consider a sequence of DTMCs parameterized by $\eta$ operating under any pricing policy satisfying Assumption \ref{condition: general_pricing} and any matching policy denoted by $\pi_\eta \in \StablePol$. Then, there exists a constant $K_w$ that depends on $(\phi_j)_{j \in [m]},\{F_j(\cdot)\}_{j \in [m]}$ and $\cost(\cdot)$, and $\eta_3(\beta)>0$ such that for all $\eta>\eta_3$, we have
\begin{align*}
    L_{\eta,w}(\pi_\eta) \geq K_w \eta^{-1/3}.
\end{align*} 
In addition, consider a sequence of DTMCs operating under the two price policy and max-weight matching policy. Then the net profit loss is $O(\eta^{-1/3})$ for the choice of $\epsilon^\eta=\eta^{-2/3}$.
\end{proposition}
The above proposition shows that two-price policy and max-weight matching policy achieves the optimal rate of convergence for the waiting time model. Also, note that the net profit-loss converges to 0 as $\eta \rightarrow \infty$ as opposed to the $O(\eta^{1/3})$ loss observed in the previous sections. The main reason for a lower net profit loss is the trade-off between the waiting time and profit loss. In particular, the expected waiting time is $1/\eta$ times the expected queue length by Little's Law and the definition of the asymptotic regime. Thus, the system can tolerate a larger expected queue length which allows the system operator to operate closer to the fluid solution which results in a higher profit.
\subsection{Generalizing the Utility Function}
\subsubsection{Model}
It is often the case in practice that the servers are aware of the probability with which they will be matched to a type of customer given the type of server. In particular, in steady state, an $i$ type of server is matched to a $j$ type of customer with rate $\E{}{\bary_{ij}}$ which is known to the servers. Motivated by this, we extend our model to incorporate a general utility function given by
\begin{align}
    u_{il}(k)=f_{il}\left(\bp^{(1)}(k),\E{}{\bbarx}\right) \quad \forall i,l \in [n] \label{eq: general_utility_function}
\end{align}
for any given continuous function $f_{il}$ for all $i,l \in [n]$. The expectation is with respect to the stationary distribution of the underlying Markov chain given the pricing and matching policy. Note that, utility function not only depends on the instantaneous actions but also depends on steady state quantities which leads to a convoluted dependence on the pricing and matching policy. This additional endogenity requires the pricing as well as the matching policy to be coherent to ensure equilibrium. After re-defining the utility function, the cost function is given similar to \eqref{eq: cost_function} with an additional dependence on $\E{}{\bbarx}$. We have \begin{align*}
     \cost\left(\bmu,\E{}{\bbarx}\right)\triangleq\min_{\bp^{(1)}} \inner{\bmu}{\bp^{(1)}} \quad \textit{subject to} \quad  \bp^{(1)} \in\mathcal{M}(\bmu, \E{}{\bbarx}) .
\end{align*}
Note that the set $\mathcal{M}(\cdot)$ is given by \eqref{eq: variational_inequality} but now we highlight the dependence of the utility function on $\E{}{\bbarx}$ by using the notation $\mathcal{M}(\bmu,\E{}{\bbarx})$. For a given $\E{}{\bbarx}=\bchi$, we denote the domain of the cost function by $\Omega(\bchi)$ and also define $\Omega=\cup_{\bchi \in \bbR_+^{n+m}} \Omega(\bchi)$.

Now, we will adopt the same framework here by first defining the probabilistic fluid model and then analyzing the perturbed stochastic policies. The probabilistic fluid model has been presented and analyzed in the Appendix \ref{app: prob_fluid_model_general}. In particular, we show that the solution of the probabilistic fluid model $(\btlambda^\star_e,\talpha^\star_e,\btchi^\star_e)$ provides an upper bound $\tilde{R}^\star_e$ on the achievable profit. Here, we will present asymptotically optimal stochastic policy.
\subsubsection{Optimality of Two Price and Random Matching Policy}
First, we will discuss the challenges that arises in the stochastic analysis by considering the max-weight matching policy \eqref{eq: modified_max_weight_matching}. From Lemma \ref{lemma: necessary_constraints_extended}, we know that under max-weight matching, the following will be satisfied:
\begin{align*}
    \E{}{\lambda_j(\bbarq)}&=\sum_{i=1}^n \E{}{\bary_{ij}} \ \forall j \in [m], \quad \E{}{\E{\alpha_{\bbarq_e}}{\mu_i}}=\sum_{j=1}^m \E{}{\bary_{ij}} \ \forall i \in [n] \\
    \E{}{\bary_{ij}}&=0 \ \forall (i,j) \notin E, \quad \E{}{\bary_{ij}} \geq 0 \ \forall (i,j) \in E.
\end{align*}
If multiple $\E{}{\bbary} \in \bbR_+^{n \times m}$ satisfies these set of equations, then, it is  difficult to characterize  $\E{}{\bbary}$ exactly or even approximately. This makes the task of verifying (let alone ensuring) if the system is in equilibrium difficult. We present the following simple example which shows that max-weight matching may not result in equilibrium.
\begin{example}
We consider a $2 \times 2$ complete graph with $F_1(\cdot)=F_2(\cdot)$, $G_1(\cdot)=G_2(\cdot)$, and utility function given by 
\begin{align*}
    u_{il}=p^{(1)}_l-K \E{}{\tilde{x}_{l \bar{l}}(\bbarq_e)} \quad \forall i,l \in \{1,2\},
\end{align*}
where $K>0$ is a constant, and $\bar{l}=3-l$. Now, if $K$ is large enough, the system operator is discouraged to match customer-server pairs using the edges $(1,2)$ and $(2,1)$ as otherwise, high prices must be offered to the drivers for them to join the system. This, along with the symmetry across the types of customers/servers, results in a fluid solution such that $\tchi_{12,e}^\star=\tchi_{21,e}^\star =0$. Although, applying max-weight matching will result in a non zero rate of matching using the edges $(1,2)$ and $(2,1)$ as intuitively, max-weight ensures that the queue lengths of both the types of customers are equal. This will result in poor utilities to the servers and thus, the system is not operating in an equilibrium. 
\end{example}
Thus, we need to modify the matching policy to ensure that the following two constraints are satisfied:
 1) $\E{}{\bbarx}$ and $\bp^{(1)}(\bq_e)$ are such that the system is operating under an equilibrium 2) The pricing and matching policy achieves $\eta^{1/3}$ loss in net-profit compared to the upper bound given by the probabilistic fluid model. To achieve optimality, from the intuition of Theorem \ref{theo: two_price_policy}, we need to operate close to the fluid solution. Thus, we need $\E{}{\bary_{ij}} \sim \tchi_{ij,e}^{*}$ for all $(i,j) \in E$ to ensure equilibrium. This motivates the introduction of random matching policy which routes the arrivals with probability proportional to the fluid solution $\btchi^{*}_e$. Although, this may lead to poor queueing performance and unstable system and we modify the pricing policy to ensure that doesn't happen. We will now formalize this intuition below.

 We extend our set of policies to be the set of stationary policies over an expanded state space $\bq_e \in S_e$ such that $\bq$ is completely determined by $\bq_e$. We define server pricing policy as the set of measures $\alpha_{\bq_e}(\cdot)$ over the Borel subsets of $\Omega \times S_e$ for all $\bq_e \in S_e$. In addition, we are only interested in the set of policies under which the Markov chain $\{\bq_e(k): k \in \bbZ_+\}$ is stable. We denote the set of such policies by $\StablePol_e$. The optimization problem \eqref{eq:opt_stoch} can now be extended to get the following:
\begin{subequations}
\begin{align}
    R^\star_e\triangleq \sup_{(\blambda(\cdot),\alpha(\cdot),\bx(\cdot))\in \StablePol_e} \E{\bbarq_e}{\inner{F(\blambda(\bbarq_e))}{\blambda(\bbarq_e)}-\E{\alpha_{\bbarq_e}}{\cost(\bmu, \E{}{\bbarx})}-\inner{\bs}{\bbarq}} \span \\
    \textit{subject to,} \quad &\blambda(\bq_e)\in \bbR^m_+,\quad \forall \bq_e\in S_e\\
    &\alpha_{\bq_e}((\Omega \times S_e) \backslash (\Omega(\E{}{\bbarx}) \times \{\bq_e\}))=0 \quad \forall \bq_e \in S_e \\
    &\bx(\cdot)\quad \textit{satisfies} \:\:  \eqref{eq: matching_constraints}
\end{align}
\label{eq:opt_stoch_extended}
\end{subequations}
\textbf{Pricing Policy:} We introduce secondary queues for each type of customers and servers. In particular, for a type $j$ customer, we introduce $q_{jr}^{(2)}$ for $r \in \left[|N(j)|\right]$ and for each type $i$ server, we introduce $q_{id}^{(1)}$ for $d \in \left[|N(i)|\right]$. Each secondary queue corresponds to one of its neighbours. The random matching policy will randomly route the incoming arrival immediately to one of these secondary queues. Then, the customers/servers wait in these secondary queues until they are matched. The compatibility graph between the secondary queues is one to one. That is, for a given $j$ and $r$, there exists a unique $i$ and $d$ such that $q_{jr}^{(2)}$ is the neighbor of $q_{id}^{(1)}$. In particular, denote the secondary graph by $G_e(N_{1,e} \cup N_{2,e},E_e)$. Then, with a slight abuse of notation, we have
\begin{align*}
    N_{1,e}=\left\{id : d \in [|N(i)|], i \in [n]\right\} \quad N_{2,e}=\left\{jr : r \in [|N(j)|], j \in [n]\right\} \\
    E_e=\left\{(id,jr) : (i,j) \in E, r=|\{N(j): N(j)\leq i\}|, d=|\{N(i) : N(i) \leq j\}| \right\}.
\end{align*}
In addition, denote the vector of all the secondary queue lengths by $\bq_e$ and it's state space by $S_e$. The pricing policy is a  two price policy corresponding to these secondary queues. In particular, the pricing policy is given by
\begin{subequations}
\label{eq: multiple_price_policy}
\begin{align}
    \lambda_{j,\eta}(\bq_e)&=\tl_{j,e}^\star-\epsilon_\eta \sum_{l=1}^{|N(j)|} \mathbbm{1}_{\left\{q_{jr}^{(2)}>0\right\}}+\epsilon_\eta \sum_{l=1}^{|N(j)|} \mathbbm{1}_{\left\{q_{jr}^{(2)}=0\right\}} \quad \forall j \in [m], \forall \bq_e \in S_e, \\
    \alpha_{\bq_e,\eta}&=\talpha^\star_e \quad \forall \bq_e \in S_e.
\end{align}
\end{subequations}
\textbf{Matching Policy:} The matching policy routes the arrivals to the secondary queue immediately such that the effective arrival rate to the secondary queues is given by
\begin{align}
    \left(\lambda_{jr,\eta}(\bq_e),\E{\alpha_{\bq_e,\eta}}{\mu_{id}}\right)&=\left(\tchi_{i j,e}^\star-\epsilon_\eta \mathbbm{1}_{\left\{q_{jr}^{(2)}>0\right\}}+\epsilon_\eta \mathbbm{1}_{\left\{q_{jr}^{(2)}=0\right\}},\tchi_{ij,e}^\star\right) \quad \forall i,j,d,r : (id,jr) \in E_e, \bq_e \in S_e. \label{eq: multiple_pricing_policy}
\end{align}
Note that, the graph formed by the secondary queues is operating under the two price policy and max-weight matching policy (trivially) and has a one to one compatibility structure with $|E|$ number of customer and server types. Thus, we can essentially apply Theorem \ref{theo: two_price_policy} to get $\eta^{1/3}$ loss in net profit. In addition, as the system is stable, by Lemma \ref{lemma: necessary_constraints} and the compatibility structure of the secondary queues, we will have $\E{}{\bary_{ij}}=\tchi_{ij,e}^\star$ for all $(i,j) \in E$. This will ensure equilibrium. We present this formally in the following corollary:
\begin{corollary} \label{corollary: optimality_nash_equilibrium}
Consider a sequence of DTMCs parametrized by $\eta$ operating under $\pi_\eta$ --- the pricing policy given by \eqref{eq: multiple_price_policy} and random matching policy. Then the the system is operating under an equilibrium such that the net profit loss $L_\eta(\pi_\eta)$ is $O(\eta^{1/3})$ for the choice of $\epsilon_\eta = \eta^{-1/3}$.
\end{corollary}
This is a strong result as this provides a stochastic pricing and matching policy which operates in an equilibrium governed by a general utility and cost function and achieves optimal rate of convergence to the optimum profit $\tilde{R}^\star_e$. %Although, it is known in the queueing literature that closed loop policies like max-weight performs much better than open loop policies like random matching in practice. In addition, we are not exploiting the network structure in the random matching policy. Motivated by this, we propose an improved random matching policy and show optimality in the Appendix \ref{app: improved_random_matching}. We believe that this policy achieves $\delta$ equilibrium in the sense that achieved utility of all the servers is at most $\delta$ away from the maximum possible utility and $\delta \rightarrow 0$ as $\eta \rightarrow \infty$. This is left as a part of future work.
\subsection{Pessimistic Equilibrium} \label{sec: pess_nash_eq}
In this section, we allow the servers to choose equilibrium of their choice among the ones that maximize their utility. By considering this extension, we relax the implicit assumption in defining the cost function in \eqref{eq: cost_function}. In particular, we analyze the worst case scenario by considering adversarial servers choosing the equilibrium that results in the worst net profit for the system operator.

We will formulate this as a min-max problem and analyze the fluid model. Further, stochastic analysis will follow similar to the original model and is left as a part of future investigation.
%Now, under this new model, we will analyze the fluid model. But first, we need to introduce some notations. 
To analyze such a model, we need to consider the set of equilibrium given the price $\bp^{(1)}$. We define the set of arrival rates that are consistent with $\bp^{(1)}$ by
   \begin{align*}
     \mathcal{N}_\rho(\bp^{(1)})\triangleq\left\{
     \bmu\in \bbR_+^n:\exists \boldsymbol{\nu}
     \in \bbR_+^{n\times n}\:\: \text{satisfying} \:\: \eqref{eq: variational_inequality},\: 
     G_i(\hatmu_i) = u_i,\: 
     \mu_i = \sum_{l=1}^n \hatmu_{l}\nu_{li}\:\: \forall i\in[n]
     \right\}. 
 \end{align*}
Next, we define the cost function as a function of both arrival rate and price as follows:
\begin{align}
    \cost_\rho(\bp^{(1)},\bmu)\triangleq \begin{cases}
    \inner{\bmu}{\bp^{(1)}}  &\textit{if} \quad  \bmu \in\mathcal{N}_\rho(\bp^{(1)}) \\
    \infty &\textit{otherwise}.
    \end{cases} \label{eq: cost_function_pessimistic}
\end{align}
As the servers can choose equilibrium of their choice, we define a probability measure $\zeta_{\bp^{(1)},\bq}$ on the set of arrival rates resulting in equilibrium given the price vector and the queue length. As $\zeta_{\bp^{(1)},\bq}$ is determined by the servers, we will later minimize the net-profit with respect to it. For technical reasons, we define $\zeta_{\bp^{(1)},\bq}$ on the common probability space --- Borel subsets of $\{(\bmu,\bp^{(1)}) : \bmu \in \mathcal{N}_\rho(\bp^{(1)}),\bp \in \bbR_+^n\} \times \calS$ and impose the following condition:
\begin{align}
    \zeta_{\bp^{(1)},\bq}\left(\{(\bmu,\bp^{(1)}) : \bmu \in \mathcal{N}_\rho(\bp^{(1)}),\bp \in \bbR_+^n\} \times \calS \backslash \mathcal{N}_\rho(\bp^{(1)}) \times \{\bp^{(1)}\} \times \{\bq\}\right)=0. \label{eq: marginal_servers}
\end{align}
Thus, we can interpret $\zeta_{\bp^{(1)},\bq}$ as a probability measure over $\mathcal{N}_\rho(\bp^{(1)})$ given $\bp^{(1)}$ and $\bq$. For each state $\bq \in \calS$, the system operator sets a price $\bp^{(1)}$ according to the probability distribution $\alpha_\bq$ and the servers react to it by picking a probability distribution $\zeta_{\bp^{(1)},\bq}$. Now, we define stability as follows:
%Also, define $\zeta(\bmu,\bp^{(1)},\bq)$ to be a joint probability measure over all the Borel subsets of $\{(\bmu,\bp^{(1)},\bq) : \bmu \in \mathcal{M}(\bp^{(1)}),\bp \in \bbR_+^n, \bq \in \calS\}$. Denote such a probability measure by $\mathcal{P}^{EQ}_\rho$. Also, let $\alpha(\bp^{(1)},\bq)$ be a probability measure over all the Borel subsets of $\bbR_+^n \times \calS$. For each state $\bq \in \calS$, the system operator sets a price $\bp^{(1)}$ according to the joint probability distribution $\alpha(\bp^{(1)},\bq)$ and the servers react to it by picking a joint probability distribution $\zeta(\bmu,\bp^{(1)},\bq)$. Now, we can re-define stability as follows: 
\begin{definition}[Stability] \label{defn: stability_pessimistic}
The discrete time Markov chain is stable if under a given pricing and matching policy $(\blambda(\cdot),\alpha(\cdot),\bx(\cdot))$, there exists $\{\zeta_{\bp^{(1)},\bq}\}_{\bp^{(1)} \in \bbR_+^n,\bq \in \calS}$ such that the communicating class containing the state $\bzero_{n+m}$ is positive recurrent and all the other states (if any) are transient.
We use $\StablePol_\rho$ to denote the set of stationary Markovian pricing and matching  policies that make the system stable.  
\end{definition} 
%\begin{definition}[Stability] \label{defn: stability_pessimistic}
% The discrete time Markov chain is stable if under a given pricing and matching policy $(\blambda(\cdot),\alpha(\cdot),\bx(\cdot))$, there exists $\zeta(\bmu, \bp^{(1)},\bq)$ such that the communicating class containing the state $\bzero_{n+m}$ is positive recurrent and all the other states (if any) are transient.
% We use $\StablePol_\rho$ to denote the set of stationary Markovian pricing and matching  policies that make the system stable.  
% \end{definition} 
In addition, given a policy $\pi \in \StablePol_\rho$, let $Z_\rho(\pi)$ denote the set of measures $\{\zeta_{\bp^{(1)},\bq}\}_{\bp^{(1)}\in \bbR_+^n,\bq \in \calS}$ satisfying \eqref{eq: marginal_servers} such that $(\blambda(\cdot), \alpha(\cdot), \bx(\cdot), \zeta(\cdot))$ makes the system stable. By the definition of stability, $Z_\rho(\pi) \neq \emptyset$. Now, the objective of the system operator can be re written as follows:
\begin{subequations}
\begin{align}
    R^\star_\rho\triangleq \sup_{\pi=(\blambda(\cdot),\alpha(\cdot),\bx(\cdot))\in \StablePol_p}\bigg\{ \inf_{\zeta \in Z_\rho(\pi)}  \E{\bbarq}{\inner{F(\blambda(\bbarq))}{\blambda(\bbarq)} 
    -\E{\alpha_\bbarq}{\E{\zeta_{\bp^{(1)},\bbarq}}{\cost(\bmu,\bp^{(1)})}}-\inner{\bs}{\bbarq}}\bigg\} \span \label{eq: net_profit_stochastic_pessimistic}\\
    \textit{subject to,} \quad &\blambda(\bq)\in \bbR^m_+\quad \forall \bq\in \calS\\
    &\bx(\cdot)\quad \textit{satisfies} \:\: \eqref{eq: matching_constraints}  \\
    &\zeta_{\bp^{(1)},\bq}(\cdot) \quad \textit{satisfies \eqref{eq: marginal_servers}} \ \forall \bp^{(1)} \in \bbR_+^n \ \forall \bq \in \calS \\
    &\alpha_\bq(\bbR_+^n \times \calS \backslash \bbR_+^n \times\{\bq\}) =0 \quad \forall \bq \in \calS.
\end{align}
\label{eq:opt_stoch_pessimistic}
\end{subequations}
The inner infimum over all the possible equilibrium under which the system is stable is to make sure that the servers picks the set of arrival rates which minimizes the profit of the system operator. The outer supremum is the system operator picking a pricing policy which maximizes it's profit. We formulate the probabilistic fluid model of this formulation in the Appendix \ref{app: pess_nash_eq} and show that it provides an upper bound on the achievable profit under any pricing and matching policy. The stochastic analysis follows similarly and is omitted for brevity.
%%%%%%%%%%%%%%%%%%%%%%%%%%%%%%%%%%%%%%%%%%%%%%%%%%%%%%%%%%%%%%%
\section{Probabilistic Fluid Model}
\subsection{A Note on Server Pricing Policy}
A Server pricing policy is given by the set of measure $\alpha_\bq(\cdot)$ for all $\bq \in \calS$. This can also be interpreted as follows: the system operator chooses a joint measure $\alpha(\cdot)$ such that the marginal measure w.r.t. $\bq$ is consistent with the stationary distribution. By consistency, we mean the following holds:
\begin{align}
   \E{\alpha}{f(\bmu,\bq)}=\E{\bbarq}{\E{\alpha_\bbarq}{f(\bmu,\bbarq)}} \quad \forall f \textit{ Borel measurable} \label{eq: joint}
\end{align}
With this alternate formulation, the new decision variables is a joint measure $\alpha(\cdot)$ over $\Omega \times \calS$. This can be formalized as follows:
\begin{proposition}
\label{prop: equivalence_joint_marginal}
\begin{align*}
    &\left\{\alpha: \alpha(\Omega \times B)=\sum_{\bq \in B} \psi(\bq) \ \forall B \subseteq \calS\right\}=\left\{\alpha: \exists \alpha_\bq \ \forall  \{\bq \in \calS: \psi(\bq)>0\}, f \textit{ Borel measurable s.t. }\right.\\
    & \left.\E{\alpha}{f(\bmu,\bq)}=\E{\bbarq}{\E{\alpha_\bbarq}{f(\bmu,\bbarq)}}, \alpha_\bq(C)=\alpha_\bq(C \cap \sigma^{-1}(\bq)) \ \forall C \subseteq \Omega \times \calS, \right\},
\end{align*}
where $\sigma_p: \Omega \times \calS \rightarrow \calS$ is the projection function given by $\sigma_p(\bmu,\bq)=\bq$.
\end{proposition}
 \begin{proof}{Proof}
 First we will show that given the first condition, the second condition is satisfied.
Note that for any $\bq \subseteq \calS$, we have
\begin{align*}
    \alpha( \sigma_p^{-1}(\bq))=\alpha(\Omega \times \{\bq\})=\psi(\bq).
\end{align*}
Thus, the push-forward measure $\alpha \circ \sigma_p^{-1}$ is the stationary distribution. Now, by applying the Disintegration theorem, for all $\bq$ such that $\psi(\bq)>0$, there exists $\alpha_\bq$ such that $\alpha_\bq(C)=\alpha_\bq(C \cap \sigma^{-1}(\bq))$ and \eqref{eq: joint} is satisfied. Now, given the second condition, the first condition can be shown to be satisfied by considering $f=\mathbbm{1}\{\Omega \times B\}$ for a Borel measurable set $B \subseteq \calS$.
\begin{align}
    \alpha(\Omega \times B)&= \sum_{\bq \in \calS} \alpha_\bq(\Omega \times B) \psi(\bq) \nonumber \\
    &= \sum_{\bq \in \calS: \psi(\bq)>0} \alpha_\bq(\Omega \times B) \psi(\bq) \nonumber \\
    &=\sum_{\bq \in \calS: \psi(\bq)>0} \alpha_\bq((\Omega \times B) \cap (\Omega \times \{\bq\})) \psi(\bq) \nonumber \\
    &=\sum_{\bq \in \calS: \psi(\bq)>0} \alpha_\bq(\Omega \times \{\bq\}) \psi(\bq) \mathbbm{1}_{\bq \in B} \nonumber \\
    &=\sum_{\bq \in \calS: \psi(\bq)>0} \alpha_\bq(\Omega \times \calS) \psi(\bq) \mathbbm{1}_{\bq \in B} \nonumber\\
    &=\sum_{\bq \in B} \psi(\bq) \quad \forall B \subseteq \calS. \nonumber \Halmos
\end{align}
 \end{proof}
To gain more intuition, let $\alpha(\cdot)$ be a discrete measure with a probability mass function. Then, the disintegration is trivial.
\begin{align*}
    \alpha(\bmu,\bq)=\frac{\alpha(\bmu,\bq)}{\psi(\bq)}\psi(\bq) \overset{\Delta}{=} \alpha_\bq(\bmu,\bq)\psi(\bq) \quad \forall \bmu \in \Omega, \bq \in \calS.
\end{align*}
For each $\bq$, the system operator chooses distribution $\alpha_\bq$ which is only non-zero on $\Omega \times \{\bq\}$. Thus, it can be thought of as choosing a pmf over $\Omega$ for all $\bq \in \calS$. Thus, for each $\alpha$, there exists a set of measures $\alpha_\bq$ and for each set of measures $\alpha_\bq$, there exists a joint measure $\alpha$ such that they are consistent.

Using the above proposition and defining $\alpha$ as a joint distribution over $\Omega \times \calS$, the optimization problem \ref{eq:opt_stoch} can be equivalently re-written as follows:
\begin{subequations}
\begin{align}
    R^\star\triangleq \sup_{(\blambda(\cdot),\alpha(\cdot),\bx(\cdot))\in \StablePol} \E{\bbarq}{\inner{F(\blambda(\bbarq))}{\blambda(\bbarq)}}-\E{\alpha}{\cost(\bmu)}-\E{\bbarq}{\inner{\bs}{\bbarq}} \span \\
    \textit{subject to,} \quad &\blambda(\bq)\in \bbR^m_+,\quad \forall \bq\in \calS\\
    &\bx(\cdot)\quad \textit{satisfies} \:\: \\
     &\alpha(\Omega \times B)=\sum_{\bbarq \in B} \psi(\bbarq) \quad \forall B \subseteq \calS  \\
     &\psi(\cdot)=\textit{Stationary Distribution} (\blambda(\cdot),\alpha(\cdot),\bx(\cdot)).
\end{align}
\label{eq:opt_stoch_equivalent}
\end{subequations}
In the further sections, we will work this alternate formulation.
\subsection{Proof of Lemma \ref{lemma: necessary_constraints}}
\proof{Proof}
As the DTMC is assumed to be stable, there exists a unique stationary distribution and we denote it by $\psi$.
By the hypothesis of the Lemma, we have $\E{\bbarq}{\inner{\bone_{n+m}}{\bbarq}}<\infty$. Thus, in steady state, we have $
    \E{}{\bbarq}=\E{}{\bbarq^+} \Rightarrow \E{}{\bbara}=\E{}{\bbarx}$, where  we denote the queue length one time slot after $\bbarq$ by $\bbarq^+=\bbarq+\bbara-\bbarx$.
Now, we will simplify the RHS and LHS separately. We have
\begin{align}
    \E{}{\bbara}&=\E{}{\E{}{\bbara|\bbarq}}=(\E{\bbarq}{\blambda},\E{\alpha}{\bmu})=(\btlambda,\E{\alpha}{\bmu})=(\btlambda,\E{\talpha}{\bmu}), \label{eq: prop_lhs}
\end{align}
where the third equality follows as $\tl_j=\E{}{\lambda_j}$ for all $j \in [m]$ and the last equality follows as the expectation of $\bmu$ over the joint distribution $\alpha$ of $(\bmu,\bbarq)$ is same as the expectation over the marginal distribution $\talpha$. The marginal distribution is given by
\begin{align*}
    \talpha(A)=\alpha(A \times \calS)=\sum_{\bq \in \calS} \alpha_\bq(A \times \calS)\psi(\bq)=\sum_{\bq \in \calS} \alpha_\bq(A \times \{\bq\})\psi(\bq) \quad \forall A \subseteq \Omega, \textit{ Borel}.
\end{align*}
Now, we will simplify the right hand side. First define $\E{}{\bary_{ij}}=\tchi_{ij}$ for all $(i,j)$. By \eqref{eq: matching_constraints}(c), we have $\tchi_{ij}=0$ for all $(i,j) \notin E$. Next, we have
\begin{align}
    \E{}{x_i^{(1)}(\bbarq)}&=\sum_{j=1}^m \E{}{\bary_{ij}}=\sum_{j=1}^m \tchi_{ij} \ \forall i \in [n], \quad
    \E{}{x_j^{(2)}(\bbarq)}=\sum_{i=1}^n \E{}{\bary_{ij}}=\sum_{i=1}^n \tchi_{ij} \ \forall j \in [m]. \label{eq: prop_rhs}
\end{align}
Now, simplifying $\E{}{\bbara}=\E{}{\bbarx}$ using \eqref{eq: prop_lhs} and \eqref{eq: prop_rhs}, we get the constraints of the optimization problem \eqref{eq: prob_fluid_model}. \hfill $\Halmos$
\endproof
\subsection{Proof of Proposition \ref{prop: fluid_model}}
\proof{Proof}
It suffices to consider only the set of pricing and matching policies under which $\E{}{\inner{\bone_{n+m}}{\bbarq}}<\infty$, as, otherwise, the net profit $R$ will be $-\infty$. We will consider the class of stationary Markovian policies. Note that, by Lemma \ref{lemma: necessary_constraints}, the constraints of the fluid problem \eqref{eq: prob_fluid_model}, are necessary constraints for stability. Now, under a given pricing and matching policy, we will upper bound the maximum profit and net profit. We have
\begin{align*}
    \E{}{\inner{F(\blambda(\bbarq))}{\blambda(\bbarq)}}-\E{\alpha}{\cost(\bmu)}-\E{}{\inner{\bs}{\bbarq}} &\leq  \E{}{\inner{F(\blambda(\bbarq))}{\blambda(\bbarq)}}-\E{\alpha}{\cost(\bmu)} \\
    &\leq \inner{F(\btlambda)}{\btlambda}-\E{\alpha}{\cost(\bmu)} \\
    &=\inner{F(\btlambda)}{\btlambda}-\E{\talpha}{\cost(\bmu)},
\end{align*}
where $\talpha$ is the marginal distribution of $\alpha$. Thus, the net profit obtained under any stationary pricing and matching policy is upper bounded by the solution of the probabilistic fluid model. This completes the proof. \hfill $\Halmos$
\endproof
\subsection{Proof of Proposition \ref{prop: convexity_LP_reduction}}
\proof{Proof}
We will first show that $\tilde{R}^\star_{co} \geq \tilde{R}^\star$. For a given $\talpha$, let us start by defining $\tilde{\tilde{\bmu}}\dfn\E{\talpha}{\btmu}$. Now, the objective function of \eqref{eq: prob_fluid_model} can be upper bounded by Jensen's inequality to get
\begin{align*}
    \inner{F(\btlambda)}{\btlambda}-\E{\talpha}{c(\btmu)} \leq \inner{F(\btlambda)}{\btlambda}-c\left(\E{\talpha}{\btmu}\right)=\inner{F(\btlambda)}{\btlambda}-c(\tilde{\tilde{\bmu}}).
\end{align*}
Thus, we have 
\begin{align*}
   \span \tilde{R}^\star \leq \max_{\btlambda,\talpha,\btchi} \inner{F(\btlambda)}{\btlambda}-c(\tilde{\tilde{\bmu}}) \\
     \textit{subject to,} \quad \tl_j&=\sum_{i=1}^n \tchi_{ij} \quad \forall j \in [m] \\
    \tilde{\tilde{\mu_i}}&=\sum_{j=1}^m \tchi_{ij} \quad \forall i \in [n]  \\
    \tchi_{ij}&=0 \quad \forall (i,j) \notin E, \quad \tchi_{ij} \geq 0 \quad \forall (i,j) \in E,
\end{align*}
Note that we can replace $\talpha$ by $\tilde{\tilde{\bmu}}$ in the arguments in terms of which we are maximizing as the objective function and constraints only depend on $\talpha$ through $\tilde{\tilde{\bmu}}$. Thus, by \eqref{eq: convex_prob_fluid_model}, we get $\tilde{R}^\star \leq \tilde{R}^\star_{co}$. Now, we will show the opposite inequality. Let the optimal solution of \eqref{eq: convex_prob_fluid_model} be $(\btlambda^\star,\tilde{\tilde{\bmu^\star}},\btchi^\star)$. Note that $(\btlambda^\star,\talpha^\star,\btchi^\star)$ is a feasible solution for \eqref{eq: prob_fluid_model}, with $\talpha^\star(\tilde{\tilde{\bmu^\star}})=1$. Under this feasible solution, the objective function value of \eqref{eq: prob_fluid_model} is $\tilde{R}^\star_{co}$. Thus, we have $\tilde{R}^\star \geq \tilde{R}_{co}^\star$. This completes the proof. \hfill $\Halmos$
\endproof
 \section{Proof of Theorem \ref{theo: two_price_policy}: Asymptotic Optimality}
\subsection{Proof of Lemma \ref{lemma: stability}}
\proof{Proof}
For all $\eta>0$, consider the Lyapunov functions $V(\bq_\eta)=\inner{\bone_{n+m}}{\bq_\eta^2}$ for all $\eta$.
We will calculate the drift of this Lyapunov function and show that it is negative outside a finite set. We have
\begin{align*}
    \E{}{\Delta V(\bq_\eta(k))|\bq_\eta(k)}&=\E{}{\inner{\bone_{n+m}}{\bq_\eta(k+1)^2}-\inner{\bone_{n+m}}{\bq_\eta(k)^2}|\bq_\eta(k)} \\
    &=\E{}{\inner{\bone_{n+m}}{\left(\bq_\eta(k)+\ba_\eta(k)-\bx_\eta(k)\right)^2}-\inner{\bone_{n+m}}{\bq_\eta(k)^2}|\bq_\eta(k)} \\
    &=\underbrace{\E{}{\inner{\bone_{n+m}}{\left(\ba_\eta(k)-\bx_\eta(k)\right)^2}|\bq_\eta(k)}}_{T_1}+2\underbrace{\E{}{\inner{\bq_\eta(k)}{\ba_\eta(k)-\bx_\eta(k)}|\bq_\eta(k)}}_{T_2}.
\end{align*}
Now, we will simplify $T_1$ separately. First note that 
\begin{align*}
    0 \leq \inner{\bone_{n+m}}{\bx_\eta(k)} \leq 2\inner{\bone_{n+m}}{\ba_\eta(k)} \quad \textit{w.p.} \ 1,
\end{align*}
as the matching policy is defined such that, at the beginning of each period, there won't be any customer-server compatible pairs waiting in the system. Thus, the maximum possible pairs that can be matched in one time epoch is the total number of arrivals. So, we have
\begin{align*}
    T_1&=\E{}{\inner{\bone_{n+m}}{\left(\ba_\eta(k)-\bx_\eta(k)\right)^2}|\bq_\eta(k)} \\
    &\leq \E{}{\inner{\bone_{n+m}}{\ba_\eta(k)^2}|\bq_\eta(k)}+\E{}{\inner{\bone_{n+m}}{\bx_\eta(k)^2}|\bq_\eta(k)} \\
    &=\sum_{i=1}^n \left(\Sigma_{i,i}^{(1)}(\alpha(\bq_\eta))+\E{\alpha_{\eta}(\bq_\eta)}{\mu_i}^2\right)+\sum_{j=1}^m \left(\Sigma_{j,j}^{(2)}(\blambda(\bq_\eta))+(\lambda_{i,\eta}(\bq_\eta))^2\right)+\E{}{\inner{\bone_{n+m}}{\bx_\eta(k)^2}|\bq_\eta(k)} \\
    &\leq  \sum_{i=1}^n \left(\Sigma_{i,i}^{(1)}(\talpha^\star)+\E{\talpha^\star}{\mu_i}^2\right)+\sum_{j=1}^m \left((\Sigma_{\max}^{(2)})_{j,j}+(\tl_i^\star+1)^2\right)+\E{}{\inner{\bone_{n+m}}{\bx_\eta(k)}^2|\bq_\eta(k)} \\
    &\leq  \sum_{i=1}^n \left(\Sigma_{i,i}^{(1)}(\talpha^\star)+\E{\talpha^\star}{\mu_i}^2\right)+\sum_{j=1}^m \left((\Sigma_{\max}^{(2)})_{j,j}+(\tl_i^\star+1)^2\right)+4\E{}{\inner{\bone_{n+m}}{\ba_\eta(k)}^2|\bq_\eta(k)} \\
    &\leq \sum_{i=1}^n \left(\Sigma_{i,i}^{(1)}(\talpha^\star)+\E{\talpha^\star}{\mu_i}^2\right)+\sum_{j=1}^m \left((\Sigma_{\max}^{(2)})_{j,j}+(\tl_i^\star+1)^2\right)+4(n+m)^2A_{\max}^2 \\
    &\dfn B_1.
\end{align*}
Now, we will simplify $T_2$ below.
\begin{align*}
    T_2&=\E{}{\inner{\bq_\eta(k)}{\ba_\eta(k)-\bx_\eta(k)}|\bq_\eta(k)} \\
    &=\E{}{\inner{\bq_\eta(k)}{\ba_\eta(k)}|\bq_\eta(k)}-\E{}{\inner{\bx_\eta(k)}{\bq_\eta(k)}|\bq_\eta(k)} \\
    &=\inner{\bq_\eta(k)}{\E{}{\ba_\eta(k)|\bq_\eta(k)}}-\E{}{\max_{\by \in \eqref{eq: matching_constraints}}\inner{\by}{\bq_\eta(k)}|\bq_\eta(k)} \\
    &=\inner{\bq_{\eta}^{(1)}(k)}{\E{\talpha^\star}{\bmu}}+\inner{\bq_{\eta}^{(2)}(k)}{\blambda_\eta(\bq_\eta(k))}-\E{}{\max_{\by \in \eqref{eq: matching_constraints}}\inner{\by}{\bq_\eta(k)}|\bq_\eta(k)} \\
    &=-\epsilon_{\eta}\inner{\bone_m}{\bq_{\eta}^{(2)}(k)}+\inner{\bq_{\eta}^{(1)}(k)}{\E{\talpha^\star}{\bmu}}+\inner{\bq_{\eta}^{(2)}(k)}{\btlambda^\star}-\E{}{\max_{\by \in \eqref{eq: matching_constraints}}\inner{\by}{\bq_\eta(k)}|\bq_\eta(k)} \\
    &=-\epsilon_{\eta}\inner{\bone_m}{\bq_{\eta}^{(2)}(k)}+\sum_{(i,j) \in E} \tchi_{ij}^\star(q_{i,\eta}^{(1)}(k)+q_{j,\eta}^{(2)}(k))-\E{}{\max_{\by \in \eqref{eq: matching_constraints}}\inner{\by}{\bq_\eta(k)}|\bq_\eta(k)}
\end{align*}
\begin{lemma} \label{lemma: negative_drift}  For all $\eta>0$, we have
\begin{align*}
   &\sum_{(i,j) \in E} \tchi_{ij}^\star(q_{i,\eta}^{(1)}(k)+q_{j,\eta}^{(2)}(k))-\E{}{\max_{\by \in \eqref{eq: matching_constraints}}\inner{\by}{\bq_\eta(k)}|\bq_\eta(k)} \leq -\epsilon_\eta \min_{i \in [n]}\left\{\sum_{j : (i,j) \in E}\frac{\tchi_{ij}^\star}{\tl_j^\star}\right\}\inner{\bone_n}{\bq^{(1)}_\eta(k)}+B_2, \\
   &\textit{where} \quad B_2=A_{\max}\sum_{(i,j) \in E} \tchi_{ij}^\star (\E{\talpha^\star}{\mu_i}+\tl_j^\star)+ A_{\max} \sum_{i=1}^n \left(\sum_{j : (i,j) \in E}\frac{\tchi^\star_{ij}}{\tl_j^\star}\E{\talpha^\star}{\mu_i}\right).
\end{align*}
\end{lemma}
The proof of the Lemma \ref{lemma: negative_drift} is provided at the end of the proof of the Lemma \ref{lemma: stability}. Now, using the Lemma \ref{lemma: negative_drift}, we have
\begin{align*}
    T_2 &\leq -\epsilon_{\eta}\inner{\bone_m}{\bq_{\eta}^{(2)}(k)}-\epsilon_\eta\min_{i \in [n]}\left\{\sum_{j : (i,j) \in E}\frac{\tchi_{ij}^\star}{\tl_j^\star}\right\}\inner{\bone_n}{\bq^{(1)}_\eta(k)}+B_2 \\
    &\leq -\epsilon_\eta\min_{i \in [n]}\left\{\sum_{j : (i,j) \in E}\frac{\tchi_{ij}^\star}{\tl_j^\star},1\right\}\inner{\bone_{n+m}}{\bq_\eta(k)}+B_2.
\end{align*}
Thus, we have
\begin{align*}
    \E{}{\Delta V(\bq_\eta)|\bq_\eta(k)}\leq B_1+2B_2-2\epsilon_\eta\min_{i \in [n]}\left\{\sum_{j : (i,j) \in E}\frac{\tchi_{ij}^\star}{\tl_j^\star},1\right\}\inner{\bone_{n+m}}{\bq_\eta(k)}.
\end{align*}
So, there exists a finite set $\calB_\eta$ such that for all $q \notin \calB_\eta$, we have $\E{}{\Delta V(\bq_\eta)} <-\delta$ for $\delta>0$ where $\calB_\eta$ is defined as:
\begin{align*}
    \calB_\eta=\left\{\bq_\eta \in \bbZ_+^{n+m}: \inner{\bone_{n+m}}{\bq_\eta(k)} \leq \frac{B_1+2B_2}{\epsilon_{\eta}\min_{i \in [n]}\left\{\sum_{j : (i,j) \in E}\frac{\tchi_{ij}^\star}{\tl_j^\star},1\right\}}\right\}.
\end{align*}
Thus, by the Foster-Lyapunov Theorem, the discrete time Markov chain for all $\eta>0$ is positive recurrent. Now, we can use the moment bound theorem, to get a bound on the expected queue length in steady state. We have
\begin{align*}
   \E{}{\inner{\bs}{\bbarq_\eta}}\leq ||\bs||_\infty \E{}{\inner{\bone_{n+m}}{\bbarq_\eta}}\leq \frac{B}{\epsilon_\eta} \dfn \frac{(B_1+2B_2)\|\bs\|_\infty}{2\epsilon_{\eta}\min_{i \in [n]}\left\{\sum_{j : (i,j) \in E}\frac{\tchi_{ij}^\star}{\tl_j^\star},1\right\}}. \Halmos
\end{align*}
 \endproof
\proof{Proof of Lemma \ref{lemma: negative_drift}}
In this proof, we will omit the $k$ and $\eta$ dependence and write $\bq$, $\ba$ and $\bx$ for $\bq_\eta(k)$, $\ba_\eta(k)$ and $\bx_\eta(k)$ respectively, for the ease of notation. The max-weight matching policy can be re-written as follows:
\begin{subequations}
\begin{align}
    \by=\arg\max_{\bv} \sum_{(i,j)\in E} v_{ij}(q_i^{(1)}+q_j^{(2)}) \span \\
    \textit{subject to } q_j^{(2)}+a_j^{(2)}-\sum_{i : (i,j)\in E} v_{ij} &\geq 0 \quad \forall j \in [m] \\
    q_i^{(1)}+a_i^{(1)}-\sum_{j: (i,j) \in E} v_{ij} &\geq 0 \quad \forall i \in [n] \\
    v_{ij} &\geq 0 \quad \forall (i,j) \in E.
\end{align}
\label{eq: max_weight_opt}
\end{subequations}
Note that $\bv$ can be relaxed to be a continuous variable. In fact, the constraint set we have is a polyhedron of the form $\{\bA \by \leq \bb\}$ where $\bb \in \bbZ_+^{n+m}$ and $\bA \in \{0,1\}^{(m+n) \times mn}$. 
Note that $A$ is the incidence matrix of the bipartite graph $G(N_1 \cup N_2, E)$ and thus, the polyhedron is integral.
% Now, note that any subset of rows $A$ admits an equitable row-bicoloring by a simple partition of the row of $A$ into the constraints corresponding to \eqref{eq: column_sums} and \eqref{eq: row_sums}. Thus, by \citep{}, $A$ is totally uni-modular and thus by \citep{} all extreme points of the resulting polyhedron are integral.
Note that for a given $\bq$ and $\ba$, a feasible solution to the above problem is 
\begin{align*}
    \hy_{ij}=\min\left\{(a_i^{(1)}+q_i^{(1)})\frac{\tchi^\star_{ij}}{\E{\talpha^\star}{\mu_i}}, (a_j^{(2)}+q_j^{(2)})\frac{\tchi^\star_{ij}}{\tl_j^\star}\right\} \quad \forall (i,j) \in E.
\end{align*}
This can be easily verified as follows:
\begin{align*}
    \sum_{i: (i,j) \in E} \hy_{ij} &\leq (a_j^{(2)}+q_j^{(2)})\frac{\sum_{i : (i,j) \in E} \tchi_{ij}^\star}{\tl_j^\star}=a_j^{(2)}+q_j^{(2)} \quad \forall j \in [m] \\
     \sum_{j: (i,j) \in E} \hy_{ij} &\leq (a_i^{(1)}+q_i^{(1)})\frac{\sum_{j : (i,j) \in E} \tchi_{ij}^\star}{\E{\talpha^\star}{\mu_i}}=a_i^{(1)}+q_i^{(1)} \quad \forall i \in [n].
\end{align*}
We will use this feasible solution to lower bound the objective function of \eqref{eq: max_weight_opt}. But before, observe that as we are using max-weight matching policy, we have for all $(i,j) \in E$ if $q_i^{(1)}>0$ then $q_j^{(2)}=0$. In other words, at the start of each epoch, there are no compatible pairs waiting to be matched. So, we have
\begin{align*}
    \lefteqn{\E{}{\hy_{ij}|q_i^{(1)}>A_{\max}\E{\talpha^\star}{\mu_i}}}\\
    &=\E{}{\min\left\{(a_i^{(1)}+q_i^{(1)})\frac{\tchi^\star_{ij}}{\E{\talpha^\star}{\mu_i}}, (a_j^{(2)}+q_j^{(2)})\frac{\tchi^\star_{ij}}{\tl_j^\star}\right\}\bigg|q_i^{(1)}>A_{\max}\E{\talpha^\star}{\mu_i}} \\
    &=\E{}{(a_j^{(2)}+q_j^{(2)})\frac{\tchi^\star_{ij}}{\tl_j^\star}\bigg|q_j^{(2)}=0} \\
    &=\E{}{a_j^{(2)}\frac{\tchi^\star_{ij}}{\tl_j^\star}\bigg|q_j^{(2)}=0}= \tchi_{ij}^\star+\epsilon_\eta\frac{\tchi^\star_{ij}}{\tl_j^\star}.
\end{align*}
Similarly, note that
\begin{align*}
   \lefteqn{ \E{}{\hy_{ij}|q_j^{(2)}>A_{\max}\tl_j^\star}}\\
   & =\E{}{\min\left\{(a_i^{(1)}+q_i^{(1)})\frac{\tchi^\star_{ij}}{\E{\talpha^\star}{\mu_i}}, (a_j^{(2)}+q_j^{(2)})\frac{\tchi^\star_{ij}}{\tl_j^\star}\right\}\bigg|q_j^{(2)}>A_{\max}\tl_j^\star} \\
    &=\E{}{(a_i^{(1)}+q_i^{(1)})\frac{\tchi^\star_{ij}}{\E{\talpha^\star}{\mu_i}}\bigg| q_i^{(1)}=0} \\
    &=\E{}{a_i^{(1)}\frac{\tchi^\star_{ij}}{\E{\talpha^\star}{\mu_i}}\bigg| q_i^{(1)}=0}= \tchi_{ij}^\star.
\end{align*}
Now, we can lower bound the objective function as follows:
\begin{align*}
 \lefteqn{\E{}{\max_{\bv \in \eqref{eq: matching_constraints}}\inner{\bv}{\bq}|\bq} \overset{(a)}{\geq} \E{}{\inner{\bhy}{\bq}|\bq}=\sum_{(i,j) \in E} \E{}{\hy_{ij}(q_i^{(1)}+q_j^{(2)})| \bq}} \\
   \overset{(b)}{=}{}&\sum_{(i,j) \in E} \E{}{\hy_{ij}(q_i^{(1)}+q_j^{(2)})| q_i^{(1)},q_j^{(2)}} \\
   \overset{(c)}{\geq}{}& \sum_{(i,j) \in E} \E{}{\hy_{ij}(q_i^{(1)}+q_j^{(2)})\mathbbm{1}_{\left\{q_i^{(1)}>A_{\max}\E{\talpha^\star}{\mu_i}\right\}}\bigg|q_i^{(1)}, q_j^{(2)}}\\
   &+\sum_{(i,j) \in E} \E{}{\hy_{ij}(q_i^{(1)}+q_j^{(2)})\mathbbm{1}_{\left\{q_j^{(2)}>A_{\max}\tl_j^\star\right\}}\bigg|q_i^{(1)}, q_j^{(2)}} \\
   \overset{(d)}{=}{}& \sum_{(i,j) \in E} \tchi^\star_{ij}\left(q_i^{(1)}\mathbbm{1}_{\left\{q_i^{(1)}>A_{\max}\E{\talpha^\star}{\mu_i}\right\}}+q_j^{(2)}\mathbbm{1}_{\left\{q_j^{(2)}>A_{\max}\tl^\star_j\right\}}\right)+\epsilon_\eta \sum_{(i,j) \in E} \frac{\tchi^\star_{ij}}{\tl_j^\star}q_i^{(1)}\mathbbm{1}_{\left\{q_i^{(1)}>A_{\max}\E{\talpha^\star}{\mu_i}\right\}} \\
   \geq{}& \sum_{(i,j) \in E}\tchi^\star_{ij}(q_i^{(1)}+q_j^{(2)})+\epsilon_\eta \min_{i \in [n]}\left\{\sum_{j : (i,j) \in E}\frac{\tchi_{ij}^\star}{\tl_j^\star}\right\}\inner{\bone_n}{\bq^{(1)}}-B_2,
\end{align*}
where 
\begin{align*}
    B_2=A_{\max}\sum_{(i,j) \in E} \tchi_{ij}^\star (\E{\talpha^\star}{\mu_i}+\tl_j^\star)+A_{\max} \sum_{i=1}^n \left(\sum_{j : (i,j) \in E}\frac{\tchi^\star_{ij}}{\tl_j^\star}\E{\talpha^\star}{\mu_i}\right).
\end{align*}
The inequality $(a)$ follows as $\hy$ is a feasible solution to the optimization problem. Next, $(b)$ follows as the feasible solution $\hy_{ij}$ only depends on $q_i^{(1)}$ and $q_j^{(2)}$. The inequality $(c)$ follows as for all $(i,j) \in E$, only one of $q_i^{(1)}$ and $q_j^{(2)}$ can be non zero. Finally, $(d)$ follows due to the following equation.
\begin{align*}
    \E{}{\hy_{ij}(q_i^{(1)}+q_j^{(2)})\mathbbm{1}_{\left\{q_j^{(2)}>A_{\max}\tl_j^\star\right\}}\bigg|q_i^{(1)}, q_j^{(2)}}&\overset{*}{=}\E{}{\hy_{ij}|q_i^{(1)},q_j^{(2)}}q_j^{(2)}\mathbbm{1}_{\left\{q_j^{(2)}>A_{\max}\tl_j^\star\right\}} \\
    &=\tchi^\star_{ij}q_j^{(2)}\mathbbm{1}_{\left\{q_j^{(2)}>A_{\max}\tl_j^\star\right\}},
\end{align*}
where $(*)$ is true as the max-weight matching policy makes sure there are no compatible pairs waiting in the system at the start of a time epoch. \hfill $\Halmos$
 \endproof
\subsection{Proof of Lemma \ref{lemma: first_order_optimality}}
We will prove the more general lemma give below and then use this lemma to prove the Lemma \ref{lemma: first_order_optimality}.
\begin{lemma} \label{lemma: first_order}
For the server pricing policy given by $\alpha_\bq=\talpha^\star$ for all $\bq \in \calS$ and any given pricing policy for customers and any matching policy under which the system is stable and $\E{}{\inner{\bone_{n+m}}{\bbarq}}<\infty$ the following holds:
\begin{align*}
    \sum_{j=1}^m \left(\tl_j^\star F'_j(\tl_j^\star)+F_j(\tl_j^\star)\right)\left(\E{}{\lambda_j(\bbarq)}-\tl_j^\star\right)=0.
\end{align*}
\end{lemma}
\proof{Proof of Lemma \ref{lemma: first_order}}
Firstly, we will define a matrix to vector operation by stacking columns on top of each other. For the matrix $\btchi$, we will denote the corresponding vector by $\bhchi$. We define it as follows:
\begin{align*}
    \hchi_k \dfn \tchi_{ij} \quad \textit{where, } i=k \% (n+1), \ j=\lceil\frac{k}{n}\rceil \quad \forall k \in [nm].
\end{align*}
Here $k\%n$ denotes the reminder obtained when $k$ is divided by $n$ and $\lceil . \rceil$ is the ceiling function which returns the smallest integer greater than or equal to the argument.

Add a constraint $\talpha=\talpha^\star$ in the probabilistic fluid model to get the following optimization problem:
\begin{subequations}
\begin{align}
    \tilde{R}^{*}_1=\lefteqn{\max_{\btlambda,\btchi} \ g(\btlambda,\btchi)= \inner{F(\btlambda)}{\btlambda}-\E{\talpha^\star}{c(\btmu)}}\\
 \textit{subject to} \quad h_j^{(2)}&=\tl_j-\sum_{i=1}^n \tchi_{ij}=0 \quad \forall j \in [m]  \\
    h_i^{(1)}&=\E{\talpha^\star}{\tmu_i}-\sum_{j=1}^m \tchi_{ij}=0 \quad \forall i \in [n] \\
    \tchi_{ij}&=0 \quad \forall (i,j) \notin E, \quad \tchi_{ij} \geq 0 \quad \forall (i,j) \in E, 
\end{align}
\label{eq: mod_prob_fluid_model}
\end{subequations}
As the optimal value of \eqref{eq: prob_fluid_model} is achieved by the feasible point $(\btlambda^\star,\btchi^\star)$ for \eqref{eq: mod_prob_fluid_model}, we have $\tilde{R}^\star_1 \geq \tilde{R}^\star$. In addition, as we added a constraint, the feasible region of \eqref{eq: mod_prob_fluid_model} is a subset of the feasible region of \eqref{eq: prob_fluid_model}, we have $\tilde{R}^\star \geq \tilde{R}^\star_1$. Thus, we have $\tilde{R}^\star=\tilde{R}_1^\star$ and the optimal solution of \eqref{eq: mod_prob_fluid_model} is $(\btlambda^\star,\btchi^\star)$. Now, we will use the KKT conditions in the following steps:
\begin{enumerate}
    \item First, we will show that the optimal point is a regular point, that is all the binding constraints are linearly independent.
    \item Then, we will use the given pricing policy to find a feasible direction for the optimization problem above.
    \item Finally, we will use the first order KKT optimality conditions as the objective function is concave and the feasible region is a polyhedron. 
\end{enumerate}

\textit{Part 1:} As we assume that $\tchi_{ij}^\star>0$ for all $i \in [n], j \in [m]$. In addition, as the feasible region is a polyhedron, the gradient vectors of all the active constraints at the optimal solution are linearly independent. Thus, the optimal point is a regular point. 

\textit{Part 2:} By hypothesis of the lemma, the DTMC operating under the given pricing and matching policy is stable. Thus, by Lemma \ref{lemma: necessary_constraints}, $(\E{}{\blambda(\bbarq)}),\E{\talpha^\star}{\btmu},\bchi^{\dagger})$ is a feasible solution to the fluid problem \eqref{eq: prob_fluid_model}, where $\bchi^\dagger$ is the corresponding `average' rate assignment matrix ($\btchi$) for the given policy. Thus, a feasible direction at the optimal point is given by 
\begin{align*}
    d=\begin{cases}
   \E{}{\lambda_j(\bbarq)}-\tl_j^\star  &\forall k \in [m] \\
    \hchi^\dagger_k-\hchi^\star_k &\forall k \in [nm]\backslash [m].
    \end{cases}
\end{align*}

\textit{Part 3:} Now, we will use the first order KKT optimality conditions for the optimization problem \eqref{eq: mod_prob_fluid_model}. There exists unique Lagrangian multipliers $(\bkappa,\bxi) \in \bbR^{m+n}\times \bbR_+^{mn}$ such that,
\begin{align*}
    \nabla g(\btlambda^\star,\btchi^\star)+\nabla \bh(\btlambda^\star,\btchi^\star)\bkappa+\hspace{-21pt}\sum_{k : (k \% (n+1),  \lceil\frac{k}{n}\rceil) \in E} \hspace{-2pt} \xi_k\be_{k+n+m} \mathbbm{1}_{\hchi^\star_k=0}+\hspace{-21pt}\sum_{k : (k \% (n+1),  \lceil\frac{k}{n}\rceil) \notin E} \hspace{-2pt} \xi_k\be_{k+n+m} =\bzero_{n+m+nm},
\end{align*}
where $\nabla g(\btlambda^\star,\btchi^\star)$ is the gradient of the objective function at the optimal point given by
\begin{align*}
    \nabla g(\btlambda^\star,\btchi^\star)=\left(F'(\btlambda^\star)\btlambda^\star+F(\btlambda^\star),\bzero_{nm}\right).
\end{align*}
In addition, as $h: \bbR^{n+nm} \rightarrow \bbR^{n+m}$, its gradient $\nabla h((\btlambda^\star,\btchi^\star))$ is a matrix in $\bbR^{(m+mn)\times (m+n)}$. Now, we will take the inner product of the optimality equation with the feasible direction $d$. Observe that
\begin{align*}
    \inner{\bd}{\nabla h_j^{(2)}}=h_j^{(2)}(\btlambda^\star,\btchi^\star)-h_j^{(2)}(\E{}{\blambda(\bbarq)},\bchi^\dagger)&=0 \quad \forall j \in [m] \\
     \inner{\bd}{\nabla h_i^{(1)}}=h_i^{(1)}(\btlambda^\star,\btchi^\star)-h_i^{(1)}(\E{}{\blambda(\bbarq)},\bchi^\dagger)&=0 \quad \forall i \in [n] \\
     \sum_{k : (k \% (n+1),  \lceil\frac{k}{n}\rceil) \in E} \xi_k(\hchi_k^\star-\hchi_k^\dagger) \mathbbm{1}_{\hchi^\star_k=0}&=0 \quad 
     \textit{(By the assumption $\btchi^\star>\bzero_{n\times m}$)} \\
     \sum_{k : (k \% (n+1),  \lceil\frac{k}{n}\rceil) \notin E}  \xi_k(\hchi_k^\star-\hchi_k^\dagger)&=0 \quad \textit{(As $\tchi_{ij}^\star=\tchi_{ij}^\dagger=0 \ \forall (i,j) \notin E$ )}
\end{align*}
Thus, we have $\inner{\nabla g(\btlambda^\star,\btchi^\star)}{\bd}=0$. This gives us the lemma. \hfill $\Halmos$
 \endproof
\proof{Proof of Lemma \ref{lemma: first_order_optimality}}
First note that, the sequence of DTMC operating under the pricing and matching policy given by \eqref{eq: two_price_policy} \eqref{eq: modified_max_weight_matching} is stable. In addition, $\E{}{\inner{\bone_{n+m}}{\bbarq_\eta}}<\infty$ by the hypothesis of the Lemma \ref{lemma: first_order_optimality}. Thus, we can use Lemma \ref{lemma: first_order}. We have
\begin{align*}
    \E{}{\lambda_{j,\eta}(\bbarq)}=\tl^\star_j+\epsilon_{\eta}\left(\P{}{\barq_{j,\eta}^{(2)}=0}-\P{}{\barq_{j,\eta}^{(2)}>0}\right) \quad \forall j\in [m].
\end{align*}
This give us
\begin{align*}
    \sum_{j=1}^m \left(\tl^\star_jF'(\tl^\star_j)+F_j(\tl^\star_j)\right)\left(\P{}{\barq_{j,\eta}^{(2)}>0}-\P{}{\barq_{j,\eta}^{(2)}=0}\right)=0. \Halmos
\end{align*}
 \endproof
\subsection{Proof of Lemma \ref{lemma: profit_loss}}
\proof{Proof}
We have 
\begin{align*}
    \frac{L^P_\eta(\pi_\eta)}{\eta}&=\tilde{R}^\star-P_\eta(\pi_\eta) \\
    ={}&\inner{F(\btlambda^\star)}{\btlambda^\star}-\E{\talpha^\star}{c(\btmu)}-\E{}{\inner{F\left(\blambda_\eta(\bbarq_\eta)\right)}{\blambda_\eta(\bbarq_\eta)}}+\E{\talpha^\star}{c(\btmu)} \\
    ={}&\inner{F(\btlambda^\star)}{\btlambda^\star}-\E{}{\inner{F\left(\blambda_\eta(\bbarq_\eta)\right)}{\blambda_\eta(\bbarq_\eta)}} \\
    \overset{(a)}{=}{}&\inner{F(\btlambda^\star)}{\btlambda^\star}-\sum_{j=1}^m \left(\tl_j^\star+\epsilon_\eta\right)F_j\left(\tl_j^\star+\epsilon_\eta\right)\P{}{\barq_{j,\eta}^{(2)}=0}\\
    &-\sum_{j=1}^m \left(\tl_j^\star-\epsilon_\eta\right)F_j\left(\tl_j^\star-\epsilon_\eta\right)\P{}{\barq_{j,\eta}^{(2)}>0} \\
    \overset{(b)}{=}{}&\inner{F(\btlambda^\star)}{\btlambda^\star}-\sum_{j=1}^m \left(\tl_j^\star+\epsilon_\eta\right)\left(F_j(\tl_j^\star)+\epsilon_\eta F'_j(\tl_j^\star)+\frac{\epsilon_\eta^2}{2}F''(\tl_j^\star)+O\left(\epsilon_\eta^3\right)\right)\P{}{\barq_{j,\eta}^{(2)}=0}\\
    &-\sum_{j=1}^m \left(\tl_j^\star-\epsilon_\eta\right)\left(F_j(\tl_j^\star)-\epsilon_\eta F'_j(\tl_j^\star)+\frac{\epsilon_\eta^2}{2}F''(\tl_j^\star)+O\left(\epsilon_\eta^3\right)\right)\P{}{\barq_{j,\eta}^{(2)}>0} \\
    ={}& \epsilon_\eta \sum_{j=1}^m \left(\tl_j^\star F_j'(\tl_j^\star)+F_j(\tl_j^\star)\right)\left(\P{}{\barq_{j,\eta}^{(2)}>0}-\P{}{\barq_{j,\eta}^{(2)}=0}\right)\\
    &-\epsilon_\eta^2\sum_{j=1}^m \left(\frac{\tl^\star_jF''(\tl^\star_j)}{2}+F'_j(\tl_j^\star)\right)+O\left(\epsilon_\eta^3\right) \\
    \overset{(c)}{=}{}& -\epsilon_\eta^2\sum_{j=1}^m \left(\frac{\tl^\star_jF''(\tl^\star_j)}{2}+F'_j(\tl_j^\star)\right)+O\left(\epsilon_\eta^3\right) 
\end{align*}
where $(a)$ follows by the definition of the pricing policy given by \eqref{eq: two_price_policy}. Next, $(b)$ follows by Taylor's series expansion and using the Assumption \ref{ass: monotonic} that the demand curve $F_j(.)$ is twice continuously differentiable. Finally, $(c)$ follows by Lemma \ref{lemma: first_order_optimality}. Also, note that 
\begin{align*}
    \frac{1}{2}\inner{F(\btlambda^\star)}{\btlambda^\star}+\inner{\bone_m}{F'(\btlambda^\star)} &=\frac{1}{2}\left(\inner{F(\btlambda^\star)}{\btlambda^\star}+\inner{\bone_m}{F'(\btlambda^\star)}\right)+\frac{1}{2}\inner{\bone_m}{F'(\btlambda^\star)}<0,
\end{align*}
as by Assumption \ref{ass: concave}, $F_j(\lambda_j)\lambda_j$ is a concave function, thus the second derivative is non positive and by Assumption \ref{ass: monotonic}, the demand function is strictly decreasing, thus the derivative is negative. \hfill $\Halmos$
 \endproof
\subsection{Proof of Theorem \ref{theo: two_price_policy}}
\proof{Proof}
The net profit loss is given by
\begin{align*}
    L_\eta&=\tilde{R}_\eta^\star-R_\eta \\
    &= \tilde{R}^\star_\eta- P_\eta +  \E{}{\inner{\bs}{\bbarq_\eta}} \\
    &\overset{(a)}{=} -\eta\epsilon_\eta^2\sum_{j=1}^m \left(\frac{\tl^\star_jF''(\tl^\star_j)}{2}+F'_j(\tl_j^\star)\right)+O\left(\eta\epsilon_\eta^3\right)+  \E{}{\inner{\bs}{\bbarq_\eta}} \\
    & \overset{(b)}{\leq} -\eta\epsilon_\eta^2\sum_{j=1}^m \left(\frac{\tl^\star_jF''(\tl^\star_j)}{2}+F'_j(\tl_j^\star)\right)+\frac{B}{\epsilon_\eta}+O\left(\eta\epsilon_\eta^3\right) \\
    &\overset{(c)}{=} O\left(\eta^{1/3}\right),
\end{align*}
where $(a)$ follows by Lemma \ref{lemma: profit_loss}, $(b)$ follows by Lemma \ref{lemma: stability} and $(c)$ follows by picking $\epsilon_\eta=\eta^{-1/3}$ considering the trade-off between the profit loss ($\eta\epsilon_\eta^2$) and the expected queue length ($\frac{1}{\epsilon_\eta}$). This completes the proof. \hfill $\Halmos$
 \endproof
\section{Proof of Theorem \ref{theo: lowerbound_queue} and \ref{theo: lowerbound_profit_loss}: Lower bound}
\subsection{Proof of Theorem \ref{theo: lowerbound_queue}}
\proof{Proof}
We will start by defining the imbalance of the DTMC given by 
\begin{align*}
    z \dfn \inner{\bone_n}{\bq^{(1)}}-\inner{\bone_m}{\bq^{(2)}}.
\end{align*}
The update equation of imbalance given the queue length vector $\bq(k)$ can be written as
\begin{align*}
    z(k+1)=z(k)+\inner{\bone_n}{\ba^{(1)}(k)}-\inner{\bone_m}{\ba^{(2)}(k)}.
\end{align*}
Note, that $z$ itself is not a Markov chain as the arrival vector $(\ba(k))$ depends on the queue length $(\bq(k))$. Denote $\gamma\dfn\max\{m,n\}$ and also denote $ p\dfn\P{}{\barz\geq 0}$, that is
\begin{align*}
    p=\sum_{x=0}^{\infty}\sum_{\bq: \inner{\bone_n}{\bq^{(1)}}-\inner{\bone_m}{\bq^{(2)}}=x} \P{}{\bbarq=\bq}.
\end{align*}
For this proof, we will couple the absolute value of the imbalance with a single server queue denoted by $q^\dagger$ and arrival and service rate denoted by $a^\dagger$ and $s^\dagger$. In particular, we will carry out the proof in the following steps:
\begin{enumerate}
    \item First, we will construct the arrival and service process of the single server queue $q^\dagger(k)$.
    \item Then, we will couple the single server queue with the imbalance such that $q^\dagger(k) \leq |z(k)|$ for all $k \in \bbZ_+$
    \item Then we will calculate $\E{}{\barq^\dagger}$ and use this to lower bound $\E{}{\inner{\bone_{n+m}}{\bbarq}}$.
\end{enumerate}

\textit{Step 1 (Single Server Queue):}  For all $k \geq 0$, generate the following random variables independent of all the other random variables:
\begin{align*}
    s_1^\dagger(k) &\stg \inner{\bone_n}{\ba^{(1)}(k)}; \quad \E{}{s_1^\dagger(k)}=\inner{\bone_m}{\btlambda^\star}+\gamma\epsilon, \ \Var{s_1^\dagger(k)}=\bone_{n \times n} \circ \Sigma^{(1)} \\
     s_2^\dagger(k) &\stg \inner{\bone_n}{\ba^{(2)}(k)}; \quad \E{}{s_2^\dagger(k)}=\inner{\bone_m}{\btlambda^\star}+\gamma\epsilon,\ \Var{s_2^\dagger(k)}=\bone_{n \times n} \circ \Sigma^{(2)}_s \\
     a_1^\dagger(k) &\stl \inner{\bone_n}{\ba^{(1)}(k)}; \quad \E{}{a_1^\dagger(k)}=\inner{\bone_m}{\btlambda^\star}-\gamma\epsilon,\ \Var{a_1^\dagger(k)}=\bone_{n \times n} \circ \Sigma^{(1)} \\
     a_2^\dagger(k) &\stl \inner{\bone_n}{\ba^{(2)}(k)}; \quad \E{}{a_2^\dagger(k)}=\inner{\bone_m}{\btlambda^\star}-\gamma\epsilon,\ \Var{a_2^\dagger(k)}=\bone_{n \times n} \circ \Sigma^{(2)}_a.
\end{align*}
where $\Sigma^{(1)}=\Sigma^{(1)}(\talpha^\star)$ and $\Sigma^{(2)}_{\min} \leq \Sigma^{(2)}_l \leq \Sigma^{(2)}_{\max}$ for $l \in \{a,s\}$. In addition, we also have $s_i^\dagger(k) \leq \gamma A_{\max}$ and $a_i^\dagger(k) \leq \gamma A_{\max}$ with probability 1 for $i \in \{1,2\}$. Note that $\inner{\bone_m}{\btlambda^\star}=\inner{\bone_n}{\E{\talpha^\star}{\bmu}}$ by the constraints of the probabilistic fluid model \eqref{eq: prob_fluid_model}. Thus, it is possible to generate such random variables as their mean is greater than or equal to the corresponding arrival process of the imbalance and their variance can be picked appropriately. For example, it suffices to just consider $s_2^\dagger(k)$ to have the same distribution as $\inner{\bone_n}{\ba^{(2)}(k)}+\epsilon\left(\frac{\gamma}{n}-1\right)$ where the price $\bp(k)$ is such that $\blambda(k)=\E{}{\ba^{(2)}(k)}=\btlambda^\star+\epsilon\bone_n$. Now, we define the arrival and service process of the single server queue. Consider a random variable $y(k) \sim \textrm{Bernoulli}(p)$ independent of $s_1^\dagger$, $s_2^\dagger$, $a_1^\dagger$ and $a_2^\dagger$.  For all $k \in \bbZ_+$, we define
\begin{align*}
    s^\dagger(k)&=s_1^\dagger(k)\mathbbm{1}_{\left\{y(k)<0\right\}}+s_2^\dagger(k)\mathbbm{1}_{\left\{y(k) \geq 0\right\}}\\
     a^\dagger(k)&=a_1^\dagger(k)\mathbbm{1}_{\left\{y(k)\geq0\right\}}+a_2^\dagger(k)\mathbbm{1}_{\left\{y(k) < 0\right\}}
\end{align*}
The marginal distribution of the arrival and service process is given by
\begin{align*}
    \P{}{s^\dagger(k) \leq x}&=(1-p)\P{}{s_1^\dagger(k) \leq x}+p\P{}{s_2^\dagger(k) \leq x} \quad \forall x \in \bbR \\
    \P{}{a^\dagger(k) \leq x}&=p\P{}{a_1^\dagger(k) \leq x}+(1-p)\P{}{a_2^\dagger(k) \leq x} \quad \forall x \in \bbR.
\end{align*}
Note that the arrival and service process are not independent and the mean and variance of them are
\begin{align*}
    \E{}{s^\dagger(k)}&=\inner{\bone_m}{\btlambda^\star}+\gamma\epsilon; \quad  \Var{s^\dagger(k)}=(1-p)\bone\circ\Sigma^{(1)}+p\bone\circ\Sigma^{(2)}_s\\
    \E{}{a^\dagger(k)}&=\inner{\bone_m}{\btlambda^\star}-\gamma\epsilon; \quad \Var{a^\dagger(k)}=p\bone\circ\Sigma^{(1)}+(1-p)\bone\circ\Sigma^{(2)}_a  \\
    \span \Cov{a^\dagger(k)}{s^\dagger(k)}=0.
\end{align*}
The variance can be calculated as follows
\begin{align*}
    \Var{a^\dagger(k)}&=\E{}{(a^\dagger(k))^2}-\E{}{a^\dagger(k)}^2 \\
    &=\E{}{\left(a_1^\dagger(k)\mathbbm{1}_{\left\{y(k)\geq0\right\}}+a_2^\dagger(k)\mathbbm{1}_{\left\{y(k) < 0\right\}}\right)^2}-\E{}{a_1^\dagger(k)\mathbbm{1}_{\left\{y(k)\geq0\right\}}+a_2^\dagger(k)\mathbbm{1}_{\left\{y(k) < 0\right\}}}^2 \\
    &=\E{}{(a_1^\dagger(k))^2}p+\E{}{(a_2^\dagger(k))^2}(1-p)-\left(\inner{\bone_m}{\btlambda^\star}-\gamma\epsilon\right)^2 \\
    &=\left(\Var{a_1^\dagger(k)}+\E{}{a_1^\dagger(k)}^2\right)p+\left(\Var{a_2^\dagger(k)}+\E{}{a_2^\dagger(k)}^2\right)(1-p)-\left(\inner{\bone_m}{\btlambda^\star}-\gamma\epsilon\right)^2 \\
    &=p\bone_{n \times n}\circ\Sigma^{(1)}+(1-p)\bone_{m \times m}\circ\Sigma^{(2)}_a.
\end{align*}
Similarly, we can also calculate the variance of $s^\dagger(k)$ and we omit it here as the steps are repetitive. In addition, we can also find the co-variance between the arrival and service process as follows:
\begin{align*}
    \Cov{a^\dagger(k)}{s^\dagger(k)}={}&\E{}{a^\dagger(k)s^\dagger(k)}-\E{}{a^\dagger(k)}\E{}{s^\dagger(k)} \\
    ={}&\E{}{\left(a_1^\dagger(k)\mathbbm{1}_{\left\{y(k)\geq0\right\}}+a_2^\dagger(k)\mathbbm{1}_{\left\{y(k) < 0\right\}}\right)\left(s_1^\dagger(k)\mathbbm{1}_{\left\{y(k)<0\right\}}+s_2^\dagger(k)\mathbbm{1}_{\left\{y(k) \geq 0\right\}}\right)}\\
    &-\E{}{a_1^\dagger(k)\mathbbm{1}_{\left\{y(k)\geq0\right\}}+a_2^\dagger(k)\mathbbm{1}_{\left\{y (k)< 0\right\}}}\E{}{s_1^\dagger(k)\mathbbm{1}_{\left\{y(k)<0\right\}}+s_2^\dagger(k)\mathbbm{1}_{\left\{y(k) \geq 0\right\}}} \\
    ={}& p\E{}{a_1^\dagger(k) s_2^\dagger(k)}+(1-p)\E{}{a_2^\dagger(k) s_1^\dagger(k)}-\inner{\bone_m}{\btlambda^\star}^2+\gamma^2\epsilon^2=0,
\end{align*}
where the last equality follows as $a_1^\dagger$, $a_2^\dagger$, $s_1^\dagger$ and $s_2^\dagger$ are independent of each other.

\textit{Step 2 (Coupling):} We couple the arrival and service process of the multiple link two sided queue and the single server queue as follows: If $z(k)\geq 0$ then $s^\dagger(k) \geq a_2(k)$ and $a^\dagger(k) \leq a_1(k)$ with probability 1. Also, if $z(k)<0$, then $a^\dagger(k) \leq a_2(k)$ and $s^\dagger(k) \geq a_1(k)$. Note that, such a coupling is possible if $\P{}{z(k)\geq0}=p$ for all $k \in \bbZ_+$. To achieve this, we will initialize $z(k)$ appropriately. Now, we prove by induction that under the above defined coupling, $q^\dagger(k) \leq |z(k)|$ for all $k \in \bbZ_+$.

\textbf{Base Case:} Initialize $\bq(0)$ by its stationary distribution, so we have
\begin{align*}
    \P{}{z(k)=x}=\sum_{\bq: \inner{\bone_n}{\bq^{(1)}}-\inner{\bone_m}{\bq^{(2)}}=x} \P{}{\bbarq=\bq} \quad \forall x \in \bbZ, \ \forall k \in \bbZ_+.
\end{align*}
In addition, initialize $q^\dagger(0)$ with the same distribution as $|z(0)|$. As both of the them has the same distribution, we can couple the two random variables such that $q^\dagger(0)=|z(0)|$. So, the base case is satisfied. In addition, $\P{}{z(k)\geq 0}=p$ for all $k \in \bbZ_+$.

\textbf{Induction Hypothesis:} $q^\dagger(k') \leq |z(k')|$ for all $k' \in [k]$.

\textbf{Induction Step:} We will consider the following two cases:

\textit{Case I:} $z(k)\geq0$. In this case, we have $s^\dagger(k)\geq a_2(k)$ and $a^\dagger(k) \leq a_1(k)$. So, we have
\begin{align*}
    q^\dagger(k+1)&=\max\left\{0,q^\dagger(k)+a^\dagger(k)-s^\dagger(k)\right\} \\
    &\leq \max\left\{0,z(k)+a^\dagger(k)-s^\dagger(k)\right\} \quad \textit{(Induction Hypothesis)} \\
    &\leq \max\left\{0,z(k)+a_1(k)-a_2(k)\right\} \quad \textit{(Coupling)} \\
    &\leq |z(k)+a_1(k)-a_2(k)|=|z(k+1)|.
\end{align*}

\textit{Case II:} $z(k)<0$. In this case, we have $a^\dagger(k) \leq a_2(k)$ and $s^\dagger(k) \geq a_1(k)$. So, we have
\begin{align*}
     q^\dagger(k+1)&=\max\left\{0,q^\dagger(k)+a^\dagger(k)-s^\dagger(k)\right\} \\
     &=-\min\left\{0,-q^\dagger(k)-a^\dagger(k)+s^\dagger(k)\right\} \\
     &\leq-\min\left\{0,z(k)-a^\dagger(k)+s^\dagger(k)\right\} \quad \textit{(Induction Hypothesis)} \\
     &\leq -\min\left\{0,z(k)-a_2(k)+a_1(k)\right\} \quad \textit{(Coupling)} \\
     &\leq |z(k+1)|.
\end{align*}
This completes our proof that $q^\dagger(k) \leq |z(k)|$ for all $k \in \bbZ_+$. Thus, $\P{}{q^\dagger(k) \leq x} \leq \P{}{|z(k)| \leq x}$ for all $x \in \bbR_+$. Taking the limit as $k$ goes to infinity, we get $\P{}{\barq^\dagger \leq x} \leq \P{}{|\barz| \leq x}$ and thus, we have $\E{}{\barq^\dagger}\leq \E{}{|\barz|}$. 

\textit{Step 3 $(\E{}{\barq^\dagger})$:} Now, we will analyze the single server queue to find its expectation in steady state.
By taking $V(q^\dagger)=(q^\dagger)^2$ as the Lyapunov function, in steady state, we have
\begin{align}
    \E{}{\Delta V(\barq^\dagger)}=0 \Rightarrow \E{}{(\barq^{\dagger,+})^2-(\barq^\dagger)^2}&=0 \nonumber\\
    \Rightarrow \E{}{(\barq^{\dagger,+}-\baru^\dagger+\baru^\dagger)^2-(\barq^\dagger)^2}&=0 \nonumber\\
    \Rightarrow \E{}{(\barq^\dagger+\bara^\dagger-\bars^\dagger)^2-(\baru^\dagger)^2-(\barq^\dagger)^2}&=0\nonumber \\
    \Rightarrow \E{}{(\bara^\dagger-\bars^\dagger)^2+2\barq^\dagger(\bara^\dagger-\bars^\dagger)-(\baru^\dagger)^2}&=0\nonumber \\
    \Rightarrow \Var{\bara^\dagger}+\Var{s^\dagger}-2\Cov{\bara^\dagger}{\bars^\dagger}
    +\E{}{\bara^\dagger-\bars^\dagger}^2-4\gamma\epsilon\E{}{\barq^\dagger}-2\gamma^2A_{\max}\epsilon&\overset{*}{=}0,  \label{eq: single_server_queue}
    \end{align}
    where $(*)$ follows by taking $V(q^\dagger)=q^\dagger$ as the Lyapunov function. We have
    \begin{align*}
        \E{}{\Delta V(\barq^\dagger)}=0 \Rightarrow \E{}{\barq^{\dagger,+}-\barq^\dagger}=0 \Rightarrow \E{}{\bara^\dagger-\bars^\dagger+\baru^\dagger}=0
        \Rightarrow \E{}{\baru^\dagger}=2\gamma\epsilon.
    \end{align*}
    In addition, as $\baru^\dagger \leq \bars^\dagger \leq \gamma A_{\max}$, we have $\E{}{(\baru^\dagger)^2} \leq \gamma A_{\max}\E{}{\baru^\dagger}=2\gamma^2A_{\max}\epsilon$.
    Now, simplifying \eqref{eq: single_server_queue}, we get
    \begin{align*}
    \Rightarrow \E{}{\barq^\dagger}&=\frac{\bone_{n \times n}\circ \Sigma^{(1)}+p\bone_{m \times m} \circ \Sigma^{(2)}_s+(1-p)\bone_{m \times m} \circ \Sigma^{(2)}_a+4\gamma^2\epsilon^2-2\gamma^2A_{\max}\epsilon}{4\gamma\epsilon} \\
    &\geq \frac{\bone_{n \times n}\circ \Sigma^{(1)}+\bone_{m \times m} \circ \Sigma^{(2)}_{\min}}{8\gamma\epsilon} \quad \forall \epsilon \leq \frac{\bone_{n \times n}\circ \Sigma^{(1)}+\bone_{m \times m} \circ \Sigma^{(2)}_{\min}}{4\gamma^2A_{\max}}
\end{align*}
Thus, we have
\begin{align*}
    \E{}{\inner{\bone_{n+m}}{\bbarq}} \geq \E{}{|\barz|} \geq \E{}{\barq^\dagger} \geq \frac{\bone_{n \times n}\circ \Sigma^{(1)}+\bone_{m \times m} \circ \Sigma^{(2)}_{\min}}{8\gamma\epsilon} \quad \forall \epsilon \leq \frac{\bone_{n \times n}\circ \Sigma^{(1)}+\bone_{m \times m} \circ \Sigma^{(2)}_{\min}}{4\gamma^2A_{\max}}.
\end{align*}
 \endproof
 \subsection{Proof of Theorem \ref{theo: lowerbound_profit_loss}}
We now present a lemma which will assist us in proving Theorem \ref{theo: lowerbound_profit_loss}.
\begin{lemma} \label{lemma: tail_probability}
Under the hypothesis of Theorem \ref{theo: lowerbound_profit_loss}, there exists a constant $\delta>0$ independent of $\eta$ and another constant $\eta_1>0$ such that for all $\eta>\eta_1$
\begin{align*}
    \E{}{\sum_{j=1}^m \phi_j^2\left(\frac{\bbarq_\eta}{\eta^\alpha}\right)} \geq \delta.
\end{align*}
\end{lemma}
\proof{Proof of Lemma \ref{lemma: tail_probability}}
 In this proof, we will couple the sequence of DTMCs $\{\bq_\eta(k): k \in \bbZ_+\}$ with a sequence of single server queues $q^{\dagger}_\eta$ with arrival and service defined as in the proof of Theorem \ref{theo: lowerbound_queue} with $\epsilon$ dependent on $\eta$. In particular, we have $\epsilon_\eta=M\eta^\beta$. By the coupling defined above, we have $\P{}{\barq^{\dagger}_\eta > x} \leq \P{}{|\barz_\eta|>x}$ for all $x>0$. In addition, we know that as $\eta \rightarrow \infty$, we have $\epsilon_\eta \rightarrow 0$ as $\beta<1$. Thus, by \citep{hurtado2019heavy} we know that 
 \begin{align*}
     \epsilon_\eta \barq^{\dagger}_\eta \overset{d}{\rightarrow} \textrm{Exp}\left(\sigma_s^2=\frac{\bone_{n \times n} \circ \Sigma^{(1)}+\bone_{m \times m}\circ p\Sigma^{(2)}_s+(1-p)\Sigma^{(2)}_a}{4\gamma}\right)
 \end{align*}
 Even though, in \citep{hurtado2019heavy} they assume the arrival process and service process are independent of each other, it suffices to have them uncorrelated. Now, by the definition of weak convergence, for $K>0$
 \begin{align*}
     \lim_{\eta \rightarrow \infty}\P{}{\epsilon_\eta \barq^{\dagger}_\eta > (n+m)KM} = e^{-\frac{(n+m)KM}{\sigma_s^2}}.
 \end{align*}
 Thus, for a given $K>0$, there exists $\eta_1(K)>0$ such that for all $\eta>\eta_1$, we have 
 \begin{align*}
     \P{}{\epsilon_\eta \barq^{\dagger}_\eta > (n+m)KM} \geq \frac{1}{2}e^{-\frac{(n+m)KM}{\sigma_s^2}}
 \end{align*}
 Now, by using the coupling, we have
 \begin{align*}
     \P{}{|z_\eta| > (n+m)K\eta^\alpha} &\geq \P{}{q^{\dagger}_\eta>(n+m)K\eta^\alpha}=\P{}{M\eta^\beta q^{\dagger}_\eta>(n+m)KM\eta^{\alpha+\beta}}  \\
     &\geq \P{}{\epsilon_\eta q^{\dagger}_\eta>(n+m)KM} \geq  \frac{1}{2}e^{-\frac{(n+m)KM}{\sigma_s^2}} \quad \forall \eta > \eta_1
 \end{align*}
 Finally, note that $ \{z_\eta > (n+m)K\eta^\alpha\} \subseteq \{||\bbarq_\eta||_\infty > K\eta^\alpha\}$, so we have
 \begin{align*}
     \E{}{\sum_{j=1}^m \phi_j^2\left(\frac{\bbarq_\eta}{\eta^\alpha}\right)} \geq \sigma^2 \P{}{||\bbarq_\eta||_\infty > K\eta^\alpha} &\geq \sigma^2 \P{}{z_\eta > (n+m)K\eta^\alpha} \\
     &\geq \frac{\sigma^2}{2} e^{-\frac{(n+m)KM}{\sigma_s^2}} \dfn \delta \quad \forall \eta>\eta_1. \Halmos
 \end{align*}
 \endproof
\proof{Proof of Theorem \ref{theo: lowerbound_profit_loss}}
 In this proof, we will use Taylor's theorem to expand the profit-loss and show that the second order term does not vanish using Lemma \ref{lemma: tail_probability}. This proof follows similarly as in \citep{varma2020dynamic}. The only non trivial step was to prove Lemma \ref{lemma: tail_probability}.
 \begin{align*}
     \lefteqn{\frac{\tilde{R}^\star_\eta-P_\eta}{\eta}}\\
     ={}&\inner{F(\btlambda^\star)}{\btlambda^\star}-\E{\talpha^\star}{c(\btmu)}-\E{}{\inner{F\left(\blambda_\eta(\bbarq_\eta)\right)}{\blambda_\eta(\bbarq_\eta)}}+\E{\talpha^\star}{c(\btmu)} \\
    ={}&\inner{F(\btlambda^\star)}{\btlambda^\star}-\E{}{\inner{F\left(\blambda_\eta(\bbarq_\eta)\right)}{\blambda_\eta(\bbarq_\eta)}} \\
    ={}&\inner{F(\btlambda^\star)}{\btlambda^\star}-\sum_{j=1}^m\E{}{\left(\tl_j^\star+\phi_j\left(\frac{\bbarq_\eta}{\eta^\alpha}\right)\right)F_j\left(\tl_j^\star+\phi_j\left(\frac{\bbarq_\eta}{\eta^\alpha}\right)\right)} \\
    ={}&\inner{F(\btlambda^\star)}{\btlambda^\star}-\sum_{j=1}^m \E{}{\left(\tl_j^\star+\phi_j\left(\frac{\bbarq_\eta}{\eta^\alpha}\right)\eta^\beta\right)\left(F_j(\tl_j^\star)+\phi_j\left(\frac{\bbarq_\eta}{\eta^\alpha}\right)F_j'(\tl_j^\star)\eta^\beta+\phi_j^2\left(\frac{\bbarq_\eta}{\eta^\alpha}\right)F_j''(\hlambda_j^\star(\bbarq_\eta))\eta^{2\beta}\right)} \\
    ={}& -\left(\sum_{j=1}^m \E{}{\phi_j\left(\frac{\bbarq_\eta}{\eta^\alpha}\right)}\left(F_j(\tl_j^\star)+\tl_j^\star F'_j(\tl_j^\star)\right)\right)\eta^\beta-\left(\sum_{j=1}^m \E{}{\phi_j^2\left(\frac{\bbarq_\eta}{\eta^\alpha}\right)\left(F_j''(\hlambda_j^\star(\bbarq_\eta))\tl_j^\star+F_j'(\tl_j^\star)\right)}\right)\eta^{2\beta} \\
    &-\sum_{j=1}^m\E{}{\phi_j^3\left(\frac{\bbarq_\eta}{\eta^\alpha}\right)F_j''(\hlambda_j^\star(\bbarq_\eta))}\eta^{3\beta} \\
    \overset{(a)}{=}{}& -\left(\sum_{j=1}^m \E{}{\phi_j^2\left(\frac{\bbarq_\eta}{\eta^\alpha}\right)\left(F_j''(\hlambda_j^\star(\bbarq_\eta))\tl_j^\star+F_j'(\tl_j^\star)\right)}\right)\eta^{2\beta}-\sum_{j=1}^m\E{}{\phi_j^3\left(\frac{\bbarq_\eta}{\eta^\alpha}\right)F_j''(\hlambda_j^\star(\bbarq_\eta))}\eta^{3\beta} \\
    \overset{(b)}{\geq}{}& \frac{\delta}{2}\left(\min_{j\in [m]}\left\{-F_j''(\tl_j^\star)\tl_j^\star-F_j'(\tl_j^\star)\right\}\right)\eta^{2\beta}-\sum_{j=1}^m\E{}{\phi_j^3\left(\frac{\bbarq_\eta}{\eta^\alpha}\right)F_j''(\hlambda_j^\star(\bbarq_\eta))}\eta^{3\beta} \\
    \overset{(c)}{\geq}{}&\frac{\delta}{4}\left(\min_{j\in [m]}\left\{-F_j''(\tl_j^\star)\tl_j^\star-F_j'(\tl_j^\star)\right\}\right)\eta^{2\beta}.
 \end{align*}
The remainder term of the Taylor's expansion is $F''(\hlambda^\star(\bbarq_\eta))$ for some $\hlambda^\star(\bbarq_\eta) \in [\tl_j^\star-M\eta^\beta,\tl_j^\star+M\eta^\beta]$ for all $\bbarq_\eta \in \calS$. Note that the second derivative of $\lambda_jF(\lambda_j)$ is negative as it is concave by Assumption \ref{ass: concave} and $F_j(.)$ is strictly decreasing by Assumption \ref{ass: monotonic}. Thus, the coefficient of $\eta^{2\beta}$ is positive. This completes the proof. Now we will justify $(a)$, $(b)$ and $(c)$ below. \textit{Proof of $(a)$} follows by Lemma \ref{lemma: first_order}.
 
 \textit{Proof of $(b)$} follows by uniform convergence of $F''_j(\hlambda^\star_j(\bbarq_\eta))$ to $F''(\tl_j^\star)$. To expound, by Taylor's Theorem and Assumption \ref{condition: general_pricing} \ref{condition: bounded}, we have $\hlambda^\star_j(\bbarq_\eta) \in [\tl_j^\star-M\eta^\beta,\tl_j^\star+M\eta^\beta]$. By Assumption \ref{ass: monotonic}, $F''(.)$ is continuous, thus, given $\bar{\gamma}=\frac{1}{2}\min_{j \in [m]}\left\{-F_j''(\tl_j^\star)-F_j'(\tl_j^\star)/\tl_j^\star\right\}>0$, there exists $\delta_2>0$, such that for all $|\tl_j^\star-\hlambda_j^\star|<\delta_2$, we have $|F''(\tl_j^\star)-F''(\hlambda_j^\star)|<\bar{\gamma}$. Thus, for all $\eta>\left(\frac{\delta_2}{M}\right)^{1/\beta}$, we have
 \begin{align*}
     \sup_{\bbarq_\eta \in \calS} |F''_j(\hlambda_j^\star(\bbarq_\eta))-F''_j(\tl_j^\star)|<\bar{\gamma} \quad \forall j \in [m].
 \end{align*}
 Thus, for $\eta > \max\{\eta_1,\left(\frac{\delta_2}{M}\right)^{1/\beta}\}$ we have
 \begin{align*}
     -\left(\sum_{j=1}^m \E{}{\phi_j^2\left(\frac{\bbarq_\eta}{\eta^\alpha}\right)\left(F_j''(\hlambda_j^\star(\bbarq_\eta))\tl_j^\star+F_j'(\tl_j^\star)\right)}\right) &\geq  \left(\sum_{j=1}^m\left(-F_j''(\tl_j^\star)\tl_j^\star-\bar{\gamma}\tl_j^\star-F_j'(\tl_j^\star)\right) \E{}{\phi_j^2\left(\frac{\bbarq_\eta}{\eta^\alpha}\right)}\right) \\
     &\overset{*}{\geq} \frac{1}{2}\left(\sum_{j=1}^m\left(-F_j''(\tl_j^\star)\tl_j^\star-F_j'(\tl_j^\star)\right) \E{}{\phi_j^2\left(\frac{\bbarq_\eta}{\eta^\alpha}\right)}\right) \\
     &\overset{**}{\geq} \frac{\delta}{2}\left(\min_{j\in [m]}\left\{-F_j''(\tl_j^\star)\tl_j^\star-F_j'(\tl_j^\star)\right\}\right)
 \end{align*}
 where $(*)$ follows by the definition of $\bar{\gamma}$ and $(**)$ follows by Lemma \ref{lemma: tail_probability}.
 
 \textit{Proof of $(c)$} follows as $\eta^{3\beta}$ is of lower order than $\eta^{2\beta}$ as $\beta<0$. In particular
 \begin{align*}
     -\sum_{j=1}^m\E{}{\phi_j^3\left(\frac{\bbarq_\eta}{\eta^\alpha}\right)F_j''(\hlambda_j^\star(\bbarq_\eta))}\eta^{3\beta} &\geq - \frac{3}{2}\sum_{j=1}^m\E{}{\bigg|\phi_j^3\left(\frac{\bbarq_\eta}{\eta^\alpha}\right)F_j''(\tl_j^\star)\bigg|}\eta^{3\beta} \quad \forall \eta>\eta_2 \\
     &\geq - \frac{3M^3}{2}\sum_{j=1}^m|F_j''(\tl_j^\star)|\eta^{3\beta} \quad \forall \eta>\eta_2
 \end{align*}
 Now, as $\beta<0$, with $\eta_3 \dfn (\frac{6M^3}{\delta\min_{j\in [m]}\left\{-F_j''(\tl_j^\star)\tl_j^\star-F_j'(\tl_j^\star)\right\}}\sum_{j=1}^m|F_j''(\tl_j^\star)|)^{-1/\beta}$ for all $\eta>\max\{\eta_2,\eta_3\}$, we have
 \begin{align*}
      \sum_{j=1}^m\E{}{\phi_j^3\left(\frac{\bbarq_\eta}{\eta^\alpha}\right)F_j''(\tl_j^\star)}\eta^{3\beta} \geq -\frac{\delta}{4}\left(\min_{j\in [m]}\left\{-F_j''(\tl_j^\star)\tl_j^\star-F_j'(\tl_j^\star)\right\}\right)\eta^{2\beta}
 \end{align*}
 This completes the proof. \hfill $\Halmos$
 \endproof
 \subsection{Proof of Corollary \ref{corollary: lower_bound_net_profit_loss}} \label{app: lower_bound_corollary}
 \proof{Proof}
Consider the sequence of DTMCs parametrized by $\eta$. By using Theorem \ref{theo: lowerbound_queue} with $\epsilon(\eta)=M\eta^\beta$ we have
 \begin{align*}
      \E{}{\inner{\bone_{n+m}}{\bbarq_\eta}}  \geq \frac{\bone_{n \times n}\circ \Sigma^{(1)}+\bone_{m \times m} \circ \Sigma^{(2)}}{8\max\{m,n\}M\eta^\beta} \quad \forall \eta>\left(\frac{\epsilon_0}{M}\right)^{1/\beta}
 \end{align*}
 Now, the net profit-loss for all $\eta>\max\left\{\left(\frac{\epsilon_0}{M}\right)^{1/\beta},\eta_1\right\}$ is given by
 \begin{align*}
     L_\eta(\pi_\eta)&=R^\star_\eta-P_\eta(\pi_\eta)-\E{}{\inner{\bs}{\bbarq_\eta}} \geq K\eta^{2\beta+1}+\min_{i,j}\{s_i^{(1)},s_j^{(2)}\}\frac{\bone_{n \times n}\circ \Sigma^{(1)}+\bone_{m \times m} \circ \Sigma^{(2)}}{8\max\{m,n\}M\eta^\beta}  \\
     &\geq \inf_{\beta<0}\left\{K\eta^{2\beta+1}+\min_{i,j}\{s_i^{(1)},s_j^{(2)}\}\frac{\bone_{n \times n}\circ \Sigma^{(1)}+\bone_{m \times m} \circ \Sigma^{(2)}}{8\max\{m,n\}M\eta^\beta}\right\}=K'\eta^{1/3}. \Halmos
 \end{align*}
\endproof
\section{Cost function and its Variations}
\subsection{Proof of Proposition \ref{prop: cost_function_reformulation}: Cost Function Reformulation}
To prove the Proposition \ref{prop: cost_function_reformulation}, we will need the following two lemmas presented below along with their proofs.
\begin{lemma} \label{lemma: variational_inequality}
Let $\bnu \in \Delta_n^n$,  the following are equivalent:
\begin{enumerate}
    \item $\bnu$ satisfies \eqref{eq: variational_inequality}.
    \item  $\bnu$ satisfies $\sum_{i=1}^n \sum_{l=1}^n u_{il}(\nu_{il}-\tnu_{il}) \geq 0$ for all $\btnu \in \calC$.
    \item There exists $\bkappa \in \bbR^n$ and $\bxi \in \bbR_+^{n \times n}$ such that $u_{il}=\kappa_i-\xi_{il}$ and $\xi_{il}\nu_{il}=0$ for all $i,l \in [n]$.
\end{enumerate}
\end{lemma}
\proof{Proof of Lemma \ref{lemma: variational_inequality}}
We will first show that $1 \Rightarrow 2$.
\begin{align*}
    \sum_{i=1}^n \sum_{l=1}^n u_{il}\nu_{il}&=\sum_{i=1}^n \sum_{l \in [n] : \nu_{il}>0} u_{il}\nu_{il} \\
    &\overset{*}{=}\sum_{i=1}^n \max_{l' \in [n]} \{u_{il'} \}\sum_{l \in [n]: \nu_{il}>0} \nu_{il} \\
    &\overset{**}{=}\sum_{i=1}^n \max_{l' \in [n]} \{u_{il'}\} \\
    &=\sum_{i=1}^n \sum_{l=1}^n \tnu_{il}\max_{l\ \in [n]}\{u_{il'}\} \\
    &\geq \sum_{i=1}^n \sum_{l=1}^n \tnu_{il}u_{il},
\end{align*}
where $(*)$ follows as $\nu_{il}>0$ only when $u_{il}$ is the maximum among all $u_{il'}$ for $l' \in [n]$ and the maximum is unique. In addition, $(**)$ follows as $\sum_{l \in[n]: \nu_{il}>0} \nu_{il}=1$ for all $i \in [n]$. Now we will show that $2 \Rightarrow 1$.

Suppose $\nu_{il}>0$. For a given $l' \in [n]$, define $\btnu$ as follows:
\begin{align*}
    \tnu_{i'r}=\begin{cases}
    \nu_{ir} &\textit{if } i' \neq i \\
    \nu_{il}+\nu_{il'} &\textit{if } i'=i, r=l' \\
    0 &\textit{if } i'=i, r=l \\
    \nu_{ir} &\textit{if } i'=i, r\neq l', r \neq l
    \end{cases}
\end{align*}
Note that $\tnu \in \calC$ and by $2$, we have
\begin{align*}
    \sum_{i'=1}^n \sum_{r=1}^n u_{i'r}(\nu_{i'r}-\tnu_{i'r}) &\geq 0 \\
    \Rightarrow u_{il'}(\nu_{il'}-\tnu_{il'})+u_{il}(\nu_{il}-\tnu_{il}) &\geq 0 \\
    \Rightarrow -u_{il'}\nu_{il}+u_{il}\nu_{il} &\geq 0 \\
    u_{il} &\geq u_{il'}.
\end{align*}
As $l' \in [n]$ is arbitrary, we deduce $u_{il} \geq u_{il'}$ for all $l' \in [n]$. 

Now, we will prove $3 \Rightarrow 1$. For a given $i \in [n]$, let $l \in [n]$ be such that $\nu_{il}>0$. Then we have, $\xi_{il}=0$ by complementary constraint. This gives us $u_{il}=\kappa_i \geq \kappa_i-\xi_{il'}=u_{il'}$ for all $l' \in [n]$ as $\xi_{il'} \geq 0$. This completes the proof.

Now, we will show $1 \Rightarrow 3$. We will show that there exists $\bkappa \in \bbR^n$ and $\bxi \in \bbR^{n \times n}$ such that $3$ is satisfied. Define $\kappa_i\dfn\max_{l' \in [n]}\{u_{il'}\}$ which gives us $\xi_{il}\dfn-u_{il}+\kappa_i$ for all $i,l \in [n]$. Thus, it is trivially true that $\bxi \geq \bzero_{n \times n}$. In addition, if $\nu_{il}>0$ for some $i,l \in [n]$, then $u_{il} \geq u_{il'}$ for all $l' \in [n]$, which implies that $\kappa_i=u_{il}$ and thus, $\xi_{il}=0$. As $i,l$ is arbitrary, we have $\xi_{il}\nu_{il}=0$ for all $i,l \in [n]$. This completes the proof. \hfill $\Halmos$
 \endproof
Next, we can write the supply constraint $G_i(\hatmu_i)=u_i$ in terms of standard inequality constraints and binary variables $\bb \in \{0,1\}^{n \times n}$ by using the following lemma.
\begin{lemma}
\label{lemma: integer_program}
There exists an $M>0$ such that the following constraints are equivalent:
\begin{enumerate}
    \item $G_i\left( \hatmu_{i}\right)=\max_{l \in [n]} \{u_{il}\}$ for all $i \in [n]$.
    \item $G_i\left( \hatmu_{i}\right) \geq u_{il}$, $G_i\left( \hatmu_{i}\right) \leq u_{il}+(1-b_{il})M$, $b_{il} \in \{0,1\}$ for all $i,l \in [n]$ and $\sum_{l=1}^nb_{il}=1$  for all $i \in [n]$.
\end{enumerate}
\end{lemma}
\proof{Proof of Lemma \ref{lemma: integer_program}}
The idea is the following: the first inequality in 2 enforces $G_i(.)$ to be greater than each of the $u_{il}$ and the second inequality along with the constraint $\sum_{l=1}^n b_{il}$ enforces $G_i(.)$ to be less than or equal to the maximum of $u_{il}$. In addition, we can take $M=\sum_{i=1}^n\left(G_i(\inner{\bone_n}{\bmu})\right)$. Now, we make this intuition concrete.

 $1 \Rightarrow 2$. For a given $i \in [n]$, as $G_i( \hatmu_{i})=\max_{l \in [n]}\{u_{il}\}$, we have $G_i(\hatmu_{i})\geq u_{il}$ for all $l \in [n]$. In addition, we also have 
 \begin{align}
     &G_i\left( \hatmu_{i}\right) \leq \max_{l \in [n]}\{u_{il}\} \nonumber \\ \overset{*}{\Leftrightarrow} {}& G_i\left( \hatmu_{i}\right) \leq u_{il}+M\left(1-\mathbbm{1}_{\left\{l=\min\left\{l' \in [n]: \left\{u_{il'}=\max_{l'' \in [n]}\{u_{il''}\}\right\}\right\}\right\}}\right) \ \forall l \in [n] \nonumber \\
     \overset{**}{\Leftrightarrow}{}& G_i\left( \hatmu_{i}\right) \leq u_{il}+M(1-b_{il}), \ \sum_{l=1}^n b_{il}=1.  \label{eq: lemma3}
 \end{align}
 For $(*)$ to hold true, we can pick $M$ to be an upper bound on the left hand side which is $\sum_{i=1}^n G_i(\inner{\bone_n}{\bmu})$. Next, $(**)$ follows by defining $b_{il}\dfn \mathbbm{1}_{\left\{u_{il}=\max_{l' \in [n]}\{u_{il'}\}\right\}}$ if the maximizer is unique. In this case, we will have $\sum_{l=1}^n b_{il}=1$. If the maximizer is not unique, it suffices to have $b_{il}=1$ for any one of the maximizer (in particular, we pick the smallest $l$) and zero for the rest. So $\sum_{l=1}^n b_{il}=1$ still holds. This completes the proof. $2 \Rightarrow 1$ follows from \eqref{eq: lemma3} along with the inequality $G_i( \hatmu_{i})\geq u_{il}$ for all $l \in [n]$. \hfill $\Halmos$
 \endproof
 \proof{Proof of Proposition \ref{prop: cost_function_reformulation}}
 The proof follows by the definition of the cost function given by \eqref{eq: cost_function}, Lemma \ref{lemma: variational_inequality} which shows equivalence of \eqref{eq: variational_inequality} with \eqref{eq: EQ_reformulation}, and Lemma \ref{lemma: integer_program} which shows equivalence of the supply constraint with \eqref{eq: supply_curve_reformulation}.
\hfill $\Halmos$
 \endproof
\subsection{Proof of Proposition \ref{prop: comparison_fluid_models}: Comparing Cost Models}
\proof{Proof} We will first prove the first statement.
 The domain of $c_*^1(\cdot)$ is a subset of the domain of $c_{**}^1(\cdot)$ which is $\bbR^{n}_+$ and they are equal for all $\bmu \in \Omega$. Thus, $c_{**}^1(\bmu) \leq c_{*}^1(\bmu)$ for all $\bmu \in \bbR^n_+$. So, we have
\begin{align*}
    \E{\alpha}{c_*^1(\bmu)} \geq \E{\alpha}{c_{**}^1(\bmu)} \overset{*}{\geq} c_{**}^1(\E{\alpha}{\bmu}) \quad \forall \alpha \in \calP^{EQ},
\end{align*}
where $(*)$ follows as $c_{**}^1(.)$ is convex by definition. Thus, we have
\begin{align*}
    \span R_{*}^1=\max_{\btlambda,\talpha,\btchi} \inner{F(\btlambda)}{\btlambda}-\E{\talpha}{c_*^1(\btmu)} \leq \inner{F(\btlambda)}{\btlambda}-c_{**}^1\left(\E{\talpha}{\btmu}\right)\\
 \textit{subject to,} \quad \tl_j&=\sum_{i=1}^n \tchi_{ij} \ \forall j \in [m] \quad
    \E{\talpha}{\tmu_i}=\sum_{j=1}^m \tchi_{ij} \ \forall i \in [n]  \\
    \tchi_{ij}&=0 \quad \forall (i,j) \notin E, \quad \tchi_{ij} \geq 0 \quad \forall (i,j) \in E.
\end{align*}
Now, substituting $\E{\talpha}{\btmu}=\bmu$ and $c_{**}^1(\bmu)=\inner{G(\bmu)}{\bmu}$ in the above optimization problem, we get
\begin{align*}
    \span \max_{\btlambda,\bmu,\btchi} \inner{F(\btlambda)}{\btlambda}-\inner{G(\bmu)}{\bmu}=R_{**}^1\\
 \textit{subject to,} \quad \tl_j&=\sum_{i=1}^n \tchi_{ij} \ \forall j \in [m] \quad
    \mu_i=\sum_{j=1}^m \tchi_{ij} \ \forall i \in [n]  \\
    \tchi_{ij}&=0 \quad \forall (i,j) \notin E, \quad \tchi_{ij} \geq 0 \quad \forall (i,j) \in E.
\end{align*}
This shows that $R_*^1 \leq R_{**}^1$.  Now, if the condition $c_{il} \geq G_l(\mu^{**,1}_l)-G_i(\mu_i^{**,1})$ for all $i,l \in [n]$ is satisfied, then $\bmu^{**,1} \in \Omega$. Thus, $\talpha=\bmu^{**,1}$ with probability 1 is a feasible solution and we have $R_{**}^1 \leq R_{*}^1$. This completes the proof. 

Next, it is obvious that (2) is true as the feasible region of the optimization problem defining the cost function for $\beta_1$-IC servers contains the feasible region of $\beta_2$-IC servers and their objective functions are identical. \hfill $\Halmos$
 \endproof
\subsection{Strong Duality: Probabilistic Fluid Model}
By solving the primal formulation, we noticed that it is taking more than a day to solve for the case of SD and $\beta$-IC for small values of $\beta$ and it does not scale well with the graph. In addition, as Gurobi is implementing branch and bound, the simulation uses a lot of memory. In this section, we will analyze the dual of the fluid optimization problem and prove that strong duality holds. It turns out that the dual problem is a convex optimization problem and thus, standard optimization methods like gradient descent can be employed to solve this problem.
We will start by presenting the Lagrangian function $L: \bbR^{n+m} \times \bbR^m \times \calP \times \bbR^{|E|} \rightarrow \bar{\bbR}$ with $\bkappa \in \bbR^{m+n}$ as the dual variables. Here, $\calP$ is the set of measures defined on the Borel sigma algebra generated by $\Omega$. The Lagrangian function $ L(\bkappa,(\blambda,\alpha,\bchi))$ is given by
\begin{align*}
  \inner{F(\blambda)}{\blambda}-\E{\alpha}{c(\bmu)}+\sum_{j=1}^n \kappa_j^{(2)}\left(\lambda_j-\sum_{i: (i,j) \in E} \chi_{ij}\right)+\sum_{i=1}^n\kappa_i^{(1)}\left(\E{\alpha}{\mu_i}-\sum_{j : (i,j) \in E} \chi_{ij}\right).
\end{align*}
The domain of the above defined Lagrangian function is $\bbR^{n+m} \times Y$ where $Y \subset \bbR^m \times \calP \times \bbR^{|E|}$ given by
\begin{align*}
    Y=\left\{(\blambda,\alpha,\bchi) \in \bbR^m_+ \times \calP \times \bbR^{|E|}_+ : \int_\Omega d\alpha=1,\alpha \succeq 0\right\}.
\end{align*}
In words, we are imposing the constraints that the arrival rates $\blambda$ and the rate of matching $\bchi$ is non negative and in addition, $\alpha$ is restricted only to a set of probability measures. The Lagrangian function is defined to be $-\infty$ outside its domain by convention as we are maximizing with respect to $(\blambda,\alpha,\bchi)$. Now, the dual function can be written as follows:
\begin{align*}
    D_*=\min_{\bkappa \in \bbR^{m+n}} \left\{\max_{(\blambda,\alpha,\bchi) \in Y} L(\bkappa, (\blambda,\alpha,\bchi))\right\}.
\end{align*}
This can be expanded by substituting the expression on $L$ and then separating the inner maximization to get
\begin{align*}
    \min_{\bkappa \in \bbR^{m+n}}&\left\{\sum_{j=1}^n\max_{\lambda_j \in \bbR_+}\left\{F_j(\lambda_j)\lambda_j+\kappa_j^{(2)}\lambda_j\right\}-\min_{\alpha \in \calP: \int_\Omega d\alpha=1, \alpha \succeq 0}  \E{\alpha}{c(\bmu)-\sum_{i=1}^n\mu_i\kappa_i^{(1)}}\right.
    \\
    &\left.-\sum_{(i,j)\in E}\min_{\chi_{ij} \geq 0}\left\{ \chi_{ij}\left(\kappa_j^{(2)}+\kappa_i^{(1)}\right)\right\}\right\}
\end{align*}
The second minimization can be reduced to minimizing only over all the Dirac measures as taking a convex combination will only increase the objective function value. In addition, if $\kappa_i^{(1)}+\kappa_j^{(2)}<0$, then $\chi_{ij}$ can be taken arbitrarily large which will make the objective function arbitrarily large and, if $\kappa_i^{(1)}+\kappa_j^{(2)} \geq 0$, then the minimization is achieved at $\chi_{ij}=0$. Thus, the above optimization problem can be reduced to the following:
\begin{subequations}
\begin{align}
    \min_{\bkappa \in \bbR^{m+n}}\left\{\sum_{j=1}^n\max_{\lambda_j \in \bbR_+}\left\{F_j(\lambda_j)\lambda_j+\kappa_j^{(2)}\lambda_j\right\}-\min_{\bmu \in \Omega}  \left\{c(\bmu)-\sum_{i=1}^n\mu_i\kappa_i^{(1)}\right\}\right\} \span \\
    \span \textit{subject to,} \ \kappa_i^{(1)}+\kappa_j^{(2)} \geq 0 \quad \forall (i,j) \in E.
\end{align}
\label{eq: dual_CLP}
\end{subequations}
Note that, the above optimization problem is a convex optimization problem with affine constraints as $\max_{\lambda_j \geq 0}\{ F_j(\lambda_j)\lambda_j+\kappa_j^{(2)}\lambda_j^{(2)}\}$ is the conjugate of the function $-F_j(\lambda_j)\lambda_j$ and $-\min_{\bmu \in \Omega}\{ c(\bmu)-\sum_{i=1}^n \mu_i\kappa_i^{(1)}\}$ is the conjugate of the function $c(\bmu)$ and thus it is convex. Next, we show that there is no duality gap between the fluid optimization problem and its dual. The result is presented below:
\begin{proposition} \label{prop: strong_duality}
The optimal values of the problems \eqref{eq: prob_fluid_model} and \eqref{eq: dual_CLP} are equal, that is $R_*=D_*$.
\end{proposition}
\proof{Proof of Proposition \ref{prop: strong_duality}}
We will use the Theorem 7.10 from the lectures on stochastic programming \citep{shapiro2014lectures}. We will verify the following three conditions:
\begin{enumerate}
    \item For every $\bkappa \in \bbR^{n+m}$, the function $L(\bkappa,.)$ is concave.
    \item For every $(\blambda,\alpha,\bchi)$, the function $L(.,(\blambda,\alpha,\bchi))$ is convex and lower semi continuous.
    \item The dual optimization problem \eqref{eq: dual_CLP} has a nonempty and bounded set of optimal solutions.
\end{enumerate}

Proof of 1. If $\bkappa: \kappa_i^{(1)}+\kappa_j^{(2)} \geq 0 \ \forall (i,j)\in E$, then we know that $L(\bkappa,.)$ is concave with respect to $\blambda$ by Assumption \ref{ass: concave}, and affine with respect to $\alpha$ and $\bchi$ in its domain $Y$ which is convex. In addition, if $\bkappa$ does not belong to the above set, then $L(\bkappa,.)=\infty$ which is concave.

Proof of 2. For $(
\blambda,\alpha,\bchi) \in Y$, $L(.,(
\blambda,\alpha,\bchi))$ is an affine function of $\bkappa$, thus it is convex and lower semi continuous. For $(
\blambda,\alpha,\bchi) \notin Y$, we have $L(.,(
\blambda,\alpha,\bchi))=-\infty$ everywhere and thus it is convex and lower semi continuous. 

Proof of 3. We already know that the dual objective function is convex and the constraints are affine. Now, we will show that the objective function is coercive which will suffice to show that the optimal solution is nonempty and bounded. Let $\tl^\star_j>0$ and $\tl^\star_j \in \textrm{dom} F_j$ for all $j \in [m]$ and $\bmu^\star>\bzero_m$ and $\bmu^\star \in \Omega$. Then we have
\begin{align*}
    \sup_{\lambda_j \geq 0}\left\{F_j(\lambda_j)\lambda_j+\kappa_j^{(2)}\lambda_j\right\} &\geq \max\left\{0,F_j(\tl_j^\star)\tl_j^\star+\kappa_j^{(2)}\tl_j^\star\right\} \\
    \inf_{\bmu \geq \bzero_m}\left\{c(\bmu^\star)-\inner{\bkappa^{(1)}}{\bmu^\star}\right\} &\leq \min\left\{0,c(\bmu)-\inner{\bkappa^{(1)}}{\bmu}\right\}
\end{align*}
The above results in the following lower bound on the objective function of the dual \eqref{eq: dual_CLP}.
\begin{align}
    \sum_{j=1}^m \max\left\{0,F_j(\tl_j^\star)\tl_j^\star+\kappa_j^{(2)}\tl_j^\star\right\}-\min\left\{0,c(\bmu^\star)-\inner{\bkappa^{(1)}}{\bmu^\star}\right\}. \label{eq: lower_bound_obj_dual}
\end{align}
Now, if $||\bkappa||\rightarrow \infty$ such that $\kappa_i^{(1)}+\kappa_j^{(2)} \geq 0$ for all $(i,j) \in E$, then there exists an $i \in [n]$ or a $j \in [m]$ such that either $\kappa_i^{(1)}\rightarrow \infty$ or $\kappa_j^{(2)}\rightarrow \infty$. This is true by the assumption that the bipartite graph is connected. This implies that \eqref{eq: lower_bound_obj_dual} $\rightarrow \infty$. Thus, the objective function of the dual \eqref{eq: dual_CLP} $\rightarrow \infty$. Thus, it is coercive. This completes the proof. \hfill $\Halmos$
 \endproof
 \section{Simulations: A Generic City Model} \label{app: simulation}
\begin{figure}[ht]
\FIGURE{
\begin{minipage}{0.6\textwidth}
 \begin{tikzpicture}[scale=0.6]
\draw[black, very thick] (0,0) -- (2,0) -- (2,1) -- (0,1);
\node[black,very thick] at (0.25,0.5) {2 \footnotesize (E)};
\draw[black, very thick] (0,1.5) -- (2,1.5) -- (2,2.5) -- (0,2.5);
\draw[black, very thick] (0,-0.5) -- (2,-0.5) -- (2,-1.5) -- (0,-1.5);
\draw[black, very thick] (0,-2) -- (2,-2) -- (2,-3) -- (0,-3);
\node[black,very thick] at (0.25,2) {1 \footnotesize (N)};
\node[black,very thick] at (0.25,-1) {3 \footnotesize (W)};
\node[black,very thick] at (0.25,-2.5) {4 \footnotesize (S)};
\draw[black,very thick] (8,0) -- (6,0) -- (6,1) -- (8,1);
\node[black,very thick] at (7.75,0.5) {2};
\draw[black, very thick] (8,1.5) -- (6,1.5) -- (6,2.5) -- (8,2.5);
\node[black,very thick] at (7.75,2) {1};
\node[black,very thick] at (7.75,-1) {3};
\node[black,very thick] at (7.75,-2.5) {4};
\node[black,very thick] at (7.75,-4) {5};
\draw[black,very thick] (13,0) -- (11,0) -- (11,1) -- (13,1);
\node[black,very thick] at (12.75,0.5) {2 \footnotesize (N,W)};
\node[black,very thick] at (12.75,-1) {3 \footnotesize(S,E)};
\node[black,very thick] at (12.75,-2.5) {4 \footnotesize(S,W)};
\node[black,very thick] at (12.75,-4) {5 \footnotesize(N,E,W,S)};
\draw[black, very thick] (13,1.5) -- (11,1.5) -- (11,2.5) -- (13,2.5);
\draw[black, very thick] (13,-0.5) -- (11,-0.5) -- (11,-1.5) -- (13,-1.5);
\draw[black, very thick] (13,-2) -- (11,-2) -- (11,-3) -- (13,-3);
\draw[black, very thick] (13,-3.5) -- (11,-3.5) -- (11,-4.5) -- (13,-4.5);
\draw[black, very thick] (13,1.5) -- (11,1.5) -- (11,2.5) -- (13,2.5);
\draw[black, very thick] (8,-0.5) -- (6,-0.5) -- (6,-1.5) -- (8,-1.5);
\draw[black, very thick] (8,-2) -- (6,-2) -- (6,-3) -- (8,-3);
\draw[black, very thick] (8,-3.5) -- (6,-3.5) -- (6,-4.5) -- (8,-4.5);
\node[black,very thick] at (12.75,2) {1 \footnotesize (N,E)};
\draw[black,thick]  (2.75, 2.1) edge[<->]  (5.25, 2.1);
\draw[black,thick]  (2.75, 0.6) edge[<->]  (5.25, 1.9);
\draw[black,thick]  (2.75, 1.9) edge[<->]  (5.25, 0.6);
\draw[black,thick]  (2.75, -0.9) edge[<->]  (5.25, 0.4);
\draw[black,thick]  (2.75, 0.4) edge[<->]  (5.25, -0.9);
\draw[black,thick]  (2.75, -2.4) edge[<->]  (5.25, -1.1);
\draw[black,thick]  (2.75, -1.1) edge[<->]  (5.25, -2.4);
\draw[black,thick]  (2.75, -2.6) edge[<->]  (5.25, -2.6);
\draw[black,thick]  (2.75, 1.7) edge[<->]  (5.25, -3.7);
\draw[black,thick]  (2.75, 0.2) edge[<->]  (5.25, -3.9);
\draw[black,thick]  (2.75, -1.3) edge[<->]  (5.25, -4.1);
\draw[black,thick]  (2.75, -2.8) edge[<->]  (5.25, -4.3);
\node[black, align=center] at (1,3) {\footnotesize Customer};
\node[black, align=center] at (4,3) {\footnotesize \shortstack{Compatible\\Matchings}};
\node[black, align=center] at (7,3) {\footnotesize \shortstack {Server \\  Queue}};
\node[black, align=center] at (12,3) {\footnotesize \shortstack {Server \\ Type}};
\node[black, align=center] at (9.5,3) {\footnotesize \shortstack {Complete \\ Graph}};
\draw[gray,ultra thin]  (10.5, -2.5) edge[-]  (8.5, -4);
\draw[gray,ultra thin]  (10.5, -4) edge[-]  (8.5, -4);
\draw[gray,ultra thin]  (10.5, -1) edge[-]  (8.5, -4);
\draw[gray,ultra thin]  (10.5, 0.5) edge[-]  (8.5, -4);
\draw[gray,ultra thin]  (10.5, 2) edge[-]  (8.5, -4);
\draw[gray,ultra thin]  (10.5, -2.5) edge[-]  (8.5, -2.5);
\draw[gray,ultra thin]  (10.5, -4) edge[-]  (8.5, -2.5);
\draw[gray,ultra thin]  (10.5, -1) edge[-]  (8.5, -2.5);
\draw[gray,ultra thin]  (10.5, 0.5) edge[-]  (8.5, -2.5);
\draw[gray,ultra thin]  (10.5, 2) edge[-]  (8.5, -2.5);
\draw[gray,ultra thin]  (10.5, -2.5) edge[-]  (8.5, -1);
\draw[gray,ultra thin]  (10.5, -4) edge[-]  (8.5, -1);
\draw[gray,ultra thin]  (10.5, -1) edge[-]  (8.5, -1);
\draw[gray,ultra thin]  (10.5, 0.5) edge[-]  (8.5, -1);
\draw[gray,ultra thin]  (10.5, 2) edge[-]  (8.5, -1);
\draw[gray,ultra thin]  (10.5, -2.5) edge[-]  (8.5, 0.5);
\draw[gray,ultra thin]  (10.5, -4) edge[-]  (8.5, 0.5);
\draw[gray,ultra thin]  (10.5, -1) edge[-]  (8.5, 0.5);
\draw[gray,ultra thin]  (10.5, 0.5) edge[-]  (8.5, 0.5);
\draw[gray,ultra thin]  (10.5, 2) edge[-]  (8.5, 0.5);
\draw[gray,ultra thin]  (10.5, -2.5) edge[-]  (8.5, 2);
\draw[gray,ultra thin]  (10.5, -4) edge[-]  (8.5, 2);
\draw[gray,ultra thin]  (10.5, -1) edge[-]  (8.5, 2);
\draw[gray,ultra thin]  (10.5, 0.5) edge[-]  (8.5, 2);
\draw[gray,ultra thin]  (10.5, 2) edge[-]  (8.5, 2);
\end{tikzpicture}
\end{minipage}
\begin{minipage}{0.4\textwidth}
$$
    c=\begin{bmatrix}
    0 & 2 & 2 & 10 & 5 \\
     2 & 0 & 10 & 2 & 5 \\
      2 & 10 & 0 & 2 & 5 \\
       10 & 2 & 2 & 0 & 5 \\
        0 & 0 & 0 & 0 & 0 
    \end{bmatrix} \times \bar{c}
$$
\end{minipage}}
{\centering{
A Generic City Model.}
    \label{fig: generic_city}}{}
\end{figure}
\subsection{IC vs FB: IC}
Motivated by our ridehailing example, we simulate the network given by Fig. \ref{fig: generic_city} with linear supply and demand curves. In particular, the demand curves are $F_1(\lambda_1)=10-\lambda_1/2$, $F_2(\lambda_2)=12-\lambda_3/2$, $F_3(\lambda_3)=12-\lambda_3/2$, $F_4(\lambda_4)=18-\lambda_4$ and the supply curves are $G_1(\mu_1)=3\mu_1-3$, $G_2(\mu_2)=2\mu_2$, $G_3(\mu_3)=\mu_3$, $G_4(\mu_4)=2.5\mu_4$ and $G_5(\mu_5)=\mu_5$. Each type of customer is described by the destination they wish to go and each type of server is described the list of destinations or a single destination they wish to go. The compatibility between a pair of customer and server holds if they wish to go to the same destination. The penalty due to waiting ($\bc$) is given in Fig. \ref{fig: generic_city} and is parametrized by a scalar $\bar{c}$. For a given $(i,l)$ pair, $c_{il}$ is high if the destinations are in the opposite directions and lower otherwise. For example, $c_{14}$ is high as the choices of destination of type 1 server does not match at all with type 2 server. Now, we compare the solution of the fluid model for IC and FB:IC for different values of $\bc$. The result is plotted in Fig. \ref{fig: fluid_vs_IC}.
We can observe that the two optimal solutions are not too different from each other. In addition, we parametrize the supply curve $G_5=i\times\frac{\mu_5}{10}$ and analyze the fluid solution of FB: IC as $i$ varies. The results are plotted in Fig. \ref{fig: fluid_g5} and Fig. \ref{fig: fluid_mu5}.
\begin{figure}
\FIGURE{
     \includegraphics[width=.5\linewidth]{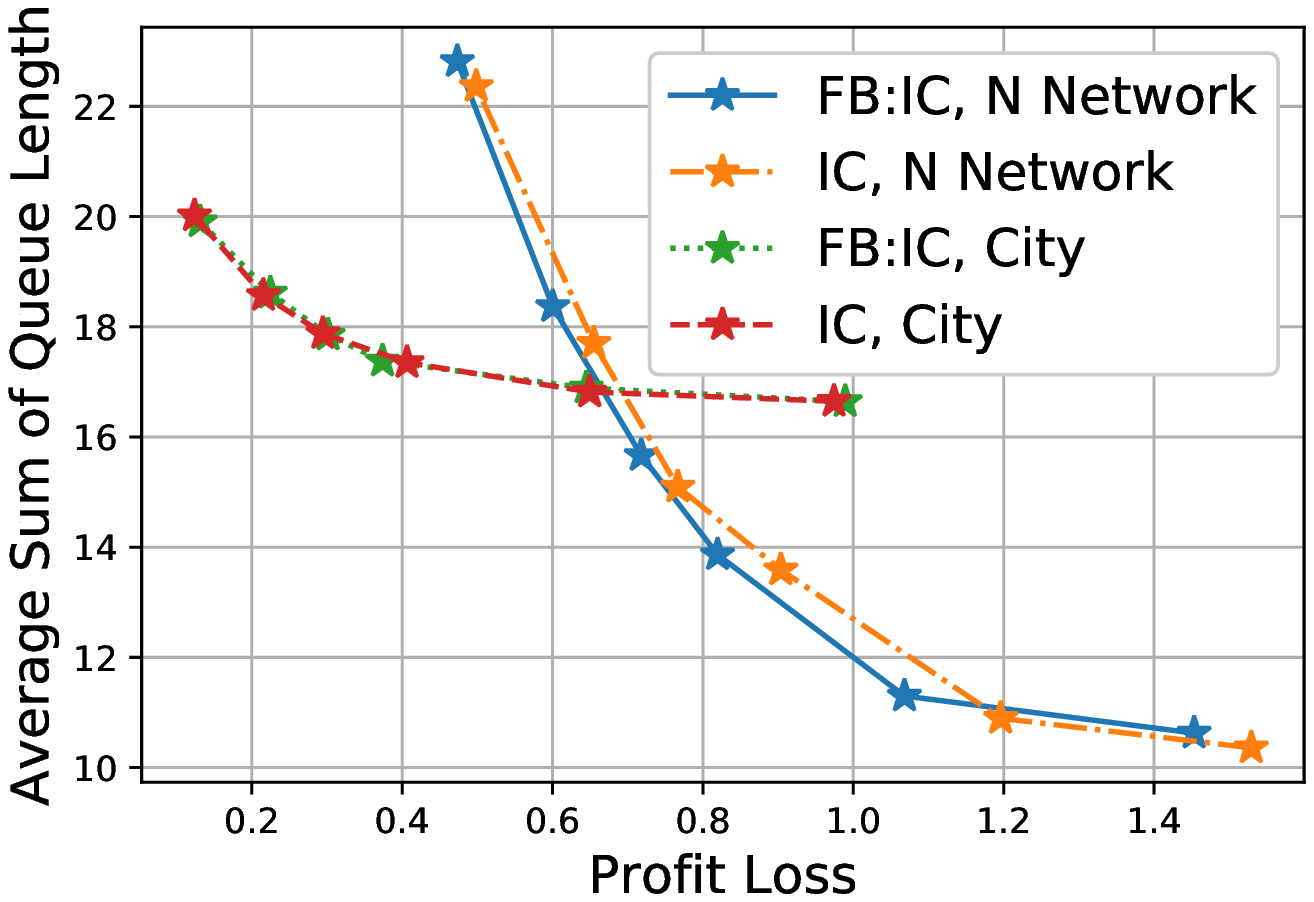}}{
     \centering{Average sum of queue length vs loss in profit for non asymptotic systems with affine supply curves for N-Network and generic city model}
     \label{fig: q_pl_models}}{}
\end{figure}
\begin{figure}
    \begin{minipage}{0.48\textwidth}
     \caption{Optimal objective value of fluid model for FB: IC vs $i$}
     \includegraphics[width=.9\linewidth]{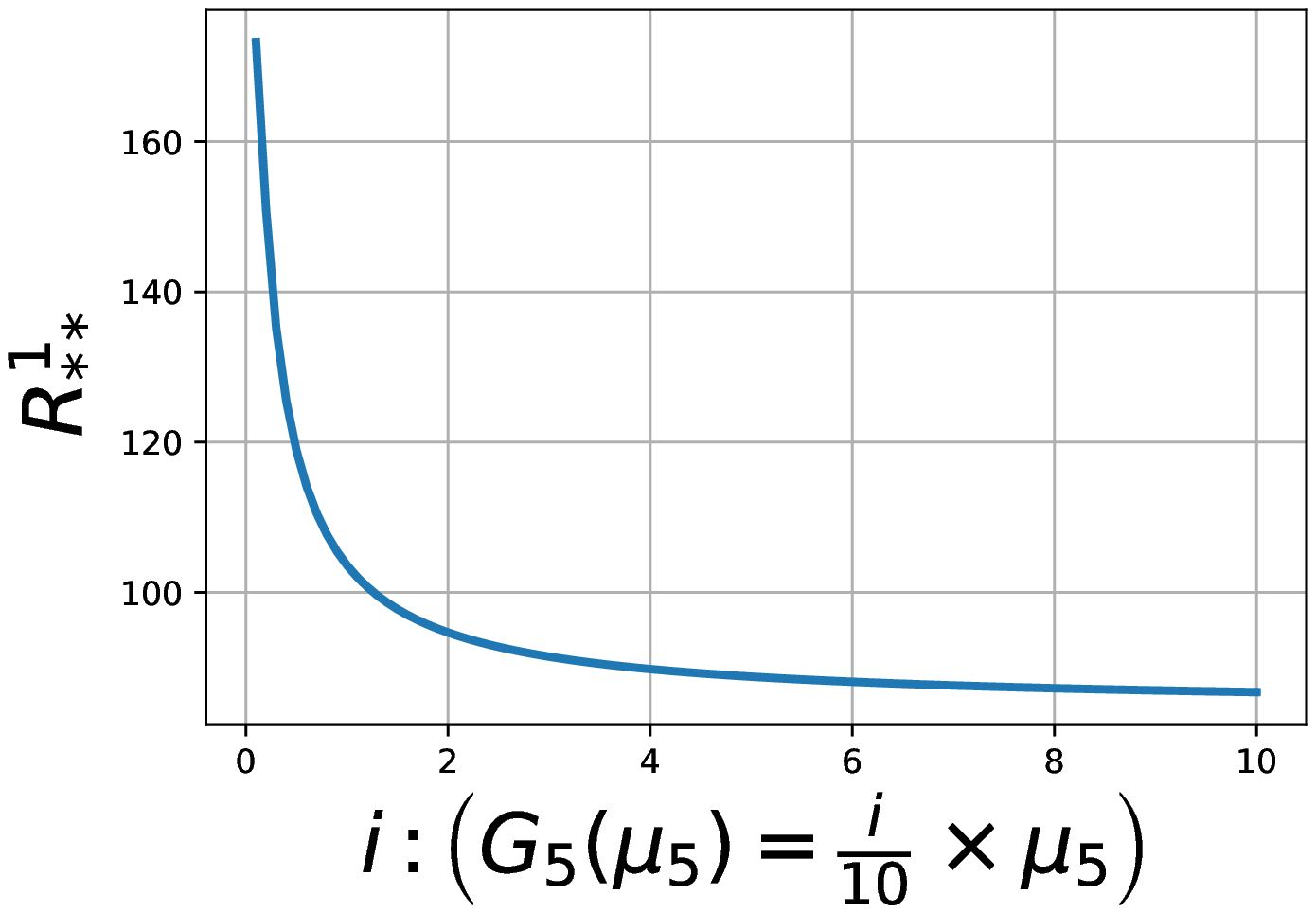}
    \label{fig: fluid_g5}
    \end{minipage}
    \begin{minipage}{0.48\textwidth}
    \caption{Optimal arrival rate and price for type 5 server vs $i$ for FB: IC}
     \includegraphics[width=.9\linewidth]{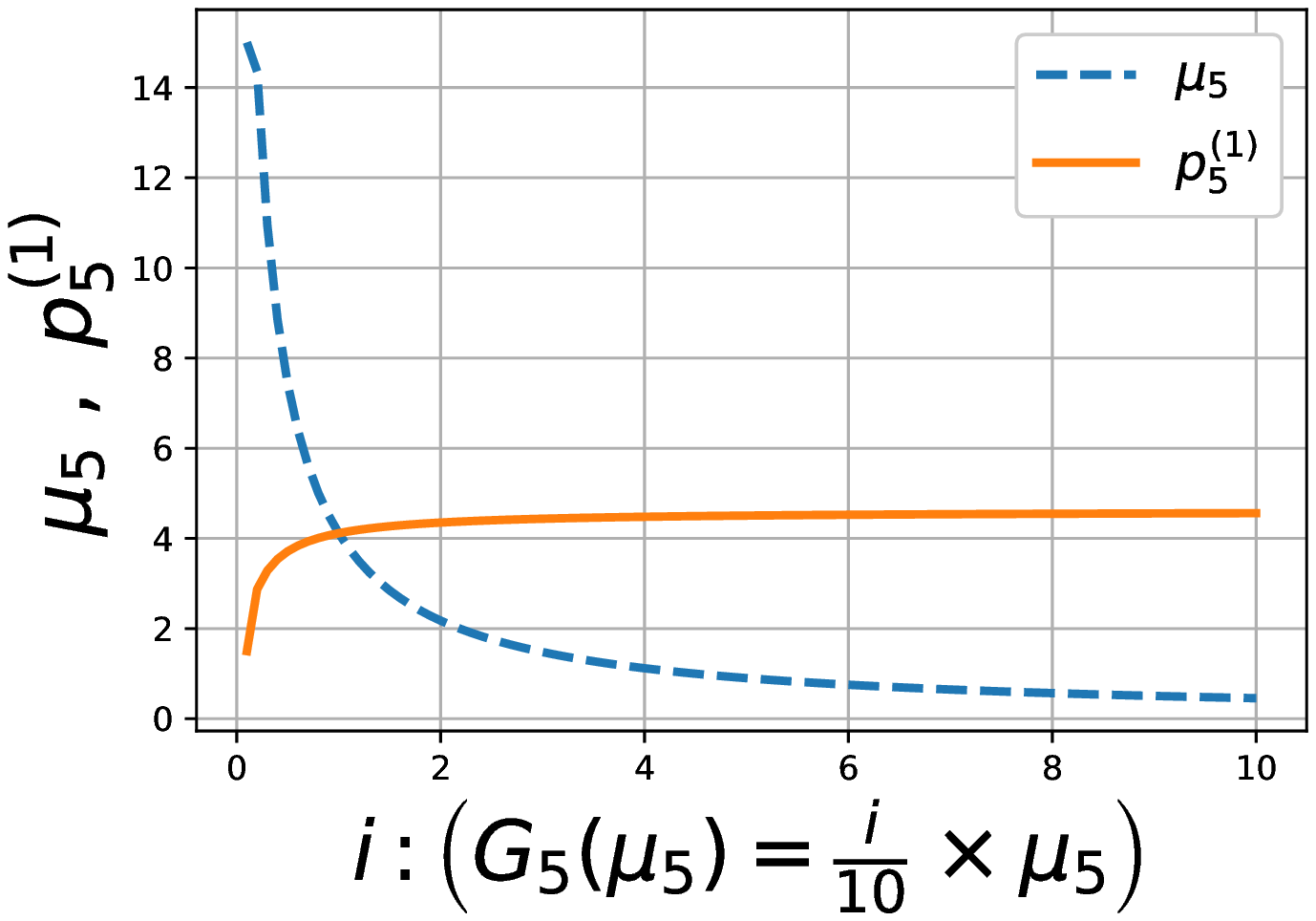}
    \label{fig: fluid_mu5}
    \end{minipage}
\end{figure}
\begin{figure}
    \begin{minipage}{0.48\textwidth}
    \caption{Optimal objective value of IC vs first best: IC vs penalty due to lying}
    \includegraphics[width=0.9\linewidth]{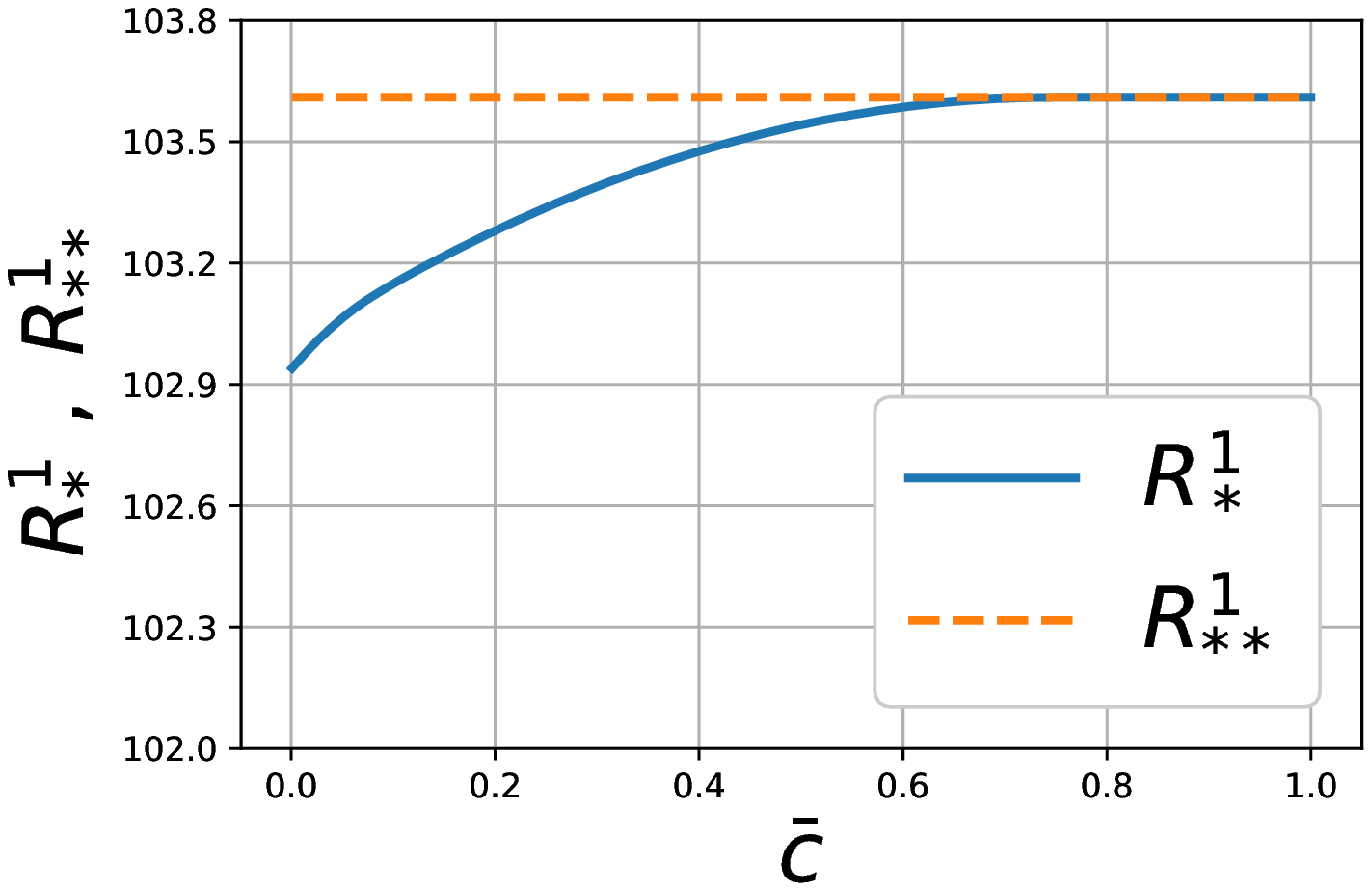}
    \label{fig: fluid_vs_IC}
    \end{minipage}
     \begin{minipage}{0.48\textwidth}
      \caption{Percentage Loss vs $\eta$ for the generic city model under different cost function models}
    \includegraphics[width=0.9\linewidth]{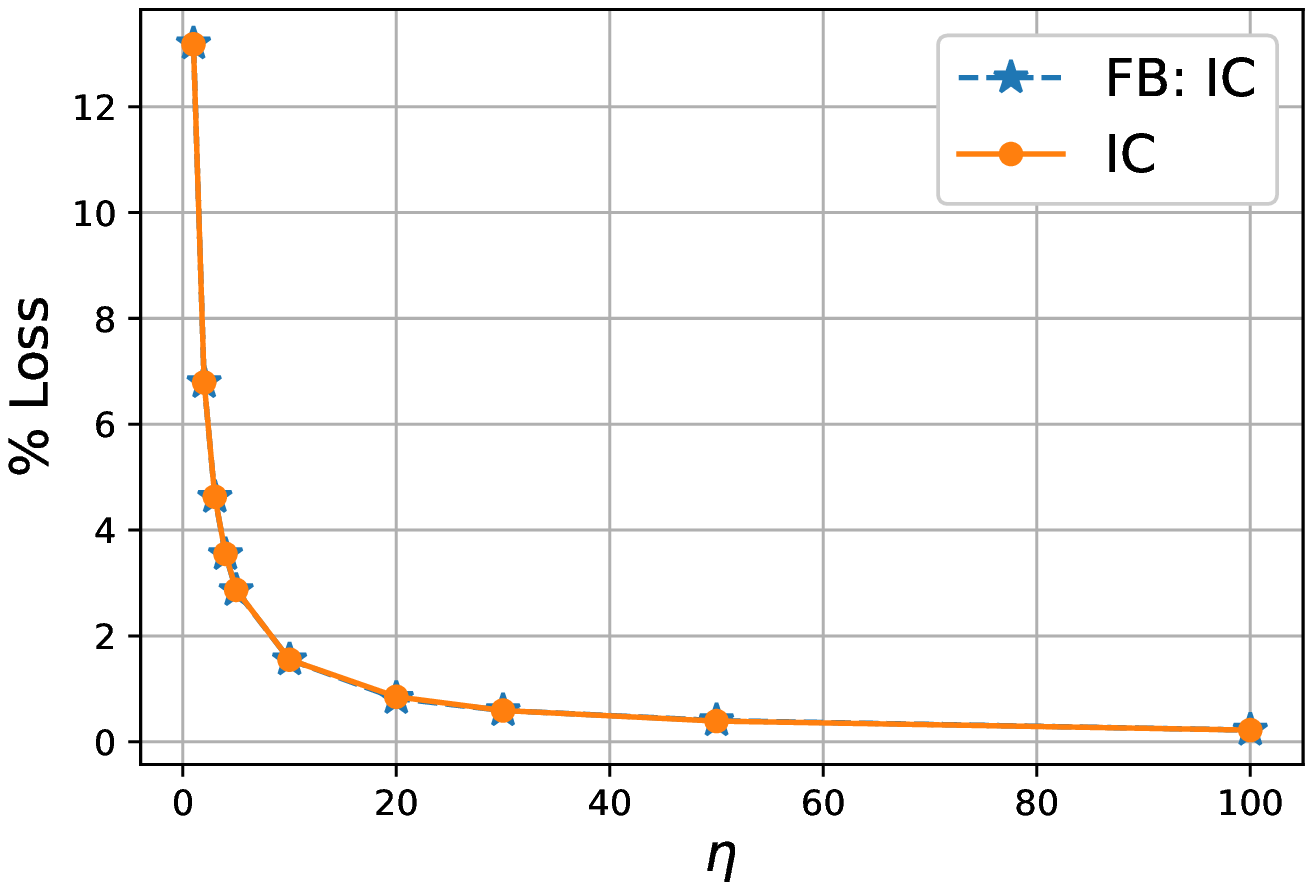}
   \label{fig: big_g_pre_limit}
    \end{minipage}
    \end{figure}
    \subsection{Stochastic Simulations}
We analyze the proposed two price policy and max weight matching policy for $\eta=\{1,5,10,20,30,50,100\}$ and report the percentage loss compared to the fluid upper bound for the case of IC and FB: IC. We pick $G_5=\mu_5/5$ and $\bar{c}=0$. The results are summarized in Fig. \ref{fig: big_g_pre_limit}. Thus, we conclude that the percentage loss is also robust for the choice of the bipartite graph and decays very fast as $\eta$ increases.
Finally, we also simulate the system free of any asymptotic regime for the generic city model and plot the trade off between the average sum of queue length and profit loss in Fig. \ref{fig: q_pl_models}. The trade-off is not sensitive to the cost function but it is sensitive to the size of the network.
\section{Extensions}
\subsection{Proof of Theorem \ref{theo: lower_bound}: Quality Driven View}
\proof{Proof of Theorem \ref{theo: lower_bound}}
In this proof, for a given pricing and matching policy $\pi \in \StablePol$, we will use Theorem \ref{theo: lowerbound_queue} to infer that $\E{}{\inner{\bone_{n+m}}{\bbarq}}\sim 1/\epsilon$ and use Theorem \ref{theo: lowerbound_profit_loss} to infer that $L(\pi) \sim \epsilon^2$. Lastly, if $\E{}{\inner{\bone_{n+m}}{\bbarq}}=C$, we obtain that $\epsilon \sim 1/C$ which shows that $L(\pi) \sim 1/C^2$. Now, we will make this intuition concrete.

By Theorem \ref{theo: lowerbound_queue}, there exists an $\epsilon_0>0$ such that for all $\epsilon<\epsilon_0$, we have
\begin{align*}
    \E{}{\inner{\bone_{n+m}}{\bbarq}}  \geq \frac{\bone_{n \times n}\circ \Sigma^{(1)}+\bone_{m \times m} \circ \Sigma^{(2)}_{\min}}{8\max\{m,n\}\epsilon}.
\end{align*}
As, $\E{}{\inner{\bone_{n+m}}{\bbarq}}=C$, we have
\begin{align}
    \epsilon \geq \frac{\bone_{n \times n}\circ \Sigma^{(1)}+\bone_{m \times m} \circ \Sigma^{(2)}_{\min}}{8\max\{m,n\}}\frac{1}{C}. \label{eq: lower_bound_epsilon}
\end{align}
Now, to use Theorem \ref{theo: lowerbound_profit_loss}, we will construct a sequence of pricing and matching policies $\pi_\eta \in \StablePol$ which satisfies the hypothesis of Theorem \ref{theo: lowerbound_profit_loss}. The sequence $\pi_\eta$ is such that $\pi_\eta=\pi$ for some $\eta>\eta_1$ where $\eta_1$ is a parameter such that Theorem \ref{theo: lowerbound_profit_loss} holds for all $\eta>\eta_1$. This will give us a bound on $L(\pi_\eta)=L(\pi)$. 

In particular, consider the sequence of pricing policies of the form \eqref{eq: general_pricing_policy}. Now, we will specify the parameters $\alpha,\beta,\{\phi_j(\cdot)\}_{j \in [m]}$ which will completely specify $\pi_\eta$. Let $\alpha=0,\beta=-0.5$ and $\{\phi_j(\cdot)\}_{j \in [m]}$ given by
\begin{align*}
    \phi_j(\bq)=\frac{\lambda_j(\bq)-\tl_j^\star}{\epsilon} \quad \forall j \in [m], \bq \in S.
\end{align*}
Now, we will verify Assumption \ref{condition: general_pricing} holds. In particular, Assumption \ref{condition: general_pricing} \ref{condition: bounded} holds with $M=1$ as follows:
\begin{align*}
    |\phi_j(\bq)|=\frac{|\lambda_j(\bq)-\tl_j^\star|}{\epsilon} \leq 1 \quad \forall j \in [m].
\end{align*}
Next, Assumption \ref{condition: general_pricing} \ref{condition: alpha_beta} holds by our choice of $\alpha$ and $\beta$. Lastly, we verify Assumption \ref{condition: general_pricing} \ref{condition: stability} holds. Consider a $\bq \in S$ such that either $q_i^{(1)}>K$ or there exists $j \in [m]$ such that $(i,j) \in E$ and $q_j^{(2)}>K$, then we have
\begin{align*}
    |\phi_j(\bq)|=\frac{|\lambda_j(\bq)-\tl_j^\star|}{\epsilon} \geq \frac{\sigma\epsilon}{\epsilon}> \sigma \quad \forall j \in [m].
\end{align*}
%Consider $C_0=\frac{\bone_{n \times n}\circ \Sigma^{(1)}+\bone_{m \times m} \circ \Sigma^{(2)}_{\min}}{8\max\{m,n\}}\frac{2}{\epsilon_1}$ such that an $\epsilon<\epsilon_1$ exists that satisfies \eqref{eq: lower_bound_epsilon} for all $C>C_0$. 
Now, note that the pricing policy for $\pi_{\epsilon^{-2}}$ is same as the pricing policy for $\pi$ as we have
\begin{align*}
    \tl^\star_j+\phi_j(\bq)(\epsilon^{-2})^{-0.5}=\tl^\star_j+\phi_j(\bq)\epsilon=\lambda_j(\bq) \quad \forall j \in [m].
\end{align*}
Finally, define $\epsilon_1=\min\{\epsilon_0,(1+\eta_1)^{-0.5}\}$ so that for any $\epsilon<\epsilon_1$, we have $\epsilon^{-2}>1+\eta_1$. Thus, for all $\epsilon<\epsilon_1$, by Theorem \ref{theo: lowerbound_profit_loss}, we get
\begin{align*}
    \tilde{R}^\star-P(\pi)=\frac{\tilde{R}^\star_{\epsilon^{-2}}-P_{\epsilon^{-2}}(\pi_{\epsilon^{-2}})}{\epsilon^{-2}} \geq  K\epsilon^2 \geq K\left(\frac{\bone_{n \times n}\circ \Sigma^{(1)}+\bone_{m \times m} \circ \Sigma^{(2)}_{\min}}{8\max\{m,n\}}\right)^2\frac{1}{C^2}
\end{align*}
This completes the first part of the proof. Now, we will prove the theorem for two price policy and max-weight matching policy which we denote by $\pi$. We will use Lemma \ref{lemma: stability}, Theorem \ref{theo: lowerbound_queue} and Lemma \ref{lemma: profit_loss}. First, by Lemma \ref{lemma: stability}, with $\epsilon_\eta=\epsilon'$, we have
\begin{align*}
    \E{}{\inner{\bs}{\bbarq}}\leq \frac{B}{\epsilon'}.
\end{align*}
As, $\E{}{\inner{\bone_{n+m}}{\bbarq}}=C$, we have
\begin{align*}
    \epsilon' \leq \frac{B}{\min_{i,j}\{s_i^{(1)},s_j^{(2)}\}C}.
\end{align*}
Now, denote by $\pi_\eta$ the sequence of two-price policy and max-weight matching policy. As $\epsilon_\eta \downarrow 0$, there exists $\eta$ such that $\epsilon_\eta=\epsilon'$. Thus, by Lemma \ref{lemma: profit_loss}, we have
\begin{align*}
    L_P(\pi)\overset{*}{=}\frac{L_\eta^P(\pi_\eta)}{\eta}&= -(\epsilon')^2\sum_{j=1}^m \left(\frac{\tl^\star_jF''(\tl^\star_j)}{2}+F'_j(\tl_j^\star)\right)+O\left((\epsilon')^3\right)  \\
   &\overset{**}{ \leq}-\sum_{j=1}^m \left(\frac{\tl^\star_jF''(\tl^\star_j)}{2}+F'_j(\tl_j^\star)\right)\left(\frac{B}{\min_{i,j}\{s_i^{(1)},s_j^{(2)}\}}\right)^2 \frac{1}{C^2}+O\left(\frac{1}{C^3}\right)
\end{align*}
where, $(*)$ follows by Definition \ref{defn: asymptotic_regime} and as $\sum_{j=1}^m \left(\frac{\tl^\star_jF''(\tl^\star_j)}{2}+F'_j(\tl_j^\star)\right)<0$, $(**)$ follows. This completes the proof. \hfill $\Halmos$
 \endproof
 \subsection{Proof of Proposition \ref{prop: waiting_time}: Waiting Time Model}
\proof{Proof of Proposition \ref{prop: waiting_time}}
Consider a sequence of policies $\pi_\eta \in \StablePol$ which satisfies the hypothesis given in Proposition \ref{prop: waiting_time}. By Little's Law, we have
\begin{align*}
    \E{}{\barq_j^{(2)}}&=\eta \E{}{\lambda_j(\bbarq)}\E{}{\barw_j^{(2)}} \quad \forall j \in [m] \\
    \E{}{\barq_i^{(1)}}&=\eta \E{}{\E{\alpha_\bbarq}{\mu_i}}\E{}{\barw_i^{(1)}} \quad \forall i \in [n]
\end{align*}
The factor of $\eta$ on the RHS is due to the definition of asymptotic regime (Definition \ref{defn: asymptotic_regime}). By using the form of the pricing policy given by \eqref{eq: general_pricing_policy}, we get
\begin{align*}
    &\E{}{\inner{\bs}{\bbarw}} \\
    \overset{(a)}{\geq}{}& \frac{1}{\eta}\left(\sum_{j=1}^m \frac{s_j^{(2)}}{\tl_j^\star+M\eta^\beta}\E{}{\barq_j^{(2)}}+\sum_{i=1}^n \frac{s_i^{(1)}}{\E{\talpha^\star}{\mu_i}}\E{}{\barq_i^{(1)}}\right) \\
    \overset{(b)}{\geq}{}& \min_{i \in [n], j \in [m]}\left\{\frac{2s_j^{(2)}}{3\tl_j^\star},\frac{s_i^{(1)}}{\E{\talpha^\star}{\mu_i}}\right\} \frac{\E{}{\inner{\bone_{n+m}}{\bbarq}}}{\eta} \quad \forall \eta>\left(\frac{\tl_j^\star}{2M}\right)^{1/\beta} \\
    \overset{(c)}{\geq}{}& \min_{i \in [n], j \in [m]}\left\{\frac{2s_j^{(2)}}{3\tl_j^\star},\frac{s_i^{(1)}}{\E{\talpha^\star}{\mu_i}}\right\}\frac{\bone_{n \times n}\circ \Sigma^{(1)}(\talpha^\star)+\bone_{m \times m} \circ \Sigma^{(2)}_{\min}}{8\max\{m,n\}M} \eta^{-\beta-1} \quad \forall \eta>\left(\frac{\tl_j^\star}{2M}+\frac{\epsilon_0}{M}\right)^{1/\beta},
\end{align*}
where $(a)$ follows by Assumption \ref{condition: general_pricing} \ref{condition: bounded}, i.e. $\phi_j(\cdot)$ is uniformly bounded by $M$ for all $j \in [m]$. Next, $(b)$ follows by upper bounding $\eta^{\beta}$ for all $\eta>\left(\frac{\tl^\star_j}{2M}\right)^{1/\beta}$. Lastly, $(c)$ follows by Theorem \ref{theo: lowerbound_queue} with $\epsilon=M\epsilon^\beta$. Now, by using Theorem \ref{theo: lowerbound_profit_loss}, for all $\eta>\max\{\eta_1, \left(\frac{\tl_j^\star}{2M}+\frac{\epsilon_0}{M}\right)^{1/\beta}\}$ we get 
\begin{align*}
     L_\eta(\pi_\eta)&=R^\star_\eta-P_\eta(\pi_\eta)-\E{}{\inner{\bs}{\bbarw_\eta}}   \\
     &\geq \inf_{\beta<0}\left\{K\eta^{2\beta+1}+\min_{i \in [n], j \in [m]}\left\{\frac{2s_j^{(2)}}{3\tl_j^\star},\frac{s_i^{(1)}}{\E{\talpha^\star}{\mu_i}}\right\}\frac{\bone_{n \times n}\circ \Sigma^{(1)}(\talpha^\star)+\bone_{m \times m} \circ \Sigma^{(2)}_{\min}}{8\max\{m,n\}M} \eta^{-\beta-1}\right\}\\
     &=K_w\eta^{-1/3},
 \end{align*}
 where the last equality follows by equating $2\beta+1=-\beta-1$ which results in the best trade off between profit and waiting time. This proves the first part of the proof. Now, let $\pi_\eta$ be the sequence of two-price policy and max-weight matching policy. Then, we have
 \begin{align*}
    \E{}{\inner{\bs}{\bbarw}}
    \leq{}& \frac{1}{\eta}\left(\sum_{j=1}^m \frac{s_j^{(2)}}{\tl_j^\star-\epsilon_\eta}\E{}{\barq_j^{(2)}}+\sum_{i=1}^n \frac{s_i^{(1)}}{\E{\talpha^\star}{\mu_i}}\E{}{\barq_i^{(1)}}\right) \\
    \leq{}& \max_{i \in [n], j \in [m]}\left\{\frac{1}{\tl_j^\star-1},\frac{1}{\E{\talpha^\star}{\mu_i}}\right\} \frac{\E{}{\inner{\bs}{\bbarq}}}{\eta} \\
    \leq{}&B\max_{i \in [n], j \in [m]}\left\{\frac{1}{\tl_j^\star-1},\frac{1}{\E{\talpha^\star}{\mu_i}}\right\} \frac{1}{\epsilon_\eta \eta},
\end{align*}
where the last inequality follows by Lemma \ref{lemma: stability}. Now, by Lemma \eqref{lemma: profit_loss}, we get
\begin{align*}
    L_\eta(\pi_\eta) \leq -\eta\epsilon_\eta^2\sum_{j=1}^m \left(\frac{\tl^\star_jF''(\tl^\star_j)}{2}+F'_j(\tl_j^\star)\right)+O\left(\eta\epsilon_\eta^3\right) +B\max_{i \in [n], j \in [m]}\left\{\frac{1}{\tl_j^\star-1},\frac{1}{\E{\talpha^\star}{\mu_i}}\right\} \frac{1}{\epsilon_\eta \eta}=O(\eta^{-1/3}),
\end{align*}
where the last equality follows by picking $\epsilon_\eta=\eta^{-2/3}$ to optimize the trade-off between $\eta\epsilon_\eta^2$ and $1/(\eta\epsilon_\eta)$. This completes the proof. \hfill
$\Halmos$
\endproof
 \subsection{General Utility Function} \label{sec: general_utility}
\subsubsection{Probabilistic Fluid Model} \label{app: prob_fluid_model_general}
The probabilistic fluid model with the re-defined utility function is similar to \eqref{eq: prob_fluid_model} and is defined below:
\begin{subequations}
\label{eq: general_prob_fluid_model}
\begin{align}
     \tilde{R}^\star_e\triangleq \lefteqn{\max_{(\btlambda,\talpha,\btchi): \talpha(\Omega \backslash \Omega(\btchi))=0} \inner{F(\btlambda)}{\btlambda}-\E{\talpha}{\cost(\btmu,\btchi)}}\\
 \textit{subject to} \quad \tl_j&=\sum_{i=1}^n \tchi_{ij} \quad \forall j \in [m] \\
    \E{\talpha}{\tmu_i}&=\sum_{j=1}^m \tchi_{ij} \quad \forall i \in [n] \\
    \tchi_{ij}&=0 \quad \forall (i,j) \notin E, \quad \tchi_{ij} \geq 0 \quad \forall (i,j) \in E.
\end{align}
\end{subequations}
Here, $c(\cdot)$ is a function of the fluid server arrival rates $\btmu$ and fluid matching rates $\btchi$ and the probability distribution $\talpha$ is defined on the space of fluid server arrival rates $\Omega$. Now, we will show that Proposition \ref{prop: fluid_model} still holds for this general setting. 
\begin{proposition} \label{prop: generalized_fluid_model}
Let $\pi=(\blambda(\cdot),\alpha(\cdot),\bx(\cdot))$ be a feasible solution of \eqref{eq:opt_stoch_extended}, then 
\begin{equation*}
R(\pi)\leq P(\pi)
%\E{\bbarq}{\inner{F(\blambda(\bbarq))}{\blambda(\bbarq)}}-\E{\alpha(\bbarq)}{\cost(\bmu)}
    \leq \tilde{R}^\star_e.
\end{equation*}
That is, the fluid profit is an upper bound for the stochastic profit and net profit. 
\end{proposition}
We use the following key lemma in the proof which is presented below and the details of the proof of Proposition \ref{prop: generalized_fluid_model} is deferred to Appendix \ref{app: generalized_fluid_model}.
\begin{lemma} \label{lemma: necessary_constraints_extended}
For a given stationary (w.r.t. $\bq_e$) Markovian pricing and matching policy $(\blambda(\cdot),\alpha(\cdot),\bx(\cdot))$, let $\btlambda=\E{}{\blambda(\bbarq_e)}$, $\talpha_e(A)=\alpha_e(A \times \calS)$ for all Borel $A \subseteq \Omega$ and $\btchi=\E{}{\bx(\bbarq_e)}$. If $(\blambda(\cdot),\alpha(\cdot),\bx(\cdot))\in \StablePol_e$, $\E{}{\inner{\bone_{n+m}}{\bbarq}}<\infty$ and $\bx(\cdot)$ satisfies \eqref{eq: matching_constraints},
then $(\btlambda,\talpha_e,\btchi)$ is feasible in the probabilistic fluid problem \eqref{eq: general_prob_fluid_model}.
%the constraints of the optimization problem \eqref{eq: prob_fluid_model} are necessary for stability.
\end{lemma}
\subsubsection{Multiple Pricing Policy and Random Matching}
 \proof{Proof of Corollary \ref{corollary: optimality_nash_equilibrium}}
Using Lemma \ref{lemma: stability} on the secondary queues, we know that the system is stable and the expected queue length is finite. Now, by Lemma \ref{lemma: necessary_constraints} and the one to one compatibility structure of the bipartite graph formed by the secondary queues, we have $\E{}{\bary_{ij}}=\tchi_{ij,e}^\star$ for all $(i,j) \in E$. In particular, for a given $(i,j) \in E$, let $r$ and $d$ be such that $ir$ and $jd$ are compatible. Then, by taking expectation of \eqref{eq: multiple_pricing_policy}, we get
\begin{align}
   \E{}{\lambda_{jd,\eta}(\bbarq_e)}=\E{}{\E{\alpha_{\bq_e,\eta}}{\mu_{ir}}}&=\E{}{\bary_{ij}} \quad \forall (i,j) \in E \nonumber \\
    \Rightarrow \tchi_{ij,e}^\star+\epsilon_\eta\left(\P{}{q_{jd}^{(2)}=0}-\P{}{q_{jd}^{(2)}>0}\right)=\tchi_{ij,e}^\star&=\E{}{\bary_{ij}} \quad \forall (i,j) \in E. \label{eq: zero_firstorderterm}
\end{align}
In addition, as we have $\alpha_{\bq_e,\eta}=\talpha^\star$ for all $\bq_e \in S_e$, the system is in equilibrium. Now, we will show optimality. Applying Lemma \ref{lemma: stability} on the secondary queues, we have
 \begin{align*}
     \E{}{\inner{\bone_{|E|}}{\bbarq_e}}=\E{}{\inner{\bone_{n+m}}{\bbarq}} \leq B \frac{1}{\epsilon_\eta} 
 \end{align*}
 for some $B>0$. Now, we will show the profit loss is of the order $\eta \epsilon_\eta^2$ similar to the proof of Lemma \ref{lemma: profit_loss}. By Taylor's series expansion, we have
 \begin{align*}
    &\frac{ L_\eta^P(\pi_\eta)}{\eta}\\
    ={}&\left(\sum_{j=1}^m F_j(\tl_{j,e}^\star)\tl_{j,e}^\star-\E{\talpha^\star_e}{c(\bmu,\btchi^\star_e)}\right)-\E{}{\sum_{j=1}^m F_j\left(\lambda_{j,\eta}(\bbarq_e)\right)\lambda_{j,\eta}(\bbarq_e)-\E{\talpha^\star_e}{c\left(\bmu,\E{}{\bbarx}\right)}} \\
    ={}&\sum_{j=1}^m F_j(\tl_{j,e}^\star)\tl_{j,e}^\star-\E{}{\sum_{j=1}^m F_j\left(\lambda_{j,\eta}(\bbarq_e)\right)\lambda_{j,\eta}(\bbarq_e)} \\
    ={}& \sum_{j=1}^m F_j(\tl_{j,e}^\star)\tl_{j,e}^\star\\
    &-\E{}{\sum_{j=1}^m \left(\tl_{j,e}^\star-\epsilon_\eta \sum_{l=1}^{|N(j)|} \mathbbm{1}_{\left\{q_{jd}^{(2)}>0\right\}}+\epsilon_\eta \sum_{l=1}^{|N(j)|} \mathbbm{1}_{\left\{q_{jd}^{(2)}=0\right\}}\right)F_j\left(\tl_{j,e}^\star-\epsilon_\eta \sum_{l=1}^{|N(j)|} \mathbbm{1}_{\left\{q_{jd}^{(2)}>0\right\}}+\epsilon_\eta \sum_{l=1}^{|N(j)|} \mathbbm{1}_{\left\{q_{jd}^{(2)}=0\right\}}\right)} \\
    \overset{*}{=}{}& \epsilon_\eta \sum_{j=1}^m\left(\sum_{l=1}^{|N(j)|} \P{}{q_{jd}^{(2)}=0}-\P{}{q_{jd}^{(2)}>0}\right)\left(\tl^\star_{j,e} F_j'(\tl^\star_{j,e})+\tl^\star_{j,e}\right)+O(\epsilon_\eta^2)\overset{**}{=}O(\epsilon_\eta^2),
 \end{align*}
 where $(*)$ follows by the Taylor's series expansion and the higher order terms are denoted by $O(\epsilon_\eta^2)$ and $(**)$ follows by \eqref{eq: zero_firstorderterm}. This completes the proof.
 \endproof
\subsubsection{Two Pricing Policy and Improved Random Matching} \label{app: improved_random_matching}
Max-weight in the context of two-sided queues is equivalent to join the shortest queue (JSQ) algorithm in the context of load balancing. In addition, random matching in both these contexts are equivalent. The improved random matching we will present in this section is equivalent to join the idle queue (JIQ) in the context of load balancing. It is known in the literature that JIQ is better than random matching but worse than JSQ. In each decision epoch, we match $y_{ij}(k)$ number of $i$ type of servers to $j$ type of customers in expectation where $y_{ij}(k)$ is defined as follows:
\begin{align*}
    y_{ij}(k)=\min\left\{\frac{\tchi_{ij,e}^\star}{\E{\talpha^\star_e}{\mu_i}}(q_i^{(1)}(k)+a_i^{(1)}(k)), \frac{\tchi_{ij,e}^\star}{\tl_{j,e}^\star}(q_j^{(2)}(k)+a_j^{(2)}(k))\right\}.
\end{align*}
In particular, we divide the customers and servers waiting in the queues proportional to the fluid solution and then match as many as possible. This is precisely the feasible solution that we have used to prove Lemma \ref{lemma: negative_drift}. Thus, the optimality of two price policy with improved random matching will follow similarly and we omit the details for brevity. In addition, note that $\E{}{y_{ij}}$ is at least $\tchi_{ij}^\star$ if either one of $q_i^{(1)}$ or $q_j^{(2)}$ is larger than a constant. Thus, to show equilibrium, we need to analyze $\P{}{\max\{q_i^{(1)},q_j^{(2)}\}>c}$ and show that it converges to 1 as $\eta \rightarrow \infty$. This is intuitive as the queue lengths scale as $\eta^{1/3}$. Although, we do not have handle on the individual queue lengths and this is the major challenge to show equilibrium.
\subsubsection{Proof of Proposition \ref{prop: generalized_fluid_model}} \label{app: generalized_fluid_model}
We omit the details of the proof of Lemma \ref{lemma: necessary_constraints_extended} as it follows similarly as the proof of Lemma \ref{lemma: necessary_constraints}. Now, by Proposition \ref{prop: equivalence_joint_marginal} the optimization problem \eqref{eq:opt_stoch_extended} can be equivalently re-written as follows:
\begin{subequations}
\begin{align}
    R^\star_e= \sup_{(\blambda(\cdot),\alpha_e(\cdot),\bx(\cdot))\in \StablePol_e} \E{\bbarq_e}{\inner{F(\blambda(\bbarq))}{\blambda(\bbarq)}}-\E{\alpha_e}{\cost(\bmu,\E{\bbarq_e}{\bbarx})}-\E{\bbarq_e}{\inner{\bs}{\bbarq}} \span \\
    \textit{subject to,} \quad &\blambda(\bq_e)\in \bbR^m_+,\quad \forall \bq_e\in S_e\\
    &\bx(\cdot)\quad \textit{satisfies} \:\: \eqref{eq: matching_constraints}\\
     &\alpha_e(\Omega(\E{}{\bbarx}) \times B_e)=\sum_{\bbarq_e \in B_e} \psi_e(\bbarq_e) \quad \forall B_e \subseteq S_e  \\
     &\psi_e(\cdot)=\textit{Stationary Distribution} (\blambda(\cdot),\alpha(\cdot),\bx(\cdot)).
\end{align}
\label{eq:opt_stoch_equivalent_extended}
\end{subequations}
 \proof{Proof of Proposition \ref{prop: generalized_fluid_model}}
 Consider a feasible policy $\pi=(\blambda(\cdot),\alpha(\cdot),\bx(\cdot))$ under which we have a stable Markov chain $\{\bq_e(k): k \in \bbZ_+\}$ as the policy is a stationary policy w.r.t. $\bq_e \in S_e$. Now, we will upper bound the net profit under $\pi$. We have
 \begin{align*}
     \lefteqn{\E{\bbarq_e}{\inner{F(\blambda(\bbarq_e))}{\blambda(\bbarq_e)} }-\E{\alpha_e}{c(\bmu, \E{\bbarq_e}{\bbarx})}-\E{\bbarq_e}{\inner{\bs}{\bbarq}} }\\
     &\leq\E{\bbarq_e}{\inner{F(\blambda(\bbarq_e))}{\blambda(\bbarq_e)} }-\E{\alpha_e}{c(\bmu, \E{\bbarq_e}{\bbarx})} \\
     &\leq\inner{F(\btlambda)}{\btlambda}-\E{\alpha_e}{c(\bmu, \E{\bbarq_e}{\bbarx})} \\
    &=\inner{F(\btlambda)}{\btlambda}-\E{\alpha_e}{c(\bmu, \btchi)}\\
     &=\inner{F(\btlambda)}{\btlambda}-\E{\talpha_e}{c(\bmu, \btchi)},
 \end{align*}
where the last two equality follows by the definition of $\btlambda$, $\talpha_e$ and $\btchi$. Now, note that
\begin{align*}
    \talpha_e(\Omega(\btchi))=\alpha(\Omega(\btchi) \times S_e)=\alpha(\Omega(\E{}{\bbarx}) \times S_e)=\sum_{\bbarq_e \in S_e} \psi_e(\bbarq_e)=1.
\end{align*}
Thus, we have $\talpha_e(\Omega \backslash\Omega(\btchi))=0$ and as $\pi$ is an arbitrary policy, this proves the proposition.
 \endproof
 \subsection{Pessimistic Equilibrium} \label{app: pess_nash_eq}
 Similar to the Proposition \ref{prop: equivalence_joint_marginal}, we can show the following:
 \begin{lemma} \label{lemma: pessimistic: joint_distribution_formulation}
 \begin{align*}
    &\left\{(\zeta,\alpha): \alpha(\bbR_+^n \times B)=\sum_{\bbarq \in B} \psi(\bbarq), \zeta(\sigma^{-1}(C) \times B)=\alpha(C \times B) \ \forall B \subseteq \calS, C \subseteq \bbR_+^n \right\} \\
    ={}&\bigg\{(\zeta,\alpha): \exists \{\zeta_{\bp^{(1)},\bq}\}_{\alpha \textit{a.e.} \bbR_+^n \times \calS}, \{\alpha_\bq\}_{\bq : \psi(\bq)>0} \textit{ s.t. } \forall f : \sigma^{-1}(\bbR_+^n) \times \calS \rightarrow \bbR, g : \bbR_+^n \times \calS \rightarrow \bbR \textit{ we have }\\
    & \E{\zeta}{f(\bmu,\bp^{(1)},\bbarq)}=\E{\bbarq}{\E{\alpha_\bbarq}{\E{\zeta_{\bp^{(1)},\bbarq}}{f(\bmu,\bp^{(1)},\bbarq)}}}, \E{\alpha}{g(\bp^{(1)},\bbarq)}=\E{\bbarq}{\E{\alpha_\bbarq}{g(\bp^{(1)},\bbarq))}},\\
    &\zeta_{\bp^{(1)},\bbarq}(A)=\zeta_{\bp^{(1)},\bbarq}(A \cap \sigma^{-1}(\bp^{(1)},\bbarq))), \alpha_\bq(C)=\alpha_\bq(C \cap \sigma^{-1}(\bq)), \forall A \subseteq \sigma^{-1}(\bbR_+^n) \times \calS, C \subseteq \bbR_+^n \times \calS\bigg\},
 \end{align*}
 where only Borel measurable sets and functions are considered.
  \end{lemma}
  Note that, with a slight abuse of notation, we denote all the projection functions by $\sigma$. The domain and the range of the function will be clear from the context. The proof of the Lemma is provided at the end of the section. Now, using the result, we can re-write the pessimistic formulation in terms of joint distributions $(\zeta,\alpha)$ as follows:
\begin{subequations}
\begin{align}
    R^\star_\rho\triangleq \sup_{\pi=(\blambda(\cdot),\alpha(\cdot),\bx(\cdot))\in \StablePol_p}&\bigg\{ \inf_{\zeta(\cdot) \in Z_\rho(\pi)}  \E{\bbarq}{\inner{F(\blambda(\bbarq))}{\blambda(\bbarq)}} 
    -\E{\zeta}{\cost(\bmu,\bp^{(1)})}-\E{\bbarq}{\inner{\bs}{\bbarq}}\bigg\}\\
    \textit{subject to,} \quad &\blambda(q)\in \bbR^m_+,\quad \forall q\in \calS\\
    &\bx(\cdot)\quad \textit{satisfies} \:\: \eqref{eq: matching_constraints} . \\
    &\zeta(\sigma^{-1}(C) \times B)=\alpha(C \times B) \quad \forall C \subseteq \bbR_+^n, B \subseteq \calS \textit{ and } C,B \textit{ Borel} \label{eq: consistency_zeta_alpha} \\
    &\alpha(\bbR_+^n \times B)=\sum_{\bbarq \in B} \psi(\bbarq) \quad \forall B \subseteq \calS \textit{ and } B \textit{ Borel} \label{eq: consistency_alpha_nu}\\
    &\psi(\cdot)=\textit{Stationary Distribution}(\blambda(\cdot),\alpha(\cdot),\bx(\cdot),\zeta(\cdot)), \label{eq: stationary_nu}
\end{align}
\label{eq:opt_stoch_pessimistic_joint}
\end{subequations}
where $\sigma^{-1}(C)=\{(\bmu,\bp^{(1)}): \bmu \in \mathcal{M}(\bp^{(1)}), \bp^{(1)} \in C\}$. The inner infimum over all the possible equilibrium under which the system is stable is to make sure that the servers picks the set of arrival rates which minimizes the profit of the system operator. The outer supremum is the system operator picking a pricing policy which maximizes it's profit. The constraint \eqref{eq: consistency_zeta_alpha} is the consistency conditions which makes sure that the marginal distribution of $\zeta(\cdot)$ with respect to $(\bp^{(1)},\bbarq)$ follows the the joint distribution of the pricing policy and stationary distribution $\alpha$. Intuitively, the drivers only have freedom to choose the distribution over $\bmu$ as the prices are set by the system operator and the stationary distribution is the response of the system dynamics. Similarly, \eqref{eq: consistency_alpha_nu} is the consistency condition for $\alpha(\cdot)$. The marginal distribution of the pricing policy should coincide with the stationary distribution.
For the above defined pessimistic model, the probabilistic fluid model can be written as follows:
\begin{subequations}
\begin{align}
  \tilde{R}_\rho^\star=\sup_{\btlambda,\talpha,\btchi}\left\{\inf_{\tzeta(\bmu,\bp^{(1)})}\left\{\inner{F(\btlambda)}{\btlambda}-\E{\tzeta}{c(\bp^{(1)},\bmu)}\right\}\right\} \span \\
    \textit{subject to, } \tl_j&=\sum_{i=1}^n \tchi_{ij} \quad \forall j \in [m] \\
    \E{\tzeta}{\mu_i}&=\sum_{j=1}^m \tchi_{ij} \quad \forall i \in [n] \\
    \tchi_{ij}&=0 \quad \forall (i,j) \notin E \\
    \tzeta(\sigma^{-1}(C))&=\talpha(C) \quad \forall C \subseteq \bbR_+^n \textit{ and } C \textit{ Borel} \label{eq: consistency_fluid} \\
    \tzeta(\sigma^{-1}(\bbR_+^n))&=1.
\end{align}
\label{eq: prob_fluid_model_pessimistic}
\end{subequations}
Similar to the model given by \eqref{eq:opt_stoch_pessimistic}, we impose a consistency condition \eqref{eq: consistency_fluid} between the `averaged' response of drivers $\tzeta(\cdot)$ and the `averaged' pricing policy $\talpha(\cdot)$. The last constraint restricts the drivers' responses to the set of rates under which the system is in equilibrium. Now, we will show that this fluid model provides an upper bound on the achievable profit under any given policy.
\begin{proposition} \label{prop: fluid_model_pessimistic}
Let $\pi=(\blambda(\cdot),\alpha(\cdot),\bx(\cdot))$ be a feasible solution of \eqref{eq:opt_stoch_pessimistic} then 
\begin{equation*}
R_\rho(\pi)\leq P_\rho(\pi)
%\E{\bbarq}{\inner{F(\blambda(\bbarq))}{\blambda(\bbarq)}}-\E{\alpha(\bbarq)}{\cost(\bmu)}
    \leq \tilde{R}_\rho^\star.
\end{equation*}
That is, the fluid profit is an upper bound for the stochastic profit and net profit. 
\end{proposition}
Thus, this proposition provides us with an upper bound on the net average profit achievable under any policy. We now present the following lemma which is a crucial step in the proof of the proposition.
\begin{lemma} \label{lemma: necessary_constraints_pessimistic}
For a given stationary Markovian pricing and matching policy $(\blambda(\cdot),\alpha(\cdot),\bx(\cdot))$, let $\btlambda=\E{}{\blambda(\bbarq)}$, $\talpha(C)=\alpha(C \times \calS)$ for all Borel $C \subseteq \bbR_+^n$ and $\btchi=\E{}{\bx(\bbarq)}$. If $(\blambda(\cdot),\alpha(\cdot),\bx(\cdot))\in \StablePol_\rho$, $\E{}{\inner{\bone_{n+m}}{\bbarq}}<\infty$ and $\bx(\cdot)$ satisfies \eqref{eq: matching_constraints},
then $(\btlambda,\talpha,\btchi)$ is feasible in the probabilistic fluid problem \eqref{eq: prob_fluid_model_pessimistic}.
%the constraints of the optimization problem \eqref{eq: prob_fluid_model} are necessary for stability.
\end{lemma}
We first use this lemma to prove the proposition and later prove the lemma.
\proof{Proof of Proposition \ref{prop: fluid_model_pessimistic}}
Given $(\blambda(\cdot),\alpha(\cdot),\bx(\cdot))\in \StablePol_\rho$, the profit under this policy can be written as follows:
\begin{align*}
    \inf_{\zeta \in Z_\rho(\pi)} \left\{\E{\bbarq}{\inner{F(\blambda(\bbarq))}{\blambda(\bbarq)}}-\E{\zeta}{c(\bmu,\bp^{(1)})}\right\} &\leq \inf_{\zeta \in Z_\rho(\pi)} \left\{\inner{F(\btlambda)}{\btlambda}-\E{\zeta}{c(\bmu,\bp^{(1)})}\right\} \\
    &=\inf_{\tzeta: \tzeta(\sigma^{-1}(\bbR_+^n))=1} \left\{\inner{F(\btlambda)}{\btlambda}-\E{\tzeta}{c(\bmu,\bp^{(1)})}\right\},
\end{align*}
where $\tzeta(A)=\zeta(A \times \calS)$ for all Borel $A \subseteq \sigma^{-1}(\bbR_+^n)$. Note that the consistency condition for $\talpha$ and $\tzeta$ can be checked as follows:
\begin{align*}
    \tzeta(\sigma^{-1}(C))=\zeta(\sigma^{-1}(C) \times \calS)= \alpha(C \times \calS) = \talpha(C) \quad \forall C \subseteq \bbR_+^n \textit{ and } C \textit{ Borel}. \hfill \Halmos
\end{align*}
\endproof
\proof{Proof of Lemma \ref{lemma: necessary_constraints_pessimistic}}
By the hypothesis of the Lemma, we have $\E{\bbarq}{\inner{\bone_{n+m}}{\bbarq}}<\infty$. Thus, in steady state, we have $
    \E{}{\bbarq}=\E{}{\bbarq^+} \Rightarrow \E{}{\bbara}=\E{}{\bbarx}$, where  we denote the queue length one time slot after $\bbarq$ by $\bbarq^+=\bbarq+\bbara-\bbarx$.
Now, we will simplify the RHS and LHS separately. We have
\begin{align}
    \E{}{\bbara}=(\E{\bbarq}{\blambda(\bbarq)},\E{\zeta}{\bmu})=(\btlambda,\E{\zeta}{\bmu})=(\btlambda,\E{\tzeta}{\bmu}). \label{eq: a_fluid_pess}
\end{align}
Also, note that
\begin{subequations}
\begin{align}
    \E{\bbarq}{\barx_i^{(1)}}&=\sum_{j=1}^m \E{\bbarq}{\bary_{ij}}=\sum_{j=1}^m \tchi_{ij} \\
     \E{\bbarq}{\barx_j^{(2)}}&=\sum_{i=1}^n \E{\bbarq}{\bary_{ij}}=\sum_{i=1}^n \tchi_{ij} 
\end{align}
\label{eq: x_fluid_pess}
\end{subequations}
Now, equating \eqref{eq: a_fluid_pess} and \eqref{eq: x_fluid_pess} gives the lemma.
$\Halmos$
\endproof
Once we obtain a fluid model which provides an upper bound on the achievable profit under any pricing and matching policy, we can define a two price policy and max-weight matching policy similar to the optimistic case. Retracing the same steps, we should be able to show that the resultant policy is asymptotically optimal. As this will mostly be repetition of the main content of the paper, we omit the details here and conclude our discussion on the pessimistic, probabilistic fluid model. Now, we conclude our discussion by presenting the proof of Lemma \ref{lemma: pessimistic: joint_distribution_formulation}.
\proof{Proof of Lemma \ref{lemma: pessimistic: joint_distribution_formulation}}
Note that the push forward measure $\alpha\circ \sigma^{-1}$ is given by
\begin{align*}
    \alpha\circ \sigma^{-1}(\bbarq)=\alpha(\bbR_+^n \times\{\bbarq\})=\psi(\bbarq) \quad \forall \bbarq \in \calS.
\end{align*}
Thus, by the disintegration theorem, there exists $\psi$ almost everywhere $\{\alpha_\bbarq\}_{\calS}$ such that
\begin{align}
    \alpha_\bbarq(C)=\alpha_\bbarq(C \cap \sigma^{-1}(\bbarq)) \quad \forall C \subseteq \bbR_+^n \textit{ and } \E{\alpha}{g(\bp^{(1)},\bbarq)}=\E{\bbarq}{\E{\alpha_\bbarq}{g(\bp^{(1)},\bbarq)}}. \label{eq: joint_pessimistic}
\end{align}
Next, note that the push forward measure $\zeta \circ \sigma^{-1}$ is given by
\begin{align*}
    \zeta \circ \sigma^{-1}(C \times B)=\zeta(\sigma^{-1}(C) \times B)=\alpha(C \times B) \quad \forall C \subseteq \bbR_+^n, B \subseteq \calS.
\end{align*}
Thus, again by the disintegration theorem, there exists $\alpha$ almost everywhere $\{\zeta_{\bp^{(1)},\bbarq}\}_{\bbR_+^n \times \calS}$ such that 
\begin{align*}
    \zeta_{\bp^{(1)},\bbarq}(A)&=\zeta_{\bp^{(1)},\bbarq}(A \cap \sigma^{-1}(\bp^{(1)},\bbarq)) \quad \forall A \subseteq \sigma^{-1}(\bbR^n_+) \times \calS \\
    \E{\zeta}{f(\bmu,\bp^{(1)},\bbarq)}&=\E{\alpha}{\E{\zeta_{\bp^{(1)},\bbarq}}{f(\bmu,\bp^{(1)},\bbarq)}}=\E{\bbarq}{\E{\alpha_\bbarq}{\E{\zeta_{\bp^{(1)},\bbarq}}{f(\bmu,\bp^{(1)},\bbarq)}}},
\end{align*}
where the last equality follows by \eqref{eq: joint_pessimistic}. This proves that the first set is a subset of the second set. Now, we will prove the opposite. For a set $B \subseteq \calS$ and $C \subseteq \bbR_+^n$, take $g(\bp^{(1)},\bbarq)=\mathbbm{1}\{C \times B\}$. We have
\begin{align*}
    \alpha(C \times B)&=\E{\alpha}{g(\bp^{(1)},\bbarq)}=\E{\bbarq}{\E{\alpha_\bbarq}{g(\bp^{(1)},\bbarq)}}=\E{\bbarq}{\alpha_\bbarq(C \times B)}\\
    &=\E{\bbarq}{\alpha_\bbarq(C \times \{\bbarq\})\mathbbm{1}\{\bbarq \in B\}}=\sum_{\bbarq \in B} \alpha_\bbarq(C \times \{\bbarq\})\psi(\bbarq) \\
    &=\sum_{\bbarq \in B}\alpha_\bbarq(C \times \calS)\psi(\bbarq).
\end{align*}
Take $C=\bbR_+^n$ to get $\alpha(\bbR_+^n \times B)=\sum_{\bbarq \in B} \psi(\bbarq)$. Now, take $f(\bmu,\bp^{(1)},\bbarq)=\mathbbm{1}\{\sigma^{-1}(C) \times B\}$ to get
\begin{align*}
    \zeta(\sigma^{-1}(C) \times B)&=\E{\zeta}{f(\bmu,\bp^{(1)},\bbarq)}=\E{\bbarq}{\E{\alpha_\bbarq}{\E{\zeta_{\bp^{(1)},\bbarq}}{f(\bmu,\bp^{(1)},\bbarq)}}} \\
    &=\E{\alpha}{\E{\zeta_{\bp^{(1)},\bbarq}}{f(\bmu,\bp^{(1)},\bbarq)}}=\E{\alpha}{\zeta_{\bp^{(1)},\bbarq}(\sigma^{-1}(C) \times B)} \\
    &=\E{\alpha}{\zeta_{\bp^{(1)},\bbarq}(\mathcal{M}(\bp^{(1)}) \times \{\bp^{(1)}\} \times \{\bbarq\})) \mathbbm{1}\{\bp^{(1)} \in C, \bbarq \in B\}} \\
    &=\E{\alpha}{\zeta_{\bp^{(1)},\bbarq}(\sigma^{-1}(\bbR_+^n) \times \calS) \mathbbm{1}\{\bp^{(1)} \in C, \bbarq \in B\}} \\
    &=\E{\alpha}{ \mathbbm{1}\{\bp^{(1)} \in C, \bbarq \in B\}} \\
    &=\alpha(C \times B).
\end{align*}
This proves that the second set is a subset of the first set. Thus, the proof is complete. \hfill $\Halmos$
\endproof
\section{Technical Details}
\subsection{Borel Measurability of the Cost Function} \label{app: borel}
We recap the cost function is given in \eqref{eq: cost_function} below for the convenience of the reader.
\begin{align*}
    c(\bmu)=\min_{\bp^{(1)} \in \mathcal{M}(\bmu)} \inner{\bp^{(1)}}{\bmu}.
\end{align*}
Note that, using the Lemma \ref{lemma: variational_inequality} and Lemma \ref{lemma: integer_program}, we can write the cost function as the value function of a non-linear mixed-integer optimization program. In particular, the only integer variables are $\bb \in \{0,1\}^{n \times n}$. Note that, we can equivalently write the cost function as follows:
\begin{align*}
    c(\bmu)=\min_{\bar{\bb} \in \{0,1\}^{n \times n}, \sum_{l=1}^n \bar{b}_{il}=1 \ \forall i \in [n]}\left\{ \min_{\bp^{(1)} \in \mathcal{M}(\bmu)} \inner{\bp^{(1)}}{\bmu} \ \textit{s.t. } \bb=\bar{\bb}\right\}.
\end{align*}
Now, denote the inner minimum as the function $c_{\bar{\bb}}(\bmu)$ and note that it is of the form of standard non-linear optimization program. The objective function is continuous and it can be easily checked that all the constraints are also continuous in $\bmu$, $\bp^{(1)}$ and auxiliary variables. Thus, by \cite[Theorem 14.36]{variational_analysis} we conclude that the set mapping defined by the constraint set is closed and measurable. Next, by \cite[Example 14.32]{variational_analysis} the sum of the objective function and the indicator of the constraint set is normal integrand. Finally, by \cite[Theorem 14.37]{variational_analysis} the function $c_{\bar{\bb}}(\bmu)$ is Borel measurable. This implies that the cost function is Borel measurable as the outer minimum is over a finite set. \hfill $\Halmos$
\subsection{Existence of Optimal Solution of Probabilistic Fluid Model} \label{app: existence}
By \citep[Proposition 6.40]{shapiro2014lectures}, the probabilistic fluid model can be equivalently written as an optimization program over a finite set of variables given by 
\begin{align*}
    \max_{\blambda,\{\bmu^l\}_{l=1}^{n+1},\bchi,\bbeta} \inner{F(\blambda)}{\blambda}-\sum_{l=1}^{n+1}c(\bmu^{l})\beta_l \span \\
    \textit{subject to, } \lambda_j&=\sum_{j=1}^n \chi_{ij} \ \forall j \in [m], \quad
    \sum_{l=1}^{n+1}\beta_l \mu^l_i=\sum_{j=1}^m \chi_{ij} \quad \forall i \in [n] \\
    \chi_{ij}&=0 \ \forall (i,j) \notin E, \quad \chi_{ij} \geq 0 \ \forall (i,j) \in E, \quad
    \inner{\bone_{n+1}}{\bbeta}=1, \ \bbeta \geq \bzero_{n+1}.
\end{align*}
Now, by substituting the definition of the cost function, we get
\begin{align*}
    \max_{\blambda,\{\bmu^l,\bp^{(1),l}\}_{l=1}^{n+1},\bchi,\bbeta} \inner{F(\blambda)}{\blambda}-\sum_{l=1}^{n+1}\inner{\bmu^l}{\bp^{(1),l}}\beta_l \span \\
    \textit{subject to, } \lambda_j&=\sum_{j=1}^n \chi_{ij} \ \forall j \in [m], \quad
    \sum_{l=1}^{n+1}\beta_l \mu^l_i=\sum_{j=1}^m \chi_{ij} \quad \forall i \in [n] \\
    \chi_{ij}&=0 \ \forall (i,j) \notin E, \quad \chi_{ij} \geq 0 \ \forall (i,j) \in E, \quad
    \inner{\bone_{n+1}}{\bbeta}=1, \ \bbeta \geq \bzero_{n+1} \\
    \bp^{(1),l} &\in \mathcal{M}(\bmu^l) \quad \forall l \in [n+1].
\end{align*}
Note that the objective function is continuous and the feasible region is closed (given by inequality, equality constraints and binary variables). We assume that following:
\begin{assumption}
There exists $M>0$ such that $(\blambda,\bmu,\bp^{(1)}) \leq M \bone_{m+2n}$.
\end{assumption}
Under this mild assumption, the constraint set is also bounded. Thus, the constraint set is compact. As the objective function is continuous and the constraint set is compact, there exists an optimal solution.
\end{APPENDICES}
\end{document}